\documentclass[11pt]{article}
\usepackage{algorithm}
\usepackage{algpseudocode}
\usepackage[a4paper,margin=1in]{geometry}
\usepackage[T1]{fontenc}
\usepackage{lmodern}
\usepackage{microtype}
\usepackage{amsmath,amssymb,amsthm,mathtools,bm}
\usepackage{mathrsfs}
\usepackage{aliascnt}
\usepackage{enumitem}
\usepackage{xcolor}
\usepackage{booktabs}
\usepackage[round,authoryear]{natbib}
\usepackage{tabularx,array}
\usepackage{graphicx}
\usepackage{hyperref}
\definecolor{mydarkblue}{rgb}{0,0.08,0.45}
\hypersetup{
  colorlinks=true,
  linkcolor=mydarkblue ,    
  filecolor=mydarkblue ,  
  urlcolor=red ,     
  citecolor=mydarkblue    
}
\usepackage[nameinlink,capitalise,noabbrev]{cleveref}

\allowdisplaybreaks
\setlist[itemize]{topsep=3pt,itemsep=2pt,parsep=1pt,leftmargin=2em}
\setlist[enumerate]{topsep=3pt,itemsep=2pt,parsep=1pt,leftmargin=2em}

\newtheorem{theorem}{Theorem}[section]
\newaliascnt{lemma}{theorem}
\newtheorem{lemma}[lemma]{Lemma}
\aliascntresetthe{lemma}
\newaliascnt{proposition}{theorem}
\newtheorem{proposition}[proposition]{Proposition}
\aliascntresetthe{proposition}
\newaliascnt{corollary}{theorem}
\newtheorem{corollary}[corollary]{Corollary}
\aliascntresetthe{corollary}
\newaliascnt{definition}{theorem}
\newtheorem{definition}[definition]{Definition}
\aliascntresetthe{definition}
\newaliascnt{remark}{theorem}
\newtheorem{remark}[remark]{Remark}
\aliascntresetthe{remark}
\newaliascnt{claim}{theorem}

\aliascntresetthe{claim}

\newaliascnt{remark*}{theorem}
\newtheorem{remark*}[remark]{Remark}
\aliascntresetthe{remark}

\crefname{theorem}{Theorem}{Theorems}
\Crefname{theorem}{Theorem}{Theorems}
\crefname{lemma}{Lemma}{Lemmas}
\Crefname{lemma}{Lemma}{Lemmas}
\crefname{proposition}{Proposition}{Propositions}
\Crefname{proposition}{Proposition}{Propositions}
\crefname{corollary}{Corollary}{Corollaries}
\Crefname{corollary}{Corollary}{Corollaries}
\crefname{definition}{Definition}{Definitions}
\Crefname{definition}{Definition}{Definitions}
\crefname{remark}{Remark}{Remarks}
\Crefname{remark}{Remark}{Remarks}
\crefname{claim}{Claim}{Claims}
\Crefname{claim}{Claim}{Claims}

\newcommand{\R}{\mathbb{R}}
\newcommand{\E}{\mathbb{E}}
\newcommand{\Pp}{\mathbb{P}}
\newcommand{\one}{\mathbf{1}}
\newcommand{\norm}[1]{\left\lVert #1\right\rVert}
\newcommand{\abs}[1]{\left\lvert #1\right\rvert}
\newcommand{\ip}[2]{\left\langle #1,#2\right\rangle}
\newcommand{\prog}{\operatorname{prog}}
\newcommand{\Lip}{\operatorname{Lip}}

\newcommand{\spanop}{\operatorname{span}}

\newcommand{\cO}{\mathcal{O}}
\newcommand{\cA}{\mathsf{A}}
\newcommand{\cX}{\mathcal{X}}

\newcommand{\thetaalpha}{\theta_{\alpha}}
\newcommand{\betaalpha}{\beta_{\alpha}}

\title{Tight Heavy-Tailed Complexity Bounds for Optimizing Polyak--\L{}ojasiewicz Objectives}
\author{Weiming Ou\footnote{Shanghai University of Finance and Economics, email: ouweiming@163.sufe.edu.cn}\qquad Xiao Wang\footnote{MoE Key Laboratory of Interdisciplinary Research of Computation and Economics, Shanghai University of Finance and Economics, email: wangxiao@sufe.edu.cn}}
\date{}

\begin{document}
\maketitle

\begin{abstract}
We study optimizing $L$-smooth and $\mu$-Polyak–Łojasiewicz (PL) objectives with unbiased stochastic gradients under $\alpha$-heavy-tailed noise ($1<\alpha\le2$). In the unrestricted \textbf{high-dimensional} setting, we establish the noise-adaptive lower bound for attaining $\epsilon$-suboptimal function value. Under the generalized mirror-PL condition, we propose a centered-clipped mirror-descent method that achieves a high-probability upper bound matching the lower bound up to logarithmic factors. We further provide the improved tight stochastic complexity in \textbf{fixed} dimensions.
\end{abstract}

\newpage
\tableofcontents

\newpage
\section{Introduction}
\addcontentsline{toc}{section}{Introduction}

The Polyak--\L{}ojasiewicz inequality provides an intermediate model between strongly convex and fully nonconvex optimization.
For a differentiable objective $f$, it requires $\|\nabla f(x)\|^2\ge 2\mu\bigl(f(x)-f^\star\bigr)$ for some $\mu>0$,
without requiring $f$ itself to be convex. Nevertheless, when $f$ is also
$L$-smooth, this inequality is sufficient for gradient descent to converge
globally at a linear rate governed by the condition number
$\kappa=L/\mu$ \citep{Polyak1963,KarimiEtAl2016}. The PL framework therefore
captures a broad class of objectives that retain a global optimization
guarantee even though their geometry may be nonconvex \citep{hardt2016identity,li2018algorithmic,liu2022loss,allen2019convergence}. Understanding the oracle complexity of PL optimization is consequently important both
for identifying the limits of first-order algorithms and for explaining
which parts of strongly convex optimization theory continue to hold under
the weaker gradient-dominance assumption.

In stochastic optimization, however, the geometry of the objective is only
one source of difficulty. Classical analyses often assume bounded variance
or sub-Gaussian gradient noise, while empirical studies of stochastic
gradients in modern learning problems indicate that their distributions can
be strongly non-Gaussian and heavy-tailed \citep{SimsekliEtAl2019}. A natural model assumes conditional unbiasedness and a finite centered
$\alpha$-moment for each stochastic gradient $G_t$:
\[
\mathbb E[G_t\mid\mathscr F_{t-1}]=\nabla f(x_t),
\qquad
\mathbb E\!\left[
\|G_t-\nabla f(x_t)\|^\alpha
\mid\mathscr F_{t-1}
\right]\le\sigma^\alpha,
\qquad 1<\alpha\le2.
\]
For $\alpha<2$, this assumption permits infinite variance; empirical averages need
not exhibit Gaussian concentration, and rare but arbitrarily large oracle
errors cannot be neglected. Gradient clipping, normalization, and robust
mean estimation have consequently become central tools for obtaining
high-probability guarantees under heavy-tailed noise
\citep{CutkoskyMehta2021,SadievEtAl2023,PuchkinEtAl2024}. Yet clipping a
stochastic gradient around the origin can introduce a bias depending on
$\|\nabla f(x)\|$, which may be unbounded on an unbounded domain. This makes
it nontrivial to obtain a high-probability theory that is simultaneously
robust to heavy tails, compatible with mirror geometry, and free of bounded-gradient and bounded-domain assumptions.

A second challenge concerns optimality. In the deterministic high-dimensional
setting, smooth PL optimization has the lower-bound scale
$\Omega(\kappa\log(\Delta_0/\epsilon))$, where $\Delta_0$ bounds the
initial objective gap. This shows that the general PL class
does not admit the $\sqrt{\kappa}$ acceleration available under strong
convexity \citep{YueEtAl2023}. Stochastic lower bounds based on probabilistic
zero-chains and random rotations provide a powerful way to hide sequential
information from arbitrary first-order algorithms
\citep{ArjevaniEtAl2023}. For \emph{heavy-tailed} noise, however, a complete lower
bound must capture two different obstructions at once: the deterministic
optimization transient and the statistical cost of estimating a mean under
only a finite $\alpha$-moment. Moreover, it should be \emph{noise-adaptive}:
when $\sigma=0$, the stochastic term vanishes, but the lower bound should
reduce to the correct exact-gradient complexity rather than becoming
trivial.

The role of dimension creates a third difficulty. Randomly rotated
zero-chains require a dimension that grows with the number of hidden
directions, whereas low-dimensional optimization may permit geometric search
strategies unavailable in high dimension. Gradient-flow trapping, for
example, exploits separating curves or hypersurfaces to locate approximate
stationary points in fixed dimension \citep{BubeckMikulincer2020}. Extending
this idea to the present stochastic setting is delicate because the
algorithm has access only to noisy gradients, rather than exact function
values on an entire separating surface. It is therefore not clear a priori
whether the condition-number factor in the high-dimensional stochastic
complexity remains unavoidable when $d$ is fixed, or how the answer depends
jointly on the dimension and the tail exponent.

Motivated by these issues, this paper asks the following question: 
\begin{center}
    \textit{What is
the optimal first-order complexity of smooth PL optimization as a function of $\kappa$, $\sigma$, $\epsilon$, the moment order $\alpha$, and
the dimension $d$, when the guarantee must hold with high probability and
remain informative in the noiseless limit?}
\end{center}

\subsection{Our Contributions}
Our results separate the roles of geometry, noise, and dimension.
For brevity, write $\theta_\alpha:=\alpha/[2(\alpha-1)]$ and
$p_\alpha:=\alpha/(\alpha-1)$; Section~\ref{sec:prelim} gives the full
notation and query model.
\begin{itemize}[leftmargin=*,itemsep=5pt]
\item \textbf{A noise-adaptive high-dimensional lower bound.}
For arbitrary adaptive randomized gradient-only algorithms, sufficiently
large $\kappa$, and sufficiently small $\epsilon/\Delta_0$, we establish
\(
\Omega\bigl(\kappa\log(\Delta_0/\epsilon)
+\kappa(\sigma^2/(\mu\epsilon))^{\theta_\alpha}\bigr)
\)
in a sufficiently large finite dimension.
The two terms capture the deterministic transient and the heavy-tailed
estimation cost. In particular, the bound recovers
$\Omega(\kappa\log(\Delta_0/\epsilon))$ when $\sigma=0$
(Theorem~\ref{thm:hd-accuracy}).

\item \textbf{A matching high-probability upper bound.}
For a nonempty closed convex domain, possibly unbounded, we combine
centered clipping with mirror descent. Under mirror-PL, the method returns
an $\epsilon$-suboptimal point with probability at least $1-\delta$ using
\(
\widetilde O\bigl(\kappa\log(\Delta_0/\epsilon)
+\kappa(\sigma^2/(\mu\epsilon))^{\theta_\alpha}\bigr)
\)
stochastic-gradient calls (Theorem~\ref{thm:upper-accuracy}).
No bounded-gradient or sub-Gaussian assumption is needed. On the
unconstrained Euclidean specialization, this upper bound matches the
lower bound up to logarithmic factors.

\item \textbf{Improved complexity in fixed dimensions.}
For every fixed $d\ge1$ and $\kappa\ge6$, we prove the lower bound
$\Omega(\log(\Delta_0/\epsilon)+(\sigma^2/(\mu\epsilon))^{\theta_\alpha})$
in the stated accuracy range (Theorem~\ref{thm:fd-lower-main}).
Robust stochastic bisection matches this bound up to logarithmic factors
for $d=1$. Gradient-flow trapping gives upper bounds
$\widetilde O(\kappa^{2/3}\log(\Delta_0/\epsilon)
+(\sigma^2/(\mu\epsilon))^{\theta_\alpha})$ for $d=2$ and
$\widetilde O(\kappa\log(\Delta_0/\epsilon)
+(\sigma^2/(\mu\epsilon))^{\theta_\alpha})$ for $d=3$.
Thus the stochastic term is matched in both dimensions. For fixed $d>3$,
the same stochastic characterization holds when $p_\alpha\ge d-1$;
the deterministic term becomes
$\widetilde O(\kappa^{(d-1)/2}\log(\Delta_0/\epsilon))$.
Theorem~\ref{cor:fd-123} also states the remaining gap when $p_\alpha<d-1$.

\item \textbf{Ordinary PL versus constrained mirror-PL.}
We construct a smooth objective satisfying ordinary PL on an unbounded
convex domain for which exact mirror descent remains at a nonoptimal
boundary point (Proposition~\ref{prop:boundary-obstruction}).
This shows why constrained mirror descent requires a condition on feasible
descent, rather than the full gradient norm. In the unconstrained Euclidean
case, mirror-PL reduces to ordinary PL.
\end{itemize}

\begin{table}[htbp]
\centering
\renewcommand{\arraystretch}{1.55}
\small
\begin{tabular}{cccc}
\toprule
\textbf{Regime} & \textbf{Lower bound} & \textbf{High-prob. upper bound} & \textbf{Comparison}\\
\midrule
Growing $d$
&
$\Omega\bigl(\kappa A_\epsilon+\kappa Q_\epsilon\bigr)$
&
$\widetilde O\bigl(\kappa A_\epsilon+\kappa Q_\epsilon\bigr)$
&
Matched 
\\
\hline
$d=1$
&
$\Omega\bigl(A_\epsilon+Q_\epsilon\bigr)$
&
$\widetilde O\bigl(A_\epsilon+Q_\epsilon\bigr)$
&
Matched
\\
\hline
$d=2$
&
$\Omega\bigl(A_\epsilon+Q_\epsilon\bigr)$
&
$\widetilde O\bigl(\kappa^{2/3}A_\epsilon+Q_\epsilon\bigr)$
&
Stochastic term matched
\\
\hline
$d=3$
&
$\Omega\bigl(A_\epsilon+Q_\epsilon\bigr)$
&
$\widetilde O\bigl(\kappa A_\epsilon+Q_\epsilon\bigr)$
&
Stochastic term matched
\\
\hline
$\begin{array}{c}
          \text{Fixed}\ d>3,\\
          p_\alpha\ge d-1
    \end{array}$ 
&
$\Omega\bigl(A_\epsilon+Q_\epsilon\bigr)$
&
$\widetilde O\bigl(
 \kappa^{(d-1)/2}A_\epsilon+Q_\epsilon\bigr)$
&
Stochastic term matched
\\
\bottomrule
\end{tabular}
\caption{Complexity bounds in the stated condition-number and accuracy
ranges; matching is up to logarithmic factors. Upper bounds hold with
probability at least $1-\delta$, and lower bounds hold in expectation and
at constant confidence. In the last row, one may take the minimum with
$\widetilde O(\kappa[A_\epsilon+Q_\epsilon])$.
The quantities $A_\epsilon$, $Q_\epsilon$, and $p_\alpha$ are defined in
\eqref{eq:heavy-exponents}--\eqref{eq:Q-alpha}.}
\label{tab:comparison}
\end{table}

\paragraph{Comparison with \citet{MasihaKiyavashThiran2026}.}
That work studies lower bounds under a sublevel geometric PL condition
with exponent $\nu\in[1,2)$, namely
$f-f^\star\le\tau\|\nabla f\|^\nu$, while retaining global smoothness
and using bounded-variance oracles. The geometric exponent $\nu$ is distinct from the noise-moment order
$\alpha$ used throughout our paper.
Our results concern the classical global PL geometry and finite centered
$\alpha$-moments, together with high-probability algorithms and
fixed-dimensional refinements. Appendix~\ref{appendix-detailed comparison}
compares the stated exponent ranges and explains the distinction between
formal endpoint substitution and a theorem valid at that endpoint.

\subsection{Related Work}
\paragraph{Polyak--\L{}ojasiewicz (PL) condition.}
The PL condition originates in the work of \citet{polyak1963gradient} and
\citet{lojasiewicz1963topological}. It relaxes strong convexity and has
been used in several optimization settings, including minimax problems
\citep{guo2020fast,guo2023fast}. Other related conditions include error
bounds \citep{luo1993error}, restricted secant inequalities
\citep{zhang2013gradient}, quadratic growth \citep{anitescu2000degenerate},
essential strong convexity \citep{liu2014asynchronous}, and weak strong
convexity \citep{necoara2019linear}; see \citet{guille2021study} for a
comparison. Gradient descent converges linearly under PL
\citep{polyak1963gradient}, and \citet{yue2023lower} established a tight
deterministic lower bound. \citet{karimi2016linear} extended the framework
to constrained optimization through proximal PL. In the stochastic
setting, \citet{khaled2020better} and \citet{li2021second} obtained
expected oracle-complexity bounds with stochastic term
$O(L\sigma^2/(\mu^2\epsilon))$, where $\sigma^2$ bounds the noise variance.

\paragraph{Heavy-tailed stochastic optimization.}
Empirical evidence of non-Gaussian stochastic gradients motivates analysis
beyond finite-variance and sub-Gaussian models \citep{SimsekliEtAl2019}.
Gradient clipping yields high-probability guarantees for smooth convex and
strongly convex optimization, including accelerated methods, without
light-tail assumptions \citep{GorbunovDanilovaGasnikov2020}.
Plain SGD has also been analyzed under specialized state-dependent
infinite-variance models for strongly convex objectives \citep{WangEtAl2021}.
For stochastic convex optimization with a finite moment strictly larger
than one, \citet{VuralEtAl2022} established minimax-optimal mirror-descent
bounds without explicit clipping or normalization.
For smooth nonconvex objectives, \citet{CutkoskyMehta2021} combined clipping,
momentum, and normalized steps, while \citet{LiuZhangZhou2023} showed how
additional sample-function structure can improve on unstructured-oracle
rates. \citet{SadievEtAl2023} developed high-probability results under
bounded centered $\alpha$-moments for nonconvex, PL, convex, and strongly
convex optimization, as well as variational inequalities.
Further developments include general Bregman geometries
\citep{FatkhullinHe2024}, clipping stochastic gradient differences
\citep{GorbunovEtAl2024}, smoothed median-of-means estimation
\citep{PuchkinEtAl2024}, and clipped adaptive methods
\citep{ChezhegovEtAl2025}. \citet{FradinEtAl2026} established lower and
upper bounds under generalized stochastic smoothness and developed
double-clipped methods with high-probability guarantees.

\paragraph{Lower bounds.}
Classical deterministic first-order lower bounds address convex
optimization \citep{nesterov2013introductory} and nonconvex stationarity
\citep{carmon2020lower}. Stochastic lower bounds were developed for convex
objectives by \citet{agarwal2009information,raginsky2011information},
extended to heavy-tailed noise by \citet{vural2022mirror}, and established
for nonconvex stationarity by \citet{arjevani2023lower}.
\citet{Saad2025new} used divergence decomposition for structured nonconvex
classes, including objectives satisfying a restricted secant inequality.
Other settings include finite-sum optimization
\citep{fang2018spider,zhou2019lower} and zeroth-order oracles
\citep{duchi2015optimal}. For PL-type objectives,
\citet{yue2023lower} treated the deterministic case,
\citet{bai2024complexity} studied finite-sum complexity, and
\citet{MasihaKiyavashThiran2026} considered the sublevel geometric condition
discussed above. Our lower bounds instead retain classical global PL
geometry while accounting for both finite-moment noise and the noiseless
transient.

\paragraph{High-probability convergence.}
\citet{nemirovski2009robust} analyzed stochastic mirror descent under
sub-Gaussian noise. \citet{gorbunov2020stochastic,gorbunov2021near}
weakened the noise assumptions in smooth and nonsmooth settings.
Beyond convexity, \citet{li2020high,cutkosky2021high} studied
high-probability guarantees for momentum-based methods, with the latter
using clipping for heavy-tailed noise. \citet{liu2023high} developed further
stochastic mirror-descent analyses. Adaptive methods have been studied
with AdaGrad stepsizes \citep{kavis2022high,attia2023sgd} and with Adam
\citep{hong2024convergence}; \citet{liu2024high} treated clipping for
nonsmooth nonconvex optimization under heavy-tailed noise.
Our focus is the joint dependence on PL geometry, the noise moment, and
dimension, with bounds that also retain the noiseless optimization cost.
\section{Preliminaries}\label{sec:prelim}
\paragraph{Notation.}
All norms and inner products are Euclidean, and $\Lip(h)$ denotes the
Lipschitz constant of a map $h$. Throughout, $L,\mu,\Delta_0>0$,
$\sigma\ge0$, and $1<\alpha\le2$. Constants $c_\alpha,C_\alpha$ depend only
on $\alpha$; subscripts $d$ or $(d,\alpha)$ indicate analogous dependence.
Unnamed constants may change between bounds. The notation
$\widetilde O$ and $\widetilde\Theta$ suppresses logarithmic factors; their
constants may depend on $\alpha$ and, in fixed-dimensional results, on $d$.
Unless a base is specified, $\log$ is the natural logarithm. We use
\begin{equation}\label{eq:heavy-exponents}
 \theta_\alpha:=\frac{\alpha}{2(\alpha-1)},\qquad
 \beta_\alpha:=\frac{2(\alpha-1)}{\alpha}=\frac1{\theta_\alpha},\qquad
 p_\alpha:=\frac{\alpha}{\alpha-1}=2\theta_\alpha,
\end{equation}
and, for noise level $\sigma$, PL parameter $\mu$, initial-gap bound
$\Delta_0$, and target accuracy $\epsilon$, define
\begin{equation}\label{eq:Q-alpha}
 Q_\epsilon:=\left(\frac{\sigma^2}{\mu\epsilon}\right)^{\theta_\alpha},
 \qquad A_\epsilon:=\log\frac{\Delta_0}{\epsilon}.
\end{equation}

\paragraph{Problem setup.}
We consider $\min_{x\in\mathcal X}f(x)$, where
$\mathcal X\subseteq\mathbb R^d$ is nonempty and convex, but need not be
bounded, and $f$ is continuously differentiable and possibly nonconvex.
Write $f^\star:=\inf_{x\in\mathcal X}f(x)>-\infty$ and
$\mathcal X^\star:=\operatorname*{arg\,min}_{x\in\mathcal X}f(x)$.
Our smoothness and geometric assumptions are as follows.

\begin{definition}[$L$-smoothness]\label[definition]{A1}
The function $f$ is $L$-smooth if, for all $x,y\in\mathcal X$,
\begin{equation}
 \|\nabla f(x)-\nabla f(y)\|\le L\|x-y\|.
\end{equation}
\end{definition}

\begin{definition}[$\mu$-PL condition]\label[definition]{A2}
The function $f$ satisfies the $\mu$-Polyak--\L{}ojasiewicz (PL) condition if
\begin{equation}\label{eq:PL-convention}
 \|\nabla f(x)\|^2\ge2\mu\bigl(f(x)-f^\star\bigr)
 \qquad(x\in\mathcal X).
\end{equation}
\end{definition}

\begin{remark}
When $\mathcal X=\mathbb R^d$, $L$-smoothness also gives
$\|\nabla f(x)\|^2\le2L(f(x)-f^\star)$ by the descent inequality.
For an $L$-smooth, $\mu$-PL objective, write $\kappa:=L/\mu$.
\end{remark}

\paragraph{Oracle model and adaptive queries.}
At physical oracle call $t\ge1$, an adaptive randomized algorithm chooses
$x_t$ and receives $G_t=G(x_t,\xi_t)$. The pre-call history
$\mathscr F_{t-1}$ includes previous queries and responses, the algorithm's
randomness used so far, and any randomness used to choose $x_t$.
Thus $x_t$ is $\mathscr F_{t-1}$-measurable before the current oracle noise
is drawn. Write $\mathbb E_t[\cdot]:=\mathbb E[\cdot\mid\mathscr F_{t-1}]$.

\begin{definition}[Heavy-tailed stochastic-gradient oracle]
\label[definition]{def:heavy-tailed oracles}
For $1<\alpha\le2$ and $\sigma\ge0$, the oracle satisfies, almost surely,
\begin{equation}\label{eq:lower-oracle-assumption}
 \mathbb E_t[G_t]=\nabla f(x_t),\qquad
 \mathbb E_t\!\left[\|G_t-\nabla f(x_t)\|^\alpha\right]\le\sigma^\alpha.
\end{equation}
We denote the class of such oracles by $\mathcal O_{\alpha,\sigma}(f)$,
with the batch convention below.
\end{definition}

\begin{remark}
Samples requested together at a fixed, adaptively chosen point are
conditionally independent given the history before that request, and each
satisfies \eqref{eq:lower-oracle-assumption}. No independence across
adaptively chosen query locations is assumed. Each sample counts as one
physical oracle call.
\end{remark}

\paragraph{Function classes and performance criteria.}
Let $\mathcal F_d(L,\mu,\Delta_0)$ be the class of globally $L$-smooth,
globally $\mu$-PL functions on $\mathbb R^d$ with $f(0)-f^\star\le\Delta_0$.
Translation allows us to use $x_0=0$ for lower bounds; the algorithms may
start at any specified $x_0$ satisfying the same gap bound.
An algorithm $\mathsf A$ using at most $T$ physical oracle calls may return
any measurable output $\widehat x_T$, not necessarily a queried point.
The minimax expected risk is
\begin{equation}\label{eq:risk-definition}
 \mathsf R_T^{(d)}:=
 \inf_{\mathsf A}\sup_{f\in\mathcal F_d(L,\mu,\Delta_0)}
 \sup_{G\in\mathcal O_{\alpha,\sigma}(f)}
 \mathbb E_{\mathsf A,G}[f(\widehat x_T)-f^\star].
\end{equation}
For high-probability accuracy, define separately
\begin{equation}\label{eq:confidence-definition}
 \mathsf T_{\epsilon,\delta}^{(d)}:=
 \inf\!\left\{T\in\mathbb N_0:\ \exists\mathsf A,\quad
 \sup_{f\in\mathcal F_d(L,\mu,\Delta_0)}
 \sup_{G\in\mathcal O_{\alpha,\sigma}(f)}
 \mathbb P_{\mathsf A,G}\bigl(f(\widehat x_T)-f^\star>\epsilon\bigr)
 \le\delta\right\}.
\end{equation}
Both criteria include algorithmic and oracle randomness; dependence on the
problem parameters is suppressed. The dimension-growing criterion takes
$\sup_{d\ge1}\mathsf R_T^{(d)}$. Only oracle calls are counted, not
arithmetic, geometric-median computation, or exact mirror subproblems.
In the algorithms, $t$ indexes optimization updates, each of which may
require several physical calls; $N$ denotes their total number.

\section{Lower Bound in the High-Dimensional Setting}\label{sec:lower}
We first establish a lower bound that separates the cost of deterministic
optimization from the cost imposed by heavy-tailed noise. The hard
dimension may grow with the target accuracy and noise level.

\begin{theorem}
\label{thm:hd-accuracy}
Fix $1<\alpha\le2$. There are constants $c_{0,\alpha}>0$ and
$\kappa_{\mathrm{hd}}\ge1$ such that the following holds whenever
$\kappa\ge\kappa_{\mathrm{hd}}$ and $0<\epsilon\le c_{0,\alpha}\Delta_0$.
For every $d\ge d_{\mathrm{hd}}$, where the threshold in
\eqref{eq:hd-dimension-maximum} satisfies
$d_{\mathrm{hd}}=\widetilde O_\alpha(\kappa^2[\Delta_0/\epsilon+Q_\epsilon])$,
any adaptive randomized gradient-only algorithm guaranteeing
$\mathbb E[f(\widehat x_T)-f^\star]\le\epsilon$ for all
$f\in\mathcal F_d(L,\mu,\Delta_0)$ and
$G\in\mathcal O_{\alpha,\sigma}(f)$ requires at least
\begin{equation}\label{eq:hd-main-lower}
\Omega\left(\kappa
 \log\frac{\Delta_0}{\epsilon}
 +\kappa\left(\frac{\sigma^2}{\mu\epsilon}\right)^{\thetaalpha}
 \right).
\end{equation}
\end{theorem}

% PROOF_IDEA_BEGIN hd
\noindent\textbf{Hard Instance: Idea.}
Two different obstacles produce the two terms. For the stochastic term, a
chain hides one new direction at a time; revealing that direction rarely,
but amplifying its value when revealed, preserves unbiasedness while
exhausting the available moment budget (Lemma~\ref{lem:thinned-oracle}).
The waiting time is paid across a chain of width proportional to $\kappa$.
For the noiseless term, geometrically decreasing layer amplitudes keep the
unresolved tail small enough that adding layers does not worsen the PL
constant (Lemmas~\ref{lem:geometric-tail-mass}--\ref{lem:weighted-chain}).
Compactification and random rotation prevent arbitrarily large or adaptive
queries from bypassing either chain. The maximum of these two lower bounds
controls their sum up to a constant; the exact-gradient family remains legal
even when a positive noise level is announced. The complete proof is in
Appendix~\ref{subsec:hd-completion}, using the constructions in
Appendices~\ref{subsec:hd-stochastic} and \ref{subsec:hd-noiseless}.
% PROOF_IDEA_END hd

\section{High-Probability Upper Bound}\label{sec:upper}
We next combine centered clipping with mirror descent to obtain a
high-probability guarantee on a closed convex domain. The geometric
assumption must control feasible descent, not merely the full gradient.

\subsection{Mirror Descent and the Mirror-PL Condition}\label{subsec:main-mirror}

Let $\mathcal X\subseteq\mathbb R^d$ be nonempty, closed, and convex,
and let $f$ be $L$-smooth on $\mathcal X$. The mirror map $\psi$ is
differentiable on an open convex neighborhood of $\mathcal X$ and
$1$-strongly convex in the Euclidean norm. Its Bregman divergence satisfies
\begin{equation}\label{eq:bregman-definition}
 D_\psi(z,x):=\psi(z)-\psi(x)-\langle\nabla\psi(x),z-x\rangle
 \ge\tfrac12\|z-x\|^2\qquad(x,z\in\mathcal X).
\end{equation}
For $v\in\R^d$ define
\begin{equation}\label{eq:mirror-model}
 Q_x(z;v):=\ip{v}{z-x}+L D_\psi(z,x),\qquad z\in\cX.
\end{equation}
As a function of $z$, this model is coercive and $L$-strongly convex,
so its minimizer exists and is unique even when $\mathcal X$ is unbounded.

\begin{definition}[Mirror-PL Condition]
\label[definition]{def:mirror-pl}
Define
\begin{equation}\label{eq:mirror-gradient-decrement}
 \mathcal G_{\psi,L}^2(x):=
 -2L\min_{z\in\cX}Q_x(z;\nabla f(x)).
\end{equation}
The function satisfies the $\mu$-mirror-PL condition if
\begin{equation}\label{eq:mirror-PL}
 \frac12\mathcal G_{\psi,L}^2(x)
 \ge\mu\bigl(f(x)-f^\star\bigr),\qquad x\in\cX.
\end{equation}
\end{definition}

If $\cX=\R^d$ and $\psi(x)=\norm{x}^2/2$, then
$\mathcal G_{\psi,L}^2(x)=\norm{\nabla f(x)}^2$, so
\eqref{eq:mirror-PL} is exactly \eqref{eq:PL-convention}.  For Euclidean
projected gradient descent, the projected-PL condition based on
$G_L(x)=L(x-\Pi_\cX(x-\nabla f(x)/L))$ implies mirror-PL, because projection
optimality yields
$\mathcal G_{\psi,L}^2(x)\ge\norm{G_L(x)}^2$.

In the unconstrained case, the PL inequality controls an available descent
direction. Under constraints, the negative gradient can instead point
outside the feasible region. The next proposition shows why ordinary PL
alone does not guarantee convergence of projected or mirror descent.

\begin{proposition}
\label[proposition]{prop:boundary-obstruction}
There exist a closed convex unbounded set $\cX\subset\R^2$, a globally
$3$-smooth function $f:\R^2\to\R$, and a constant $\mu=8/9$ such that
\begin{equation}\label{eq:ordinary-pl-counterexample}
 \frac12\norm{\nabla f(x)}^2
 \ge \mu\bigl(f(x)-\inf_{z\in\cX}f(z)\bigr)
 \qquad\text{for every }x\in\cX,
\end{equation}
but exact projected or mirror descent can remain forever at a nonoptimal
point of $\cX$.
\end{proposition}

The construction and proof are given in Appendix~\ref{app:upper-proofs}.

\subsection{Centered Gradient Clipping}
For $v\in\mathbb R^d$ and radius $r>0$, define the radial clipping operator
\begin{equation}\label{eq:clip-definition}
 \operatorname{clip}_r(v):=
 \begin{cases}
 v,&\|v\|\le r,\\[1mm]
 r\,v/\|v\|,&\|v\|>r.
 \end{cases}
\end{equation}

Clipping a raw stochastic gradient around the origin can introduce a bias
that depends on $\|\nabla f(x)\|$, which can be arbitrarily large on an unbounded
domain. We avoid this dependence by first estimating the gradient with a
geometric-median pilot and then clipping a second collection of block means
around that pilot, using samples independent of the pilot batch.

At a fixed query point $x$, divide $2Kb$ fresh samples into $2K$ blocks
and form
\begin{equation}\label{eq:block-means}
 Z_j:=\frac1b\sum_{\ell=1}^b G(x,\xi_{j,\ell}),
 \qquad j=1,\ldots,2K.
\end{equation}
Choose a geometric median of the first $K$ block means as the pilot,
\begin{equation}\label{eq:geometric-median}
 a\in\arg\min_{u\in\mathbb R^d}\sum_{j=1}^K\|u-Z_j\|,
\end{equation}
and average the remaining block means after clipping around $a$:
\begin{equation}\label{eq:centered-clipped-blocks}
\widehat g(x):=\frac1K\sum_{j=1}^K Y_j, \quad \text{where}\quad Y_j:=a+\operatorname{clip}_r(Z_{K+j}-a).
\end{equation}
Using $\widehat g(x)$ in the mirror update gives Centered-Clipped Mirror
Descent (\textbf{CC-MD}), presented in
Algorithm~\ref{alg:stagewise-centered-clipped-mirror-descent}.
Appendix~\ref{subsec:upper-clipping} gives the estimator parameters and
its conditional accuracy guarantee.

\begin{algorithm}[t]
\caption{Centered-Clipped Mirror Descent (CC-MD)}
\label{alg:stagewise-centered-clipped-mirror-descent}
\begin{algorithmic}[1]
\Require  Initial point $x_0\in\mathcal X$, constants $L,\mu>0$, mirror map $\psi$, number of stages $S$, stage length $J$, number of blocks per half $K$, block sizes $\{b_s\}_{s=0}^{S-1}$, clipping thresholds $\{r_s\}_{s=0}^{S-1}$.

\State \textbf{if} $\mu>L$ \textbf{then return} $x_0$.
\For{$s=0,\ldots,S-1$}

    \For{$t=sJ,\ldots,(s+1)J-1$}
        \State Query $2Kb_s$ fresh samples
            $G_{t,k,\ell}=G(x_t,\xi_{t,k,\ell})$,
            $k=1,\ldots,2K$, $\ell=1,\ldots,b_s$

        \State Form the block means
            $\displaystyle
                Z_{t,k}\gets
                \frac{1}{b_s}\sum_{\ell=1}^{b_s}G_{t,k,\ell},
                \quad k=1,\ldots,2K.
            $

        \State Compute a geometric-median pilot from the first $K$ blocks:
        \Statex \hspace{\algorithmicindent}\hspace{\algorithmicindent}
            $\displaystyle
                a_t\in
                \operatorname*{arg\,min}_{u\in\mathbb R^d}
                \sum_{k=1}^{K}\|u-Z_{t,k}\|.
            $

        \State Clip the remaining $K$ block means around the pilot:
        \Statex \hspace{\algorithmicindent}\hspace{\algorithmicindent}
            $\displaystyle
                Y_{t,k}\gets
                a_t+\operatorname{clip}_{r_s}(Z_{t,K+k}-a_t),
                \quad k=1,\ldots,K.
            $

        \State $\displaystyle
            \widehat g_t\gets\frac{1}{K}\sum_{k=1}^{K}Y_{t,k}.
        $

        \State Perform the mirror-descent update:
        \Statex \hspace{\algorithmicindent}\hspace{\algorithmicindent}
            $\displaystyle
                x_{t+1}\gets
                \operatorname*{arg\,min}_{z\in\mathcal X}
                \left\{
                    \langle\widehat g_t,z-x_t\rangle
                    +L D_\psi(z,x_t)
                \right\}.
            $
    \EndFor
\EndFor

\State \Return $x_{\mathrm{out}}\gets x_{JS}$.
\end{algorithmic}
\end{algorithm}

\subsection{Convergence Result}
The next theorem combines the robust estimator with the mirror-PL
contraction. The parameter choices and oracle counts are explicit in the
appendix.

\begin{theorem}
\label{thm:upper-accuracy}
Under the domain and mirror-map assumptions of Section~\ref{subsec:main-mirror},
let $f$ be $L$-smooth and $\mu$-mirror-PL, with
$f(x_0)-f^\star\le\Delta_0$. Suppose the oracle satisfies
Definition~\ref{def:heavy-tailed oracles}. For every
$\epsilon\in(0,\Delta_0)$ and $\delta\in(0,1/2)$, return $x_{\mathrm{out}}=x_0$ with zero queries if $\mu>L$; otherwise run CC-MD with the parameter choices in Appendix~\ref{subsec:upper-md}. The output satisfies $\Pp\bigl(f(x_{\mathrm{out}})-f^\star\le\epsilon\bigr)\ge1-\delta$. For $\mu\le L$, the number of stochastic-gradient calls is at most
\[
\widetilde O\left(\left(\kappa\left\lceil\log_2\frac{\Delta_0}{\epsilon}\right\rceil
 +\kappa\left(\frac{\sigma^2}{\mu\epsilon}\right)^{\thetaalpha}\right)\log\left(
 \frac{16\kappa}{\delta}
 \left\lceil\log_2\frac{\Delta_0}{\epsilon}\right\rceil
 \right)\right)
\]
\end{theorem}

% PROOF_IDEA_BEGIN upper
\noindent\textbf{Upper Bound: Idea.}
The estimator must suppress outliers without truncating a possibly large
true gradient. Block averages provide weak concentration, and a
geometric-median pilot supplies a reliable center; clipping fresh blocks
around this center yields accuracy independent of the gradient magnitude
(Lemma~\ref{lem:centered-clipping}). The deterministic mirror step then
converts squared estimation error into an additive objective-gap error
(Lemma~\ref{lem:inexact-mirror-contraction}). At gap scale $\Delta$, an
error of order $\sqrt{\mu\Delta}$ permits a constant-factor improvement
after $O(\kappa)$ steps. Successive stages halve the gap: their deterministic
costs add, whereas their increasing sampling costs are dominated by the
finest scale. A conditional union bound makes every stage valid simultaneously.
The full proof and parameter choices are in Appendix~\ref{subsec:upper-md}.
% PROOF_IDEA_END upper

\begin{remark}
The related guarantees discussed above cover constrained mirror descent
under sub-Gaussian noise \citep{FatkhullinHe2024} and unconstrained PL optimization under finite
centered $\alpha$-moments \citep{SadievEtAl2023}.
Here centered clipping yields a high-probability guarantee under the $\alpha$-heavy tailed model on closed convex, possibly unbounded domains.
\end{remark}

\section{Fixed-Dimensional Refinements}\label{sec:fixed}
We now fix the dimension and work with the unconstrained class
$\mathcal F_d(L,\mu,\Delta_0)$. The high-dimensional chains no longer
apply directly; different information-hiding and geometric-search
mechanisms yield sharper bounds.

\begin{theorem}
\label{thm:fd-lower-main}
Fix $d\ge1$, $1<\alpha\le2$, $\sigma\ge0$, and $\kappa\ge6$.
There is a constant $c_0>0$ such that, for $0<\epsilon\le c_0\Delta_0$,
any adaptive randomized gradient-only algorithm guaranteeing
$\mathbb E[f(\widehat x_T)-f^\star]\le\epsilon$ for all
$f\in\mathcal F_d(L,\mu,\Delta_0)$ and
$G\in\mathcal O_{\alpha,\sigma}(f)$ requires at least
\begin{equation}\label{eq:fd-lower-sum}
\Omega\left(
 \log\frac{\Delta_0}{\epsilon}
 +\left(\frac{\sigma^2}{\mu\epsilon}\right)^{\thetaalpha}
 \right).
\end{equation}
\end{theorem}

% PROOF_IDEA_BEGIN fd_lower
\noindent\textbf{Hard Instance: Idea.}
Fixed dimension requires information hiding without a long chain of
orthogonal directions. For the stochastic term, two shifted quadratics
produce identical oracle answers until a rare informative sample appears;
finite-moment calibration makes the waiting time equal to the claimed
statistical scale (Lemma~\ref{lem:fd-rare-event}). Conditional on no
revelation, the average squared distance to the two minimizers stays large,
which directly gives an expected-error lower bound. For the noiseless term,
a recursively nested interval hides the minimizer, and each exact-gradient
query reveals only a constant expected number of location bits
(Lemma~\ref{lem:fd-tree-info}). Both constructions occupy one coordinate
and embed into any fixed dimension. Taking their maximum combines the two
obstacles without imposing a condition-number multiplier on the noise term.
The complete proof is in Appendix~\ref{subsec:fd-lower-bisection}.
% PROOF_IDEA_END fd_lower

\begin{remark*}[Origin of the $\kappa$ multiplier]
The stochastic barrier $\kappa Q_\epsilon$ in
Theorem~\ref{thm:hd-accuracy} couples mean estimation with sequential
exploration of $m=\Theta(\kappa)$ hidden directions. In the genuinely noisy
regime, the next direction is revealed with probability at most
$q=\Theta_\alpha(Q_\epsilon^{-1})$, forcing $\Omega(m/q)$ queries;
random rotation prevents shortcuts (Appendix~\ref{subsec:hd-stochastic}).
The exact chain likewise uses $\Theta(\kappa)$ directions per accuracy
layer (Appendix~\ref{subsec:hd-noiseless}). Thus these multipliers reflect
the high-dimensional geometry of the constructions, not mean estimation
alone. Indeed, Theorem~\ref{cor:fd-123} removes the polynomial $\kappa$
multiplier from the stochastic term when $p_\alpha\ge d-1$, including
$d=1,2,3$. The deterministic gaps for $d\ge2$ and the complementary
stochastic regime remain unresolved.
\end{remark*}

For the upper bounds, we use geometric search tailored to the dimension.
Both algorithms access only stochastic gradients and are evaluated using
$\mathsf T_{\epsilon,\delta}^{(d)}$ from \eqref{eq:confidence-definition}.

\noindent\textbf{Robust stochastic bisection ($d=1$).}
The PL condition makes the minimizer set an interval, with the derivative
negative to its left and positive to its right. Starting from a bracket
obtained from $\Delta_0$, we estimate the derivative at its midpoint.
A reliable sign selects the half containing a minimizer; a small estimate
certifies accuracy through PL. Algorithm~\ref{alg:robust-stochastic-bisection}
gives the pseudocode, and Appendix~\ref{subsec:fd-lower-bisection}
proves the guarantee.

\noindent\textbf{Gradient-flow trapping ($d\ge2$).}
Maintain a compact box $R$ and an interior pivot $x$ with the boundary
certificate $f(y)>f(x)-c\|y-x\|$ for all $y\in\partial R$.
If every gradient in $R$ exceeded $c$ in norm, the normalized descent flow
would reach the boundary with a larger decrease, a contradiction.
Hence $R$ contains a point with gradient norm at most $c$
(Lemma~\ref{lem:fd-trap}). Given $f(x)-f^\star\le\Delta$, set
$\eta=\sqrt{\mu\Delta}$ and center a cube of side $32\Delta/\eta$
at $x$, with $c=\eta/8$. The gap bound certifies its boundary without any queries.
Two parallel cuts along a longest side either retain the pivot or replace
it by a certified lower-value grid point. Each refinement
reduces volume by at least $1/6$ while controlling the increase in $c$
(Lemma~\ref{lem:fd-box-refinement}). Reconstruction tolerances and mesh
widths decrease with the cut scale, keeping the accumulated slope increase
below $\eta/8$; a strict interpolation margin extends the certificate from
grid vertices to the entire new boundary.

Relative values on the cuts are reconstructed by stochastic line integrals.
A dyadic tree shares coarse-segment estimates among grid points, rather
than independently estimating each full path from the pivot
(Lemma~\ref{lem:fd-hyperplane-profile}). After $O_d(\log(2+\kappa))$
refinements, $c\le\eta/4$ and $\operatorname{diam}(R)\le\eta/(4L)$;
smoothness makes the pivot $\eta/2$-stationary, and PL halves the gap.
The flow is only a proof device; no function-value oracle is queried.
See Algorithm~\ref{alg:gradient-flow-trapping} and
Appendices~\ref{subsec:fd-reconstruction}--\ref{subsec:fd-trapping}
for reconstruction, error budgets, and the complete stagewise method.
Combining these guarantees with the lower bounds gives the following phase diagram.

Below, we provide the upper bound for any fixed $d\ge 1$.
\begin{theorem}
\label{cor:fd-123}
Fix $d\ge1$, $1<\alpha\le2$, $\sigma\ge0$, and $\kappa\ge6$.
For a sufficiently small constant $c_0>0$, let
$0<\epsilon\le c_0\Delta_0$ and $0<\delta\le1/4$.
Suppressing logarithmic factors in $\kappa$, $\Delta_0/\epsilon$, and
$1/\delta$, we have
\begin{align}
 &d=1:\
 \mathsf T^{(1)}_{\epsilon,\delta}
 =\widetilde\Theta\left[
 \log\frac{\Delta_0}{\epsilon}+Q_\epsilon
 \right],
 \label{eq:fd-d1-summary}\\
 &d=2:\
 \Omega\left[
 \log\frac{\Delta_0}{\epsilon}+Q_\epsilon
 \right]
 \le \mathsf T^{(2)}_{\epsilon,1/4}
 \le\widetilde O\left[
 \kappa^{2/3}\log\frac{\Delta_0}{\epsilon}
 +Q_\epsilon
 \right],
 \label{eq:fd-d2-summary}\\
 &d=3:\
\Omega\left[
 \log\frac{\Delta_0}{\epsilon}+Q_\epsilon
 \right]
 \le \mathsf T^{(3)}_{\epsilon,1/4}
 \le\widetilde O\left[
 \kappa\log\frac{\Delta_0}{\epsilon}
 +Q_\epsilon
 \right],
 \label{eq:fd-d3-summary}\\
&d>3, p_\alpha\ge d-1:\
 \Omega\left[
 \log\frac{\Delta_0}{\epsilon}+Q_\epsilon
 \right]
 \le \mathsf T^{(d)}_{\epsilon,1/4}
 \le \widetilde O\left[
 \kappa^{(d-1)/2}\log\frac{\Delta_0}{\epsilon}
 +Q_\epsilon
 \right],
 \label{eq:fd-d-summary}\\
 &d>3, p_\alpha<d-1:
 \ \Omega\left[\log\frac{\Delta_0}{\epsilon}+Q_\epsilon
 \right]
 \le \mathsf T^{(d)}_{\epsilon,1/4}
 \le \widetilde O\left[
 \kappa^{\frac{d-1}{2}}\log\frac{\Delta_0}{\epsilon}
 +\kappa^{\frac{d-1-p_\alpha}{2}}Q_\epsilon
 \right].\label{eq:fd-dgreater3-summary}
\end{align}
\end{theorem}

% PROOF_IDEA_BEGIN fd_phase
\noindent\textbf{Upper Bound: Idea.}
In one dimension, the derivative's sign identifies which side of a query
contains the minimizer interval, so a robust sign estimate supports bisection
(Lemma~\ref{lem:fd-1d-geometry}). In higher dimensions, boundary values
certify that a box traps a descending gradient trajectory; shrinking such a
box eventually makes its pivot nearly stationary
(Lemmas~\ref{lem:fd-trap}--\ref{lem:fd-box-refinement}). Function values
are unavailable, so shared stochastic line integrals reconstruct relative
values on each cutting surface. Across dyadic levels, the growth in the
number of surface points competes with the decay in integration length.
The resulting noise sum is bounded up to logarithms precisely when
$p_\alpha\ge d-1$ (Lemma~\ref{lem:fd-hyperplane-profile}); otherwise its
finest levels produce the extra factor. PL converts stationarity to accuracy,
and Theorem~\ref{thm:app-fixed-lower} supplies the lower terms.
The full derivation is in Appendix~\ref{subsec:fd-phase}, in the proof of
Corollary~\ref{cor:app-fixed-phase}.
% PROOF_IDEA_END fd_phase

Thus the full complexity is matched up to logarithmic factors in dimension
one. In dimensions two and three, and in fixed dimension $d>3$ when
$p_\alpha\ge d-1$, the stochastic term $Q_\epsilon$ is matched; the
condition-number dependence of the deterministic term remains unresolved.

\section{Conclusion and Future Directions}\label{sec:conclusion}
We establish noise-adaptive complexity bounds for smooth PL optimization
with conditionally unbiased stochastic gradients satisfying only a finite
centered $\alpha$-moment. In growing dimension, the lower bound separates
the exact-gradient transient from the heavy-tailed estimation cost and
remains informative when $\sigma=0$. Under mirror-PL, centered-clipped
mirror descent matches this bound up to logarithmic factors with high
probability, including on closed convex, unbounded domains. Fixed dimension
permits sharper geometric search: robust stochastic bisection characterizes
the one-dimensional complexity up to logarithmic factors, while gradient-flow trapping matches the
stochastic term for $d=2,3$ and for $d>3$ whenever $p_\alpha\ge d-1$.

The remaining questions concern the deterministic condition-number
dependence in dimensions $d\ge2$ and the stochastic complexity when
$p_\alpha<d-1$. In particular, it remains to determine whether the extra
surface-reconstruction factor is necessary or can be removed by a different
search strategy. Developing methods that do not require prior knowledge
of the noise and PL parameters is another natural direction.

\newpage
\bibliographystyle{plainnat}
\bibliography{reference}

\appendix
\newpage
\begin{center}
    \large APPENDIX
\end{center}

\makeatletter
\providecommand*{\theHALG@line}{}
\renewcommand*{\theHALG@line}{\thealgorithm.\arabic{ALG@line}}
\makeatother
\noindent\textbf{Notation and scope of the proofs.}
We use the conventions of Section~\ref{sec:prelim} throughout:
$f^\star$ and $\mathcal X^\star$ denote the optimal value and minimizer
set, and $\theta_\alpha$, $\beta_\alpha$, $p_\alpha$, $Q_\epsilon$, and
$A_\epsilon$ are defined in
\eqref{eq:heavy-exponents}--\eqref{eq:Q-alpha}.
The filtration $\mathscr F_{t-1}$ records the full pre-call history;
$\mathcal F_d(L,\mu,\Delta_0)$ denotes a function class, not a filtration.
The expected risk $\mathsf R_T^{(d)}$ and the high-probability query
complexity $\mathsf T_{\epsilon,\delta}^{(d)}$ are distinct criteria,
defined in \eqref{eq:risk-definition} and
\eqref{eq:confidence-definition}. All conditional oracle statements use
this full history, with conditional independence only for samples requested
together at a fixed location.

The confidence parameter is $\delta$; spatial mesh widths are $h$ or $h_t$.
A prescribed-budget output is denoted by $\widehat x_T$, and an algorithmic
output by $x_{\mathrm{out}}$. Construction-specific indices and parameters
are introduced locally. In particular, the layer count in Appendix~\ref{app:lower-proofs}
and the block count in Appendix~\ref{app:upper-proofs} are different uses of
$K$. Constants such as $C_d,C_\alpha,C_{d,\alpha}$ may change between
bounds; named construction constants and algorithm parameters do not.
Only physical oracle calls are counted, under the convention in
Section~\ref{sec:prelim}.

\section{Comparison with \texorpdfstring{\citet{MasihaKiyavashThiran2026}}{Masiha et al. (2026)}}\label{appendix-detailed comparison}
We compare with the statements in version~2 of
\citet{MasihaKiyavashThiran2026}. Write $\nu\in[1,2)$ for their geometric
exponent, to distinguish it from the noise-moment order $\alpha$ and the
geometric decay factor $\gamma$ used below. Their model imposes
$f-f^\star\le\tau\|\nabla f\|^\nu$ on the initial sublevel set while retaining
global smoothness. Their Theorems~2.1, 3.1, and 4.1 respectively establish
triviality under a global inequality with $\nu<2$, a deterministic lower
bound of order $L\tau^{2/\nu}\epsilon^{-(2-\nu)/\nu}$, and a bounded-variance
lower bound adding $L\sigma^2\tau^{4/\nu}\epsilon^{-(4-\nu)/\nu}$.
Those theorems do not state a separate globally PL endpoint result at $\nu=2$.

Our geometric assumption is instead the global quadratic PL inequality,
while the centered noise has only a finite $\alpha$-moment.
The finite-variance specialization of our stochastic term is
$\kappa\sigma^2/(\mu\epsilon)$. Substituting $\nu=2$ and
$\tau=(2\mu)^{-1}$ formally into the polynomial noise expression above gives
$L\sigma^2/(4\mu^2\epsilon)$, the same algebraic scale.
This algebraic comparison does not extend their theorems beyond the
stated exponent range. Our proofs separately establish the global-PL
heavy-tail bound, its exact-gradient transient, the mirror-PL upper bound,
and the fixed-dimensional refinements.

\section{Proofs for Section~\ref{sec:lower}}\label{app:lower-proofs}

The lower bound separates two obstructions: learning the next hidden direction
under heavy-tailed noise, and traversing many accuracy scales even with exact
gradients. We first develop each normalized construction and then choose its
parameters when matching the prescribed $L$, $\mu$, $\sigma$, and $\epsilon$.
The two resulting ambient dimensions are compared only when the constructions
are combined in the final subsection.

\subsection{Main ideas and proof sketch}\label{subsec:hd-overview}

The lower bound has two logically independent hard families. The first is
the Bernoulli-thinned heavy-tailed chain developed in the next sections. 
The second is an exact-gradient geometrically weighted chain. 
Because an exact oracle has centered $\alpha$-moment zero, it is legal for
every announced noise level $\sigma\ge0$. The noise-adaptive lower bound is therefore at least the maximum of the two component risks, which gives the additive noise-adaptive statement.

For the stochastic obstruction, take a chain whose width is proportional to
$\kappa$. Its next hidden gradient coordinate is returned rarely and enlarged
on revelation so that the oracle remains unbiased. The moment constraint
sets the waiting time per coordinate, while radial compactification makes the
objective globally smooth and globally PL. A random rotation prevents an
adaptive method from bypassing the chain by guessing hidden directions.

The exact-gradient transient needs a different construction. Concatenating
unweighted chains would worsen the PL constant with the number of layers.
Instead, we decrease the coordinate amplitudes geometrically: the squared
mass of the unresolved tail stays proportional to the layer width times the
first unresolved squared amplitude. This permits logarithmically many layers
without changing the order of the condition number. A guard that vanishes
near the minimizer and a rotation argument with nonuniform thresholds preserve
the last layer's small objective gap.

\subsection{Bernoulli-thinned robust zero-chain and global PL compactification}\label{subsec:hd-chain}
We first construct the scalar activation, coupling, and barrier functions. Their local bounds will control both the information structure and the global geometry of the chain.

Define
\begin{equation}\label{eq:Theta-def}
\hat{\Theta}(t)=
\begin{cases}
0,&t\le\frac12,\\[1mm]
3(2t-1)^2-2(2t-1)^3,&\frac12<t<1,\\[1mm]
1,&t\ge1,
\end{cases}
\end{equation}
\begin{equation}\label{eq:phi-def}
  \phi(t)=1-\exp\!\left(-\frac{(t-1)^2}{2}\right),
\end{equation}
and
\begin{equation}\label{eq:B-def}
B(t)=
\begin{cases}
2+\frac12(t+\frac12)^2,&t\le-\frac12,\\[1mm]
2,&-\frac12<t\le\frac12,\\[1mm]
2(1-\hat{\Theta}(t)),&\frac12<t<1,\\[1mm]
\frac12(t-1)^2,&t\ge1.
\end{cases}
\end{equation}
All three functions are continuously differentiable.  Their derivatives are
Lipschitz, although their second derivatives need not agree at the break
points.

\begin{lemma}\label[lemma]{lem:scalar}
For every $t\in\R$,
\begin{align}
&0\le\hat{\Theta}(t)\le1,
&&\norm{\hat{\Theta}'}_\infty\le3,
&&\Lip(\hat{\Theta}')\le24,\label{eq:theta-bounds}\\
&0\le\phi(t)\le1,
&&\norm{\phi'}_\infty=e^{-1/2}<1,
&&\Lip(\phi')\le1,\label{eq:phi-bounds}\\
&&&\Lip(B')\le48.\label{eq:B-smooth}
\end{align}
In addition,
\begin{equation}\label{eq:B-aux-bounds}
 B(t)\le\frac12t^2+2,
 \qquad
 \abs{B'(t)-t}\le7,
 \qquad
 tB'(t)\ge\frac12t^2-\frac{13}{2}.
\end{equation}
\end{lemma}

\begin{proof}
On $(1/2,1)$ put $u=2t-1\in(0,1)$.  Then
\[
 \hat{\Theta}(t)=3u^2-2u^3,
 \qquad
 \hat{\Theta}'(t)=12u(1-u),
 \qquad
 \hat{\Theta}''(t)=24(1-2u).
\]
This proves \eqref{eq:theta-bounds}, including the endpoint matching.
Furthermore,
\[
 \phi'(t)=(t-1)e^{-(t-1)^2/2},
 \qquad
 \phi''(t)=\bigl(1-(t-1)^2\bigr)e^{-(t-1)^2/2}.
\]
The first expression has maximum absolute value $e^{-1/2}$, and the second
has absolute value at most $1$, proving \eqref{eq:phi-bounds}.

For $B$, the derivative equals
\[
 B'(t)=
 \begin{cases}
 t+\frac12,&t<-\frac12,\\
 0,&-\frac12<t<\frac12,\\
 -2\hat{\Theta}'(t),&\frac12<t<1,\\
 t-1,&t>1.
 \end{cases}
\]
The largest absolute slope of these pieces is $48$, and the derivative is
continuous at the break points, giving \eqref{eq:B-smooth}.

The three inequalities in \eqref{eq:B-aux-bounds} follow by checking the
same four intervals.  For completeness, on $t\le-1/2$ one has
$B(t)=2+\frac12(t+1/2)^2\le2+t^2/2$ and $B'(t)-t=1/2$; on
$[-1/2,1/2]$ they are immediate; on $(1/2,1)$, one uses
$0\le B\le2$, $-6\le B'\le0$, and $t\le1$; and on $[1,\infty)$ one uses
$B(t)=\frac12(t-1)^2$ and $B'(t)=t-1$.  In particular, on $(1/2,1)$,
$tB'(t)\ge-6\ge t^2/2-13/2$.  This proves all claims.
\end{proof}

We now combine the scalar functions into a chain. Its support property is the information constraint that will survive Bernoulli thinning.

For $N\ge1$, define
\begin{equation}\label{eq:chain-def}
 F_N(x)
 :=\phi(x_1)
   +\sum_{i=1}^{N-1}\hat{\Theta}(x_i)\phi(x_{i+1})
   +\sum_{i=1}^N B(x_i),
 \qquad x\in\R^N.
\end{equation}
For $a\ge0$, define the progress index
\begin{equation}\label{eq:progress-def}
 \prog_a(x)
 :=\max\bigl(\{i:\abs{x_i}>a\}\cup\{0\}\bigr).
\end{equation}

\begin{lemma}\label[lemma]{lem:base-chain}
For every $N\ge1$:
\begin{enumerate}[label=(\roman*)]
\item $F_N(x)\ge0$ for all $x$, $F_N(\one_N)=0$, and
$F_N(0)<\frac52N$.
\item $F_N$ is globally $79$-smooth.
\item If $k=\prog_{1/4}(x)$, then
\begin{equation}\label{eq:zero-chain-support}
  \nabla_jF_N(x)=0\qquad\text{for every }j\ge k+2.
\end{equation}
If $k<N$, then
\begin{equation}\label{eq:next-coordinate-bound}
  \abs{\nabla_{k+1}F_N(x)}\le1.
\end{equation}
\item If $N=2m$ and $\prog_{1/4}(x)<m$, then
\begin{equation}\label{eq:tail-barrier}
  F_{2m}(x)\ge2m.
\end{equation}
\end{enumerate}
\end{lemma}

\begin{proof}
Nonnegativity follows from the nonnegativity of $\hat{\Theta},\phi,B$.  At
$x=\one_N$, one has $\phi(1)=B(1)=0$, hence $F_N(\one_N)=0$.  At the origin,
all coupling terms vanish because $\hat{\Theta}(0)=0$, so
\[
 F_N(0)=\phi(0)+2N<\frac52N.
\]

For an interior coordinate $2\le i\le N-1$,
\begin{equation}\label{eq:chain-gradient-coordinate}
 \nabla_iF_N(x)
 =\hat{\Theta}(x_{i-1})\phi'(x_i)
  +\hat{\Theta}'(x_i)\phi(x_{i+1})
  +B'(x_i).
\end{equation}
At $i=1$, the incoming term is $\phi'(x_1)$; at $i=N$, omit the outgoing
term. Omit nonexistent neighbors throughout, including when $N=1$. By
\cref{lem:scalar}, for arbitrary $x,y$,
\begin{align*}
 \abs{\nabla_iF_N(x)-\nabla_iF_N(y)}
 &\le 3\abs{x_{i-1}-y_{i-1}}
      +73\abs{x_i-y_i}
      +3\abs{x_{i+1}-y_{i+1}}.
\end{align*}
Indeed, the first product contributes at most
$3\abs{x_{i-1}-y_{i-1}}+\abs{x_i-y_i}$, the second contributes at most
$24\abs{x_i-y_i}+3\abs{x_{i+1}-y_{i+1}}$, and $B'$ contributes at most
$48\abs{x_i-y_i}$.  The corresponding nonnegative tridiagonal coefficient
matrix has maximum row sum and maximum column sum at most $79$, hence its
Euclidean operator norm is at most
$\sqrt{\norm{A}_1\norm{A}_\infty}\le79$.  Thus $F_N$ is $79$-smooth.

Let $k=\prog_{1/4}(x)$.  If $j\ge k+2$, then
$\abs{x_{j-1}},\abs{x_j},\abs{x_{j+1}}\le1/4$.  Consequently
$\hat{\Theta}(x_{j-1})=0$, $\hat{\Theta}'(x_j)=0$, and $B'(x_j)=0$, which proves
\eqref{eq:zero-chain-support}.  For $j=k+1$ with $k\ge1$, all terms in
\eqref{eq:chain-gradient-coordinate} vanish except possibly
$\hat{\Theta}(x_k)\phi'(x_{k+1})$.  When $k=0$, the only possibly nonzero term
is instead the leading term $\phi'(x_1)$.  In both cases,
$\abs{\nabla_{k+1}F_N(x)}\le1$, proving
\eqref{eq:next-coordinate-bound}.

Finally, if $N=2m$ and $\prog_{1/4}(x)<m$, then each coordinate
$i=m+1,\ldots,2m$ lies in $[-1/4,1/4]$.  Hence $B(x_i)=2$, and all other
terms are nonnegative.  Thus
\[
 F_{2m}(x)\ge\sum_{i=m+1}^{2m}B(x_i)=2m.
\]
\end{proof}

Fix $m\ge1$ and $q\in(0,1]$.  Let $z\sim\mathrm{Bernoulli}(q)$.  For a
query $x\in\mathbb R^{2m}$, put $k=\operatorname{prog}_{1/4}(x)$ and define
\begin{equation}\label{eq:thinned-oracle}
 [g_{m,q}(x,z)]_i
 :=\nabla_iF_{2m}(x)
 \left[1+\mathbf 1_{\{i>k\}}\left(\frac zq-1\right)\right].
\end{equation}
Already exposed coordinates are exact, whereas every hidden coordinate is
multiplied by $z/q$.  By the zero-chain property, at most the single
coordinate $k+1$ can be both hidden and nonzero.

\begin{lemma}
\label[lemma]{lem:thinned-oracle}
For every $x\in\mathbb R^{2m}$,
\begin{align}
 \mathbb E_z[g_{m,q}(x,z)]&=\nabla F_{2m}(x),
 \label{eq:thin-unbiased}\\
 \mathbb E_z\|g_{m,q}(x,z)-\nabla F_{2m}(x)\|^\alpha
 &\le2q^{1-\alpha}.
 \label{eq:thin-alpha-moment}
\end{align}
When $q=1$, the left-hand side of \eqref{eq:thin-alpha-moment} is in fact
zero.  Moreover,
\begin{equation}\label{eq:prob-zero-chain}
 \operatorname{prog}_0(g_{m,q}(x,z))
 \le \operatorname{prog}_{1/4}(x)+1
\end{equation}
always, and
\begin{equation}\label{eq:reveal-probability}
 \mathbb P_z\!\left(
 \operatorname{prog}_0(g_{m,q}(x,z))>
 \operatorname{prog}_{1/4}(x)\right)\le q.
\end{equation}
\end{lemma}

\begin{proof}
For $i\le k$, the multiplier in \eqref{eq:thinned-oracle} is one; for
$i>k$, it is $z/q$, whose expectation is one.  This proves unbiasedness.
Only coordinate $k+1$ can fluctuate, and its true magnitude is at most one.
Therefore
\begin{align*}
 &\mathbb E_z\|g_{m,q}(x,z)-\nabla F_{2m}(x)\|^\alpha\\
 &\quad\le
 q\left(\frac1q-1\right)^\alpha+(1-q)
 =(1-q)^\alpha q^{1-\alpha}+(1-q)\\
 &\quad\le q^{1-\alpha}+1\le2q^{1-\alpha},
\end{align*}
where the last inequality uses $q^{1-\alpha}\ge1$.  For $q=1$, the
multiplier is identically one and the error is zero.  The support statement
follows from the deterministic zero-chain property.  Progress beyond the
current coordinate can occur only when $z=1$, which has probability $q$.
\end{proof}

The zero-chain support rule alone does not prevent unbounded queries from
bypassing the hidden-coordinate sequence. We therefore compactify the
effective queries and add a quadratic guard. The construction below
preserves the information constraint and verifies global smoothness and PL;
its Jacobian estimates also control the oracle's centered moment.

Fix
\begin{equation}\label{eq:eta-R}
 \eta:=\frac1{20},
 \qquad
 R_m:=16\sqrt m.
\end{equation}
Define $h_m:[0,\infty)\to[0,2R_m)$ by
\begin{equation}\label{eq:h-def}
 h_m(r)=
 \begin{cases}
 r,&0\le r\le R_m,\\[1mm]
 R_m+R_m\left(1-e^{-(r-R_m)/R_m}\right),&r>R_m,
 \end{cases}
\end{equation}
and define $\rho_m:\R^d\to\R^d$ by
\begin{equation}\label{eq:rho-def}
 \rho_m(x)=
 \begin{cases}
 \dfrac{h_m(\norm{x})}{\norm{x}}x,&x\ne0,\\[2mm]
 0,&x=0.
 \end{cases}
\end{equation}

\begin{lemma}\label[lemma]{lem:radial}
The map $\rho_m$ is continuously differentiable and satisfies
\begin{align}
 &\rho_m(x)=x\quad\text{if }\norm{x}\le R_m,
 &&\norm{\rho_m(x)}<2R_m,\label{eq:rho-range}\\
 &\norm{J\rho_m(x)}_{\mathrm{op}}\le1,
 &&\Lip(J\rho_m)\le\frac{16}{R_m}.\label{eq:rho-jacobian}
\end{align}
In addition, if $\norm{x}\ge16R_m$, then
\begin{equation}\label{eq:rho-far-jacobian}
 \norm{J\rho_m(x)}_{\mathrm{op}}\le\frac{2R_m}{\norm{x}}.
\end{equation}
\end{lemma}

\begin{proof}
For $x\ne0$, put $r=\norm{x}$, $e=x/r$, and $a(r)=h_m(r)/r$.
A direct differentiation gives
\begin{equation}\label{eq:radial-jacobian-formula}
 J\rho_m(x)=a(r)I+\bigl(h_m'(r)-a(r)\bigr)ee^\top.
\end{equation}
Thus the radial eigenvalue is $h_m'(r)$ and every tangential eigenvalue is
$a(r)$.  For $r\le R_m$, both equal $1$.  For $r>R_m$,
$0<h_m'(r)=e^{-(r-R_m)/R_m}\le1$.  Also $h_m(R_m)=R_m$ and
$h_m'\le1$, so $h_m(r)\le r$ and hence $0<a(r)\le1$.  This proves the first
part of \eqref{eq:rho-jacobian}; \eqref{eq:rho-range} is immediate from
\eqref{eq:h-def}.

For $r>R_m$, one has $h_m(r)\le2R_m$,
$\abs{h_m''(r)}\le R_m^{-1}$,
\[
 \abs{a'(r)}
 =\abs{\frac{h_m'(r)r-h_m(r)}{r^2}}
 \le\frac3{R_m},
 \qquad
 \abs{(h_m'-a)'(r)}\le\frac4{R_m}.
\]
Moreover, the differential of $ee^\top$ has operator norm at most $2/r$.
Differentiating \eqref{eq:radial-jacobian-formula} therefore gives, on
$r>R_m$,
\[
 \norm{D(J\rho_m)(x)}_{\mathrm{op}}
 \le\frac3{R_m}+\frac4{R_m}
     +\frac{2\abs{h_m'(r)-a(r)}}r
 \le\frac{13}{R_m}.
\]
The Jacobian is continuous at $r=R_m$.  Integrating this derivative along
line segments, splitting a segment at its intersections with the sphere
$r=R_m$ if necessary, yields the stated safe bound $16/R_m$.

Finally, for $r\ge16R_m$, $a(r)<2R_m/r$.  Also
$h_m'(r)=e^{-(r-R_m)/R_m}\le2R_m/r$.  Equation
\eqref{eq:radial-jacobian-formula} then proves
\eqref{eq:rho-far-jacobian}.
\end{proof}

We can now define the compactified family and its oracle. The following estimates verify smoothness and admissibility before we turn to global PL.

Let $U\in\R^{d\times 2m}$ have orthonormal columns.  Define
\begin{equation}\label{eq:Fhat-def}
 \widehat F_{m,U}(x)
 :=F_{2m}\bigl(U^\top\rho_m(x)\bigr)
   +\frac\eta2\norm{x}^2,
 \qquad x\in\R^d,
\end{equation}
and the associated stochastic oracle
\begin{equation}\label{eq:Ghat-def}
 \widehat G_{m,q,U}(x,z)
 :=J\rho_m(x)^\top U
     g_{m,q}\bigl(U^\top\rho_m(x),z\bigr)
   +\eta x.
\end{equation}

\begin{lemma}
\label[lemma]{lem:Fhat-basic}
For all $m\ge1$, all orthonormal $U$, and every $q\in(0,1]$,
\begin{align}
 \operatorname{Lip}(\nabla\widehat F_{m,U})&\le L_b:=128,
 \label{eq:Fhat-smooth}\\
 \widehat F_{m,U}(0)-\widehat F_{m,U}^\star&<5m,
 \label{eq:Fhat-initial}\\
 \mathbb E_z[\widehat G_{m,q,U}(x,z)]&=\nabla\widehat F_{m,U}(x),
 \label{eq:Fhat-unbiased}\\
 \mathbb E_z\|\widehat G_{m,q,U}(x,z)
             -\nabla\widehat F_{m,U}(x)\|^\alpha
 &\le2q^{1-\alpha}.
 \label{eq:Fhat-alpha-moment}
\end{align}
For $q=1$, the stochastic error in the last display is identically zero.
\end{lemma}

\begin{proof}
Write $N=2m$ and separate
\[
 H_N(y):=\phi(y_1)+\sum_{i=1}^{N-1}\hat{\Theta}(y_i)\phi(y_{i+1}),
 \qquad
 F_N(y)=H_N(y)+\sum_{i=1}^N B(y_i).
\]
Each coordinate of $\nabla H_N$ has absolute value at most $4$, so
\begin{equation}\label{eq:H-gradient}
 \|\nabla H_N(y)\|\le4\sqrt N.
\end{equation}
By the scalar estimates proved above,
\begin{equation}\label{eq:F-gradient-growth}
 \|\nabla F_N(y)\|\le\|y\|+11\sqrt N.
\end{equation}
If $y=U^\top\rho_m(x)$, then $\|y\|<2R_m$.  Since
$R_m=8\sqrt2\sqrt N>11\sqrt N$,
\begin{equation}\label{eq:F-gradient-on-clipped-range}
 \|\nabla F_N(y)\|\le3R_m.
\end{equation}
Using the chain smoothness and the radial clipping estimates, for arbitrary
$x,x'$ the gradient of the composed chain obeys
\begin{align*}
 &\|J\rho_m(x)^\top U\nabla F_N(U^\top\rho_m(x))
       -J\rho_m(x')^\top U\nabla F_N(U^\top\rho_m(x'))\|\\
 &\quad\le79\|x-x'\|
       +\frac{16}{R_m}(3R_m)\|x-x'\|
 \le127\|x-x'\|.
\end{align*}
Adding the gradient $\eta x$ proves \eqref{eq:Fhat-smooth}.  At the origin,
$\widehat F_{m,U}(0)=F_{2m}(0)<5m$, whereas
$\widehat F_{m,U}^\star\ge0$, proving \eqref{eq:Fhat-initial}.
Unbiasedness follows from Lemma~\ref{lem:thinned-oracle} and the chain rule.
Finally, $U$ is an isometry on $\mathbb R^{2m}$ and
$\|J\rho_m(x)\|_{\rm op}\le1$, so
\begin{align*}
 &\mathbb E_z\|\widehat G_{m,q,U}(x,z)
             -\nabla\widehat F_{m,U}(x)\|^\alpha\\
 &\quad\le
 \mathbb E_z\|g_{m,q}(U^\top\rho_m(x),z)
             -\nabla F_{2m}(U^\top\rho_m(x))\|^\alpha
 \le2q^{1-\alpha}.
\end{align*}
The $q=1$ assertion is immediate.
\end{proof}
It remains to establish global PL, which is not implied by the zero-chain property. We first analyze a regularized chain without radial composition.

Define, for general $N$,
\begin{equation}\label{eq:Gamma-def}
 \Gamma_N(y):=F_N(y)+\frac\eta2\norm{y}^2.
\end{equation}

\begin{lemma}\label[lemma]{lem:Gamma-PL}
For every $N\ge1$ and every $y\in\R^N$,
\begin{equation}\label{eq:Gamma-PL}
 \Gamma_N(y)-\Gamma_N^\star
 \le C_\Gamma N\norm{\nabla\Gamma_N(y)}^2,
 \qquad C_\Gamma:=20000.
\end{equation}
\end{lemma}

\begin{proof}
Let $r=\norm{y}$.  From \eqref{eq:H-gradient} and
\eqref{eq:B-aux-bounds},
\begin{align}
 \Gamma_N(y)
 &\le \frac{21}{40}r^2+3N,
 \label{eq:Gamma-upper}\\
 \ip{\nabla\Gamma_N(y)}{y}
 &\ge \frac{11}{20}r^2-4\sqrt N\,r-\frac{13}{2}N.
 \label{eq:Gamma-radial-lower}
\end{align}
Indeed, $0\le H_N\le N$, $\sum_iB(y_i)\le r^2/2+2N$,
$\ip{\nabla H_N(y)}{y}\ge-4\sqrt N\,r$, and
$\sum_i y_iB'(y_i)\ge r^2/2-13N/2$.

We next isolate a box containing every critical point.  Put
\[
 \mathcal B_N:=[7/8,9/8]^N.
\]
On this box,
\begin{equation}\label{eq:box-scalar-bounds}
 \hat{\Theta}\ge\frac{27}{32},\quad
 \abs{\hat{\Theta}'}\le\frac94,\quad
 \abs{\hat{\Theta}''}\le24,\quad
 \phi''\ge\frac12,\quad
 \phi\le\frac1{128},\quad
 \abs{\phi'}\le\frac18.
\end{equation}
At points of twice differentiability, every diagonal entry of
$\nabla^2F_N$ is at least $79/64$, while every off-diagonal entry has
absolute value at most $9/32$.  To verify the diagonal bound, note that a
typical diagonal entry is
\[
 \hat{\Theta}(y_{i-1})\phi''(y_i)
 +\hat{\Theta}''(y_i)\phi(y_{i+1})+B''(y_i).
\]
At $i=1$, the incoming term is $\phi''(y_1)$; at $i=N$, the outgoing
term is omitted. If $y_i\ge1$, the last two terms sum to
$1$; if $y_i<1$, then $\hat{\Theta}''(y_i)\le-12$ and
$B''(y_i)=-2\hat{\Theta}''(y_i)$, so their sum is
$(-\hat{\Theta}''(y_i))(2-\phi(y_{i+1}))>23$.  The off-diagonal term is
$\hat{\Theta}'(y_i)\phi'(y_{i+1})$.  Gershgorin's theorem therefore gives
\begin{equation}\label{eq:box-strong-convexity}
 \nabla^2F_N\succeq\frac12 I
 \quad\text{almost everywhere on }\mathcal B_N.
\end{equation}
For line segments contained in a breakpoint hyperplane, approximate them
by generic segments and pass to the limit using gradient continuity.
Thus the Lipschitz continuity of $\nabla F_N$ and integration along
line segments imply
that $F_N$, and hence $\Gamma_N$, is $1/2$-strongly convex on
$\mathcal B_N$.

We now prove that the gradient is uniformly nonzero outside the box.  Let
$y\notin\mathcal B_N$, and let $j$ be the first index for which
$y_j\notin[7/8,9/8]$.  Set $a=1$ if $j=1$ and
$a=\hat{\Theta}(y_{j-1})$ otherwise.  Then $a\ge27/32$.  With
$s=\phi(y_{j+1})\in[0,1]$ when $j<N$ and $s=0$ when $j=N$, one has
\begin{equation}\label{eq:first-bad-coordinate}
 \nabla_j\Gamma_N(y)
 =a\phi'(y_j)+s\hat{\Theta}'(y_j)+B'(y_j)+\eta y_j.
\end{equation}
Direct substitution of the scalar formulas gives
\begin{equation}\label{eq:first-bad-signs}
 \nabla_j\Gamma_N(y)
 \begin{cases}
 \ge 1/8,&y_j>9/8,\\
 \le-9/160,&1/2<y_j<7/8,\\
 \le-3/40,&-1/2\le y_j\le1/2,\\
 \le-1/8,&y_j<-1/2.
 \end{cases}
\end{equation}
For $1/2<y_j<7/8$, put $u=1-y_j\in(1/8,1/2)$.  Since
$s\hat{\Theta}'+B'=(s-2)\hat{\Theta}'\le-\hat{\Theta}'\le0$ and
$u e^{-u^2/2}\ge(1/8)e^{-1/128}$,
\[
 a\phi'(y_j)+\eta y_j
 \le-\frac{27}{256}e^{-1/128}+\frac7{160}
 <-\frac9{160}.
\]
On $[-1/2,1/2]$, the last two chain terms vanish.  Here
$(1-y_j)e^{-(1-y_j)^2/2}\ge(1/2)e^{-1/8}$, and therefore
$a\phi'(y_j)+\eta y_j\le-(27/64)e^{-1/8}+1/40<-3/40$.  On $(-\infty,-1/2)$, split at $-1$:
for $y_j\le-1$, $(1+\eta)y_j+1/2\le-11/20$, and for
$-1<y_j<-1/2$, $a\phi'(y_j)\le-(27/32)2e^{-2}$ while
$(1+\eta)y_j+1/2\le-1/40$.  Thus all constants in
\eqref{eq:first-bad-signs} are valid.  In particular,
\begin{equation}\label{eq:gradient-away-box}
 y\notin\mathcal B_N
 \quad\Longrightarrow\quad
 \norm{\nabla\Gamma_N(y)}\ge\frac9{160}.
\end{equation}

The function $\Gamma_N$ is coercive, so it has a minimizer.  By
\eqref{eq:gradient-away-box}, every minimizer lies in $\mathcal B_N$.
Fix one minimizer $y^\star$.

If $y\in\mathcal B_N$, strong convexity on the box gives
\[
 \Gamma_N(y)-\Gamma_N(y^\star)
 \le \ip{\nabla\Gamma_N(y)}{y-y^\star}
      -\frac14\norm{y-y^\star}^2
 \le\norm{\nabla\Gamma_N(y)}^2.
\]

Suppose next that $y\notin\mathcal B_N$ and $r\le9\sqrt N$.  Since
$\Gamma_N^\star\ge0$, \eqref{eq:Gamma-upper} and
\eqref{eq:gradient-away-box} imply
\[
 \Gamma_N(y)-\Gamma_N^\star
 \le\frac{1821}{40}N
 \le20000N\norm{\nabla\Gamma_N(y)}^2.
\]

Finally, suppose $r>9\sqrt N$.  The right-hand side of
\eqref{eq:Gamma-radial-lower} is at least $r^2/40$; this follows by writing
$r=s\sqrt N$, observing that
$\frac{21}{40}s^2-4s-13/2\ge0$ for $s\ge9$.  Hence
$\norm{\nabla\Gamma_N(y)}\ge r/40$.  Also
\eqref{eq:Gamma-upper} and $N<r^2/81$ give
\[
 \Gamma_N(y)-\Gamma_N^\star
 \le\left(\frac{21}{40}+\frac3{81}\right)r^2
 \le900\norm{\nabla\Gamma_N(y)}^2.
\]
Combining the three regions proves \eqref{eq:Gamma-PL}.
\end{proof}

The auxiliary PL estimate can now be transferred through radial clipping. Inside the clipping ball the transfer is exact; outside it a radial or tangential gradient component prevents spurious near-stationarity.

\begin{lemma}
\label[lemma]{lem:Fhat-PL}
For every $m\ge1$, every $d\ge2m$, every orthonormal
$U\in\R^{d\times2m}$, and every $x\in\R^d$,
\begin{equation}\label{eq:Fhat-PL}
 \widehat F_{m,U}(x)-\widehat F_{m,U}^\star
 \le C_{\rm PL}m\norm{\nabla\widehat F_{m,U}(x)}^2,
 \qquad C_{\rm PL}:=40000.
\end{equation}
\end{lemma}

\begin{proof}
Write $N=2m$.  We first identify the minimum value.  For any $x$, let
$y=U^\top\rho_m(x)$.  Since
$\norm{y}\le\norm{\rho_m(x)}\le\norm{x}$,
\begin{equation}\label{eq:min-lower-transfer}
 \widehat F_{m,U}(x)
 \ge F_N(y)+\frac\eta2\norm{y}^2
 =\Gamma_N(y)\ge\Gamma_N^\star.
\end{equation}
Conversely, a minimizer $y^\star$ of $\Gamma_N$ satisfies
\[
 \frac\eta2\norm{y^\star}^2
 \le\Gamma_N(y^\star)\le\Gamma_N(0)=F_N(0)<5m.
\]
Thus $\norm{y^\star}^2<200m<R_m^2$.  Consequently
$\rho_m(Uy^\star)=Uy^\star$, and
$\widehat F_{m,U}(Uy^\star)=\Gamma_N(y^\star)$.  Therefore
\begin{equation}\label{eq:min-value-equality}
 \widehat F_{m,U}^\star=\Gamma_N^\star.
\end{equation}

First consider $\norm{x}\le R_m$.  Decompose
\[
 y=U^\top x,
 \qquad
 w=x-Uy,
 \qquad U^\top w=0.
\]
Since $\rho_m(x)=x$,
\begin{align}
 \widehat F_{m,U}(x)-\widehat F_{m,U}^\star
 &=\Gamma_N(y)-\Gamma_N^\star+\frac\eta2\norm{w}^2,
 \label{eq:inside-gap}\\
 \norm{\nabla\widehat F_{m,U}(x)}^2
 &=\norm{\nabla\Gamma_N(y)}^2+\eta^2\norm{w}^2.
 \label{eq:inside-gradient}
\end{align}
By \cref{lem:Gamma-PL}, $N=2m$, and $1/(2\eta)=10$, these identities imply
\eqref{eq:Fhat-PL} inside the ball.

It remains to consider $r:=\norm{x}>R_m$.  Put
\[
 e=x/r,\qquad h=h_m(r),\qquad a=h/r,
 \qquad z=U^\top e,\qquad y=hz,
\]
\[
 g=\nabla F_N(y),
 \qquad \beta=\ip{g}{z}.
\]
Using \eqref{eq:radial-jacobian-formula}, the radial and tangential parts of
the gradient are orthogonal and satisfy
\begin{equation}\label{eq:radial-tangential-gradient}
 \nabla\widehat F_{m,U}(x)
 =\bigl(h_m'(r)\beta+\eta r\bigr)e
  +a\bigl(Ug-\beta e\bigr),
\end{equation}
so
\begin{equation}\label{eq:radial-tangential-norm}
 \norm{\nabla\widehat F_{m,U}(x)}^2
 =\bigl(h_m'(r)\beta+\eta r\bigr)^2
  +a^2\bigl(\norm{g}^2-\beta^2\bigr).
\end{equation}

If $\beta\ge-\eta r/2$, then, because $0\le h_m'(r)\le1$, the radial term
in \eqref{eq:radial-tangential-gradient} has magnitude at least
$\eta r/2$.  Suppose instead that $\beta<-\eta r/2$.  Then
$\ip{g}{y}=h\beta<0$.  The analogue of
\eqref{eq:Gamma-radial-lower} without the regularizing term gives
\begin{equation}\label{eq:F-radial-lower}
 \ip{\nabla F_N(y)}{y}
 \ge\frac12\norm{y}^2-4\sqrt N\norm{y}-\frac{13}{2}N.
\end{equation}
Therefore
\[
 \norm{y}<(4+\sqrt{29})\sqrt N.
\]
Since $h\ge R_m=8\sqrt2\sqrt N$, it follows that
$\norm{z}^2<7/10$.  Hence
\[
 \beta^2\le\norm{g}^2\norm{z}^2
 \le\frac7{10}\norm{g}^2,
 \qquad
 \norm{g}^2-\beta^2\ge\frac37\beta^2.
\]
Also $a\abs{\beta}>\eta h/2\ge\eta R_m/2$.  In either case,
\begin{equation}\label{eq:outside-gradient-floor}
 \norm{\nabla\widehat F_{m,U}(x)}^2
 \ge\frac{3\eta^2R_m^2}{28}.
\end{equation}

If $R_m<r\le16R_m$, then \eqref{eq:B-aux-bounds} and $0\le H_N\le N$
give
\[
 F_N(y)\le\frac12\norm{y}^2+3N\le2R_m^2+3N.
\]
Because $\widehat F_{m,U}^\star\ge0$,
\begin{align*}
 \widehat F_{m,U}(x)-\widehat F_{m,U}^\star
 &\le2R_m^2+3N+\frac\eta2(16R_m)^2\\
 &\le31448\norm{\nabla\widehat F_{m,U}(x)}^2
 \le C_{\rm PL}m\norm{\nabla\widehat F_{m,U}(x)}^2.
\end{align*}
The numerical inequality uses $R_m^2=128N$ and
\eqref{eq:outside-gradient-floor}.

Finally, suppose $r>16R_m$.  By \eqref{eq:F-gradient-on-clipped-range},
$\norm{g}\le3R_m$, and by \eqref{eq:rho-far-jacobian},
\[
 \norm{J\rho_m(x)^\top Ug}
 \le\frac{6R_m^2}{r}
 \le\frac{\eta r}{2},
\]
where the last inequality follows from $r^2>256R_m^2$ and
$\eta=1/20$.  Thus
$\norm{\nabla\widehat F_{m,U}(x)}\ge\eta r/2=r/40$.  Moreover,
\[
 \widehat F_{m,U}(x)-\widehat F_{m,U}^\star
 \le2R_m^2+3N+\frac{r^2}{40}
 =259N+\frac{r^2}{40}
 \le53\norm{\nabla\widehat F_{m,U}(x)}^2.
\]
This completes all regions and proves \eqref{eq:Fhat-PL}.
\end{proof}

\subsection{Random rotations, scaling, and the stochastic component}\label{subsec:hd-stochastic}
This section proves the information-theoretic part of the construction.
The argument is a self-contained form of the probabilistic zero-chain and
random-rotation technique used in stochastic first-order lower bounds; see,
for example, \citet{arjevani2023lower}.

The only concentration input for the rotation argument is the following elementary spherical tail estimate.

\begin{lemma}\label[lemma]{lem:spherical-tail}
Let $v$ be uniform on the unit sphere in an $r$-dimensional Euclidean space,
where $r\ge4$, and let $a$ be a fixed unit vector.  For every $s\in(0,1)$,
\begin{equation}\label{eq:spherical-tail}
 \Pp\bigl(\abs{\ip{v}{a}}\ge s\bigr)
 \le4\exp\left(-\frac{rs^2}{16}\right).
\end{equation}
\end{lemma}

\begin{proof}
By rotational invariance, take $a=e_1$ and write
$v=Z/\norm{Z}$ with $Z\sim N(0,I_r)$.  If
$\abs{v_1}\ge s$, then either
$\abs{Z_1}\ge s\sqrt{r/2}$ or $\norm{Z}^2\le r/2$.  The Gaussian tail gives
\[
 \Pp\bigl(\abs{Z_1}\ge s\sqrt{r/2}\bigr)
 \le2e^{-rs^2/4}.
\]
For the second event, Markov's inequality applied to $e^{-\norm{Z}^2/2}$ gives
\[
 \Pp(\norm{Z}^2\le r/2)
 \le e^{r/4}\E e^{-\norm{Z}^2/2}
 =e^{r/4}2^{-r/2}
 \le e^{-r/16}.
\]
Since $s<1$, the sum is at most
$4e^{-rs^2/16}$.
\end{proof}

We next combine this concentration bound with the one-coordinate revelation rule. The lemma permits coefficients to depend on all current projections, so its proof keeps track of the enlarged projection history.

We state the lemma only in the form needed here.

\begin{lemma}
\label[lemma]{lem:rotation}
Let $D\ge M\ge5$, $q\in(0,1]$, $\delta\in(0,1/2)$, and $R\ge1$.
Suppose a measurable stochastic oracle $g:\R^D\times\mathcal Z\to\R^D$ satisfies, for
all $x$ and every history,
\begin{align}
 &\prog_0(g(x,z))\le\prog_{1/4}(x)+1\quad\text{almost surely},
 \label{eq:general-zero-chain}\\
 &\Pp\bigl(\prog_0(g(x,z))>\prog_{1/4}(x)\mid x,\text{history}\bigr)
 \le q.
 \label{eq:general-reveal}
\end{align}
Let $U\in\R^{d\times D}$ be Haar-uniform on the Stiefel manifold of
orthonormal $D$-frames, independently of the algorithm and the fresh,
independent oracle seeds $z_1,z_2,\ldots$.  An arbitrary adaptive randomized
algorithm makes queries $y_t\in\R^d$ with $\norm{y_t}\le R$ and receives
$Ug(U^\top y_t,z_t)$.

For the universal constant $C_{\rm rot,s}=4096$, if
\begin{equation}\label{eq:rotation-dimension}
 d\ge 2D+C_{\rm rot,s}R^2\frac Mq
       \log\left(\frac{16DM}{q\delta}\right),
\end{equation}
then, for every positive integer
\begin{equation}\label{eq:rotation-horizon}
 t\le\frac{M-\log(2/\delta)}{2q},
\end{equation}
\begin{equation}\label{eq:rotation-conclusion}
 \Pp\left(
   \prog_{1/4}(U^\top y_s)<M
   \text{ for every }s\le t
 \right)\ge1-\delta.
\end{equation}
The probability includes $U$, the oracle randomness, and the internal
randomness of the algorithm.
\end{lemma}

\begin{proof}
We use deferred decisions for the columns $u_1,\ldots,u_D$ of $U$.  Put
\[
 \pi_s:=\max_{i\le s}\prog_{1/4}(U^\top y_i),
 \qquad
 \gamma_s:=\max_{i\le s}\prog_0(g(U^\top y_i,z_i)),
 \qquad \gamma_0:=0.
\]
Let $\mathcal V_s$ be the event
$\prog_{1/4}(U^\top y_s)\le\gamma_{s-1}$.  Define the two stopping times
\[
 \tau_V:=\inf\{s\ge1:\mathcal V_s^c\},
 \qquad
 \tau_\gamma:=\inf\{s\ge1:\gamma_s\ge M\},
\]
with the convention $\inf\varnothing=\infty$.  If $\pi_t\ge M$, then either
$\tau_\gamma\le t$ occurs before $\tau_V$, or $\tau_V\le t$ occurs no later
than $\tau_\gamma$.  Indeed, consider the first query whose progress reaches
$M$: if no leakage has occurred, its progress is bounded by the previously
revealed support, while otherwise the leakage stopping time has already
occurred.  Hence
\begin{equation}\label{eq:rotation-race}
 \Pp(\pi_t\ge M)
 \le \Pp(\tau_\gamma\le t,\tau_\gamma<\tau_V)
    +\Pp(\tau_V\le t,\tau_V\le\tau_\gamma).
\end{equation}

We first control the progress-before-leakage event.  Let $r_{\cA}$ denote the
algorithm's random seed and use the filtration
\[
 \mathscr G_i:=\sigma(r_{\cA},U,z_1,\ldots,z_i).
\]
The query $y_i$, the event $\mathcal V_i$, and $\gamma_{i-1}$ are
$\mathscr G_{i-1}$-measurable.  Before either stopping time, the support rule
\eqref{eq:general-zero-chain} implies that
$\gamma_i-\gamma_{i-1}\in\{0,1\}$, and
\eqref{eq:general-reveal} bounds the conditional probability of the value
$1$ by $q$.  Define the stopped increments
\[
 \iota_i:=(\gamma_i-\gamma_{i-1})
 \one_{\{i<\tau_V,\,\gamma_{i-1}<M\}}.
\]
Then
\[
 \E[e^{\iota_i}\mid\mathscr G_{i-1}]
 \le1-q+qe\le e^{2q}.
\]
If $\tau_\gamma\le t$ and $\tau_\gamma<\tau_V$, then
$\sum_{i=1}^t\iota_i\ge M$.  Markov's inequality and iterated conditional
expectations therefore give
\begin{equation}\label{eq:progress-chernoff}
 \Pp(\tau_\gamma\le t,\tau_\gamma<\tau_V)
 \le \Pp\left(\sum_{i=1}^t\iota_i\ge M\right)
 \le e^{2qt-M}.
\end{equation}
Under \eqref{eq:rotation-horizon}, this is at most $\delta/2$.

It remains to control leakage before the support reaches $M$, namely the
second event in \eqref{eq:rotation-race}.  Fix a horizon $t$ satisfying
\eqref{eq:rotation-horizon}.  For each column index $j$ and query index $i$,
apply Gram--Schmidt to $y_i$ after the vectors
$u_1,\ldots,u_{j-1},y_1,\ldots,y_{i-1}$.  If the residual is nonzero, denote
its normalized direction by $w_{j,i}$; otherwise put $w_{j,i}=0$.  Thus the
nonzero vectors $w_{j,1},\ldots,w_{j,i}$ are orthonormal and orthogonal to
$u_1,\ldots,u_{j-1}$.

We record the deferred-decisions step precisely.  Let $r_{\cA}$ denote the
algorithm's random seed and define
\begin{equation}\label{eq:rotation-sigma-field}
 \mathscr U_{j-1}^{i-1}
 :=\sigma\bigl(r_{\cA},z_1,\ldots,z_{i-1},
        u_1,\ldots,u_{j-1},
        U^\top y_1,\ldots,U^\top y_{i-1}\bigr).
\end{equation}
The event $\{j>\gamma_{i-1}\}$ is measurable with respect to this
$\sigma$-field.  On that event, every past oracle response can be reconstructed
from the information in \eqref{eq:rotation-sigma-field}: the vector
$g(U^\top y_s,z_s)$ is known and has support below $j$, so
\[
 Ug(U^\top y_s,z_s)
 =\sum_{\ell<j}[g(U^\top y_s,z_s)]_\ell u_\ell.
\]
It follows recursively that the next query $y_i$, and hence $w_{j,i}$, is
fixed conditional on $\mathscr U_{j-1}^{i-1}$ and $\{j>\gamma_{i-1}\}$.

Put
\[
 \mathcal S_{j,i-1}
 :=\spanop\{u_1,\ldots,u_{j-1},y_1,\ldots,y_{i-1}\},
 \qquad P_{j,i-1}:=P_{\mathcal S_{j,i-1}^{\perp}}.
\]
Conditional on the same information, the normalized residual
\begin{equation}\label{eq:residual-column}
 \widehat u_j:=\frac{P_{j,i-1}u_j}{\norm{P_{j,i-1}u_j}}
\end{equation}
is uniform on the unit sphere of $\mathcal S_{j,i-1}^{\perp}$ whenever the
residual is nonzero. If the residual is zero, then
$\langle u_j,w_{j,i}\rangle=0$ and the tail estimate below is automatic,
so no normalized residual is needed. Here is a
complete invariance argument.  Let $Q$ be any orthogonal transformation that
fixes $\mathcal S_{j,i-1}$ pointwise.  Left multiplication $U\mapsto QU$
preserves Haar measure.  It fixes $u_1,\ldots,u_{j-1}$ and every previous
query, and it also preserves every conditioned projection because
$(QU)^\top y_s=U^\top Q^\top y_s=U^\top y_s$.  It therefore preserves both
the $\sigma$-field in \eqref{eq:rotation-sigma-field} and the event
$\{j>\gamma_{i-1}\}$.  The corresponding regular conditional law of
$P_{j,i-1}u_j$ is consequently invariant under every orthogonal map of
$\mathcal S_{j,i-1}^{\perp}$.  Its norm is fixed by the conditioned
projections and the identity $\norm{u_j}=1$; hence its direction has the
unique rotation-invariant probability law on the sphere, namely the uniform
law.  Since
$\dim(\mathcal S_{j,i-1}^{\perp})\ge d-i-j+2$, the safe lower bound
$d-i-j$ may be used below.

Consequently, whenever $j>\gamma_{i-1}$ and no previous leakage has
occurred, $w_{j,i}\in\mathcal S_{j,i-1}^{\perp}$ and
$\abs{\ip{u_j}{w_{j,i}}}\le\abs{\ip{\widehat u_j}{w_{j,i}}}$.
Thus \cref{lem:spherical-tail} gives
\begin{equation}\label{eq:one-leak-bound}
 \Pp\left(
   \abs{\ip{u_j}{w_{j,i}}}\ge\frac1{4R\sqrt t}
   \ \middle|\ \text{past}
 \right)
 \le4\exp\left(
   -\frac{d-i-j}{256R^2t}
 \right).
\end{equation}
Here conditioning on more information than the algorithm observes is
legitimate and only strengthens the deferred-decisions argument.

Suppose all the inner products in \eqref{eq:one-leak-bound} are below their
threshold for a fixed $j$ and all $i\le s$.  Since $u_j$ is orthogonal to
$u_1,\ldots,u_{j-1}$, the Gram--Schmidt expansion of $y_s$ and
Cauchy--Schwarz give
\begin{align}
 \abs{\ip{u_j}{y_s}}
 &\le
 \left(\sum_{i=1}^s\abs{\ip{u_j}{w_{j,i}}}^2\right)^{1/2}
 \left(\sum_{i=1}^s\abs{\ip{w_{j,i}}{y_s}}^2\right)^{1/2}
 \notag\\
 &<\frac{\sqrt s}{4R\sqrt t}\norm{y_s}
 \le\frac14.
 \label{eq:no-leak-cs}
\end{align}
At the stopping time $\tau_V$ on the event
$\{\tau_V\le t,\tau_V\le\tau_\gamma\}$, one has
$\gamma_{\tau_V-1}<M$, and a failure of one of the events in
\eqref{eq:one-leak-bound} must occur for some $i\le\tau_V$ and some
$j\le D$.  A union bound therefore yields
\begin{equation}\label{eq:leak-union-bound}
 \Pp(\tau_V\le t,\tau_V\le\tau_\gamma)
 \le4tD\exp\left(
   -\frac{d-t-D}{256R^2t}
 \right).
\end{equation}

Let $A=\log(16DM/(q\delta))$.  Since $t\le M/(2q)$, the
dimension requirement \eqref{eq:rotation-dimension} implies
$d-t-D\ge \tfrac12 C_{\rm rot,s}R^2(M/q)A$ and hence
\[
 \frac{d-t-D}{256R^2t}\ge16A.
\]
Substituting this into \eqref{eq:leak-union-bound}, and using
$4tD\le2DM/q$, shows that the leakage probability is at most $\delta/2$.
Together with \eqref{eq:rotation-race} and
\eqref{eq:progress-chernoff}, this proves
\eqref{eq:rotation-conclusion}.
\end{proof}

We now combine the geometry and information arguments.

\begin{proposition}\label[proposition]{prop:normalized}
There is a universal constant $c_{\rm n}>0$ such that the following holds.
Let $m\ge5$, $q\in(0,1]$, and define
\begin{equation}\label{eq:normalized-dimension}
 d_{m,q}:=
 \left\lceil
 4m+2^{22}\frac{m^2}{q}
  \log\left(\frac{128m^2}{q}\right)
 \right\rceil.
\end{equation}
Let $U$ be Haar-uniform on the Stiefel manifold of
$d_{m,q}\times2m$ orthonormal matrices.  For every adaptive randomized
algorithm using the oracle \eqref{eq:Ghat-def}, whenever
$T\le c_{\rm n}m/q$,
\begin{align}
 \Pp_{U,z,\cA}\left(
  \widehat F_{m,U}(\widehat x_T)-\widehat F_{m,U}^\star
  \ge(2-\eta)m
 \right)&\ge\frac34,
 \label{eq:normalized-hardness-probability}\\
 \E_{U,z,\cA}\bigl[
   \widehat F_{m,U}(\widehat x_T)-\widehat F_{m,U}^\star
 \bigr]&\ge m.
 \label{eq:normalized-hardness}
\end{align}
One may take $c_{\rm n}=1/16$.
\end{proposition}

\begin{proof}
Given an algorithm querying the clipped oracle at $x$, simulate it using the
strictly stronger rotated-chain oracle as follows.  Query
$y=\rho_m(x)$, which satisfies $\norm{y}<2R_m$, receive
$Ug_{m,q}(U^\top y,z)$, and return to the original algorithm
\[
 J\rho_m(x)^\top Ug_{m,q}(U^\top y,z)+\eta x.
\]
This is exactly \eqref{eq:Ghat-def}.  Thus it suffices to lower-bound the
stronger experiment.

Apply \cref{lem:rotation} with
\[
 D=2m,
 \qquad M=m,
 \qquad R=2R_m=32\sqrt m,
 \qquad \delta=\frac14.
\]
The dimension in \eqref{eq:normalized-dimension} is larger than the
requirement \eqref{eq:rotation-dimension}.  The output point may be arbitrary
and need not have been queried.  Append
$y_{T+1}=\rho_m(\widehat x_T)$ as one additional query and ignore its
response.  If $T\le m/(16q)$ and $m\ge5$, then
\[
 T+1\le\frac{m-\log8}{2q},
\]
so \cref{lem:rotation} implies, with probability at least $3/4$,
\begin{equation}\label{eq:output-low-progress}
 \prog_{1/4}\bigl(U^\top\rho_m(\widehat x_T)\bigr)<m.
\end{equation}
On this event, \eqref{eq:tail-barrier} gives
\[
 F_{2m}\bigl(U^\top\rho_m(\widehat x_T)\bigr)\ge2m.
\]
On the other hand, because $\sqrt{2m}<R_m$,
\[
 \widehat F_{m,U}^\star
 \le\widehat F_{m,U}(U\one_{2m})
 =\frac\eta2\norm{U\one_{2m}}^2
 =\eta m.
\]
The quadratic term at the output is nonnegative, so on
\eqref{eq:output-low-progress},
\[
 \widehat F_{m,U}(\widehat x_T)-\widehat F_{m,U}^\star
 \ge(2-\eta)m.
\]
The event \eqref{eq:output-low-progress} has probability at least $3/4$,
so the preceding inequality proves
\eqref{eq:normalized-hardness-probability}.  The gap is always nonnegative;
hence its expectation is at least
$\frac34(2-\eta)m>m$, proving \eqref{eq:normalized-hardness}.
\end{proof}

\begin{remark}[A common fixed hard rotation]\label[remark]{rem:fixed-U}
Averaging over $U$ is only an existence argument. For a fixed randomized
algorithm define
\[
 p_{\cA}(U)=\mathbb P_{z,\cA}\!\left(
 \widehat F_{m,U}(\widehat x_T)-\widehat F_{m,U}^\star
 \ge(2-\eta)m\right).
\]
Proposition~\ref{prop:normalized} gives
$\mathbb E_Up_{\cA}(U)\ge3/4$. Hence there is a fixed $U$ with
$p_{\cA}(U)\ge3/4$. For this same $U$, nonnegativity of the gap implies
\[
 \mathbb E_{z,\cA}[\widehat F_{m,U}(\widehat x_T)-\widehat F_{m,U}^\star]
 \ge\tfrac34(2-\eta)m>m.
\]
Thus one rotation is simultaneously hard in probability and expectation.
Selecting separate maximizers of the two criteria would not justify this
conclusion. The instance and oracle are fixed before the algorithm is run.
\end{remark}

The normalized instance is now ready to be calibrated. Its smoothness and
PL constants determine the chain width through
\begin{equation}\label{eq:A-constant}
 A_0:=C_{\rm PL}L_b=40000\cdot128,
 \qquad
 \kappa_0:=12A_0.
\end{equation}
Assume that $\kappa=L/\mu\ge\kappa_0$ and set
\begin{equation}\label{eq:m-choice}
 m_{\rm s}:=\left\lfloor\frac{\kappa}{2A_0}\right\rfloor,\qquad m:=m_{\rm s}.
\end{equation}
Then $m\ge5$ and
\begin{equation}\label{eq:m-comparable-kappa}
 \frac{\kappa}{3A_0}\le m\le\frac{\kappa}{2A_0}.
\end{equation}
For $0<\epsilon\le\Delta_0/10$, choose the objective scale to make the
unrevealed tail exceed the target accuracy, and the spatial scale to match $L$:

\begin{equation}\label{eq:lambda-s-choice}
 \lambda:=\frac{2\epsilon}{m},
 \qquad
 s^2:=\frac{L}{\lambda L_b}.
\end{equation}
For a fixed rotation $U$, define
\begin{equation}\label{eq:scaled-function}
 f_U(x):=\lambda\widehat F_{m,U}(s x),
\end{equation}
and scale the oracle consistently:
\begin{equation}\label{eq:scaled-oracle}
 G_U(x,z):=\lambda s\,\widehat G_{m,q,U}(s x,z).
\end{equation}

First we verify the geometric requirements after scaling.

By \eqref{eq:Fhat-smooth},
\begin{equation}\label{eq:scaled-smoothness}
 \operatorname{Lip}(\nabla f_U)\le\lambda s^2L_b=L.
\end{equation}
By the global PL estimate for $\widehat F_{m,U}$,
\begin{align*}
 f_U(x)-f_U^\star
 &\le\lambda C_{\rm PL}m
       \|\nabla\widehat F_{m,U}(sx)\|^2\\
 &=\frac{C_{\rm PL}m}{\lambda s^2}\|\nabla f_U(x)\|^2
 =\frac{C_{\rm PL}mL_b}{L}\|\nabla f_U(x)\|^2
 \le\frac1{2\mu}\|\nabla f_U(x)\|^2.
\end{align*}
Thus $f_U$ is globally $\mu$-PL.  Moreover,
\begin{equation}\label{eq:scaled-initial-gap}
 f_U(0)-f_U^\star<5\lambda m=10\epsilon.
\end{equation}

Next we calibrate the revelation probability to the announced noise moment.

For $q<1$, Lemma~\ref{lem:Fhat-basic} gives
\begin{align}
 &\mathbb E\|G_U(x,z)-\nabla f_U(x)\|^\alpha\notag\\
 &\quad\le2(\lambda s)^\alpha q^{1-\alpha}
 =2\left(\frac{\lambda L}{L_b}\right)^{\alpha/2}
 q^{1-\alpha}.
 \label{eq:scaled-alpha-moment}
\end{align}
Define
\begin{equation}\label{eq:q-choice}
 q_\star:=
 \left[\frac{2(\lambda L/L_b)^{\alpha/2}}{\sigma^\alpha}
 \right]^{1/(\alpha-1)}
 =2^{1/(\alpha-1)}
 \left(\frac{\lambda L}{L_b\sigma^2}\right)^{\thetaalpha},
 \qquad
 q:=\min\{1,q_\star\}=q_{\rm s}\quad(\sigma>0).
\end{equation}
For $\sigma=0$, set $q_{\rm s}=q=1$ directly and omit $q_\star$; the oracle is exact.
For $\sigma>0$, if $q_\star\le1$, \eqref{eq:scaled-alpha-moment} is at most
$\sigma^\alpha$.  If $q_\star>1$, then $q=1$ and the oracle is exact, so
its centered moment is zero.

The normalized lower bound transfers exactly through scaling: the objective
gap is multiplied by $\lambda$, and a scaled algorithm can be simulated on
the normalized instance.  Thus, for $T\le c_{\rm n}m/q$, there is one fixed
rotation $U$ such that
\begin{align}
 \mathbb P\bigl(f_U(\widehat x_T)-f_U^\star
   \ge2(2-\eta)\epsilon\bigr)&\ge\frac34,
 \label{eq:scaled-heavy-probability}\\
 \mathbb E[f_U(\widehat x_T)-f_U^\star]&\ge2\epsilon.
 \label{eq:scaled-heavy-expectation}
\end{align}

Suppose first that $\sigma>0$ and $q_\star\le1$.

Using \eqref{eq:lambda-s-choice},
\begin{align}
 \frac mq
 &=2^{-1/(\alpha-1)}m
   \left(\frac{L_b\sigma^2}{\lambda L}\right)^{\thetaalpha}\notag\\
 &=2^{-1/(\alpha-1)}m
   \left(\frac{mL_b\sigma^2}{2\epsilon L}\right)^{\thetaalpha}.
 \label{eq:mq-heavy-exact}
\end{align}
Since $mL_b/L\ge1/(3C_{\rm PL}\mu)$ and
$m\ge\kappa/(3A_0)$,
\begin{equation}\label{eq:mq-heavy-lower}
 \frac mq\ge
 \frac{2^{-1/(\alpha-1)}}{3A_0}
 \left(\frac1{6C_{\rm PL}}\right)^{\thetaalpha}
 \kappa\left(\frac{\sigma^2}{\mu\epsilon}\right)^{\thetaalpha}.
\end{equation}

For the remaining positive-noise case $q_\star>1$, the calibrated oracle is exact but still gives the required barrier.

The inequality $q_\star>1$ implies
\[
 \sigma^2<2^{2/\alpha}\frac{\lambda L}{L_b}
 =2^{1+2/\alpha}\frac{\epsilon L}{mL_b}.
\]
Consequently,
\begin{equation}\label{eq:exact-case-heavy-ratio}
 \left(\frac{\sigma^2}{\mu\epsilon}\right)^{\thetaalpha}
 <\left(2^{1+2/\alpha}\frac{\kappa}{mL_b}
\right)^{\thetaalpha}.
\end{equation}
Because $m=\lfloor\kappa/(2A_0)\rfloor$ and $m\ge5$,
$\kappa<(12/5)A_0m$.  Hence the right-hand side of
\eqref{eq:exact-case-heavy-ratio} is bounded by a constant depending only on
$\alpha$.  Also $\kappa\le(12/5)A_0m$.  Therefore
\begin{equation}\label{eq:exact-chain-dominates-heavy}
 \kappa\left(\frac{\sigma^2}{\mu\epsilon}\right)^{\thetaalpha}
 \le C_\alpha m.
\end{equation}
The order-$m$ exact-gradient barrier thus dominates the desired lower bound.

Combining the two cases, there is $c_{{\rm large},\alpha}>0$ such that, for
$\kappa\ge\kappa_0$, every
\begin{equation}\label{eq:large-kappa-threshold}
 T\le c_{{\rm large},\alpha}\kappa
 \left(\frac{\sigma^2}{\mu\epsilon}\right)^{\thetaalpha}
\end{equation}
leaves expected error at least $2\epsilon$ and gap at least $2\epsilon$ with
probability at least $3/4$ on some fixed valid instance. The initial-gap
condition follows from \eqref{eq:scaled-initial-gap}, since
$10\epsilon\le\Delta_0$ in the stated accuracy range.

With the width and revelation probability now fixed, the rotation argument
requires the ambient dimension
\begin{equation}\label{eq:hd-ds}
 d_{\rm s}:=d_{m_{\rm s},q_{\rm s}}
 =\left\lceil4m_{\rm s}+2^{22}\frac{m_{\rm s}^2}{q_{\rm s}}
       \log\left(\frac{128m_{\rm s}^2}{q_{\rm s}}\right)\right\rceil.
\end{equation}
This is precisely \eqref{eq:normalized-dimension} at the calibrated parameters.
It completes the stochastic construction, including its finite hard dimension.

\subsection{Geometrically weighted exact chain and the noiseless component}\label{subsec:hd-noiseless}
The stochastic construction is deliberately tuned to the noise scale.
Setting $\sigma=0$ in its final formula only produces the trivial value zero;
it does not by itself prove the exact-gradient transient
$\Omega(\kappa\log(\Delta_0/\epsilon))$. We now build a separate exact-gradient
hard family. Since an exact oracle belongs to
$\cO_{\alpha,\sigma}(f)$ for every $\sigma\ge0$, this family can be
combined with the stochastic family by taking the maximum of the two minimax
obstructions.

The new ingredient is a geometrically weighted robust zero-chain. It has
$K$ layers, each of width $m=\Theta(\kappa)$. The layer amplitudes decay
geometrically, so the last unresolved layer has value comparable to the
prescribed accuracy. Crucially, the total squared weight after any first bad
coordinate is only $O(m)$ times that coordinate's squared weight. This keeps
the global PL coefficient at $O(m)$ although the total chain length is $mK$.

We now specify the amplitudes and the weighted chain explicitly.

Fix
\begin{equation}\label{eq:gamma-weighted}
 \gamma:=\frac78,
 \qquad r_\gamma:=\gamma^2=\frac{49}{64}.
\end{equation}
Let $m,K\ge1$ be integers, put $N:=mK$, and define
\begin{equation}\label{eq:weighted-amplitudes}
 a_i:=\gamma^{\lfloor(i-1)/m\rfloor},
 \qquad u_i:=\frac{x_i}{a_i},
 \qquad i=1,\ldots,N.
\end{equation}
Thus each layer has $m$ equal amplitudes, and successive layers differ by the
factor $\gamma$. Using the scalar functions from
\eqref{eq:Theta-def}--\eqref{eq:B-def}, define
\begin{equation}\label{eq:weighted-chain-def}
 \mathsf W_{m,K}(x)
 :=a_1^2\phi(u_1)
 +\sum_{i=1}^{N-1}a_{i+1}^2\hat{\Theta}(u_i)\phi(u_{i+1})
 +\sum_{i=1}^N a_i^2B(u_i).
\end{equation}
For the nonuniform threshold vector
$\boldsymbol a=(a_1,\ldots,a_N)$, put
\begin{equation}\label{eq:weighted-progress}
 \prog_{\boldsymbol a}(x)
 :=\max\left(\left\{i:\abs{x_i}>\frac{a_i}{4}\right\}\cup\{0\}\right).
\end{equation}
Finally, define the tail squared-weight mass
\begin{equation}\label{eq:tail-weight-mass}
 S_j:=\sum_{i=j}^N a_i^2,
 \qquad j=1,\ldots,N.
\end{equation}

\begin{lemma}[Geometric tail mass]\label[lemma]{lem:geometric-tail-mass}
For every $j$,
\begin{equation}\label{eq:tail-mass-bound}
 S_j\le \frac{m}{1-\gamma^2}a_j^2
 =\frac{64}{15}ma_j^2<5ma_j^2.
\end{equation}
In particular, $S_1<5m$.
\end{lemma}

\begin{proof}
Let $j$ belong to layer $\ell$, so $a_j=\gamma^\ell$. There are at most
$m$ remaining entries in that layer and exactly $m$ entries in each later
layer. Therefore
\[
 S_j\le m\gamma^{2\ell}\sum_{q=0}^{\infty}\gamma^{2q}
 =\frac{m}{1-\gamma^2}a_j^2.
\]
Since $1-\gamma^2=15/64$, the numerical statements follow.
\end{proof}

The tail-mass estimate is what prevents the PL coefficient from growing
with the number of layers. We next verify all properties of the weighted
chain together, so that the information barrier and its geometry refer to
the same function.

\begin{lemma}
\label[lemma]{lem:weighted-chain}
There are universal constants
\begin{equation}\label{eq:weighted-chain-constants}
 C_{\rm init}:=11,
 \qquad L_{\rm W}:=79,
 \qquad C_{\rm W}:=10^5,
\end{equation}
such that, for every $m,K\ge1$:
\begin{enumerate}[label=(\roman*)]
\item $\mathsf W_{m,K}\ge0$, $\mathsf W_{m,K}(\boldsymbol a)=0$, and
\begin{equation}\label{eq:weighted-initial}
 \mathsf W_{m,K}(0)\le C_{\rm init}m.
\end{equation}
\item $\mathsf W_{m,K}$ is globally $L_{\rm W}$-smooth.
\item If $k=\prog_{\boldsymbol a}(x)$, then
\begin{equation}\label{eq:weighted-zero-chain}
 \nabla_j\mathsf W_{m,K}(x)=0\qquad(j\ge k+2),
\end{equation}
and, when $k<N$,
\begin{equation}\label{eq:weighted-next-gradient}
 \abs{\nabla_{k+1}\mathsf W_{m,K}(x)}\le a_{k+1}.
\end{equation}
\item If $\prog_{\boldsymbol a}(x)\le N-m$, then
\begin{equation}\label{eq:weighted-last-layer-barrier}
 \mathsf W_{m,K}(x)\ge2m\gamma^{2(K-1)}.
\end{equation}
\item The global gradient-dominance inequality
\begin{equation}\label{eq:weighted-global-PL}
 \mathsf W_{m,K}(x)
 \le C_{\rm W}m\norm{\nabla\mathsf W_{m,K}(x)}^2
 \qquad\text{for every }x\in\mathbb R^N
\end{equation}
holds. In particular, its coefficient depends on the layer width $m$ but
not on the number of layers $K$.
\end{enumerate}
\end{lemma}

\begin{proof}
Nonnegativity and $\mathsf W_{m,K}(\boldsymbol a)=0$ follow exactly as for
the unweighted chain. At the origin all coupling terms vanish, so
\[
 \mathsf W_{m,K}(0)=\phi(0)+2S_1<1+10m\le11m
\]
for $m\ge1$, proving (i).

For $i=1,\ldots,N$, define
\[
 A_i:=\begin{cases}1,&i=1,\\ \hat{\Theta}(u_{i-1}),&i\ge2,\end{cases}
 \qquad
 r_i:=\frac{a_{i+1}}{a_i}\in\{1,\gamma\}\quad(i<N).
\]
A direct differentiation gives
\begin{equation}\label{eq:weighted-gradient-coordinate}
 \nabla_i\mathsf W_{m,K}(x)
 =a_i\left[A_i\phi'(u_i)+B'(u_i)
 +\mathbf 1_{\{i<N\}}r_i^2\hat{\Theta}'(u_i)\phi(u_{i+1})\right].
\end{equation}
The dependence of this coordinate on $x_{i-1},x_i,x_{i+1}$ has respective
Lipschitz coefficients at most $3,73,3$. Omit nonexistent neighbors and,
when $i=N$, the entire outgoing term. All adjacent amplitude ratios are
at most one, and the scalar estimates are those of
Lemma~\ref{lem:scalar}. The same tridiagonal row/column-sum argument used in
Lemma~\ref{lem:base-chain} therefore proves global $79$-smoothness.

Let $k=\prog_{\boldsymbol a}(x)$. If $j\ge k+2$, then
$\abs{u_{j-1}},\abs{u_j},\abs{u_{j+1}}\le1/4$; hence every term in
\eqref{eq:weighted-gradient-coordinate} vanishes. For $j=k+1$, only the
incoming term can remain, and its magnitude is at most $a_{k+1}$. This
proves (iii). If the progress is at most $N-m$, the last $m$ normalized
coordinates lie in $[-1/4,1/4]$, where $B=2$, and their common amplitude is
$\gamma^{K-1}$. This proves (iv).

It remains to establish the global PL inequality. We divide the proof into
four steps.

\emph{Step 1: local strong convexity.}

Let
\begin{equation}\label{eq:weighted-box}
 \mathcal B_{\boldsymbol a}:=
 \prod_{i=1}^N\left[\frac78a_i,\frac98a_i\right].
\end{equation}
Using the explicit formulas for the scalar derivatives (and
$e^{-x}\ge1-x$), on $[7/8,9/8]$,
\begin{equation}\label{eq:weighted-box-scalar}
 \hat{\Theta}\ge\frac{27}{32},\quad
 \abs{\hat{\Theta}'}\le\frac94,\quad
 \phi\le\frac1{128},\quad
 \abs{\phi'}\le\frac18,\quad
 \phi''\ge\frac{31}{32}.
\end{equation}
At points of twice differentiability, the diagonal Hessian entry is
\[
 A_i\phi''(u_i)+B''(u_i)
 +\mathbf 1_{\{i<N\}}r_i^2\hat{\Theta}''(u_i)\phi(u_{i+1}),
\]
and the $(i,i+1)$ entry is
$r_i\hat{\Theta}'(u_i)\phi'(u_{i+1})$. If $u_i\ge1$, the diagonal is at least
$1+(27/32)(31/32)>7/4$. If $u_i<1$, then
$\hat{\Theta}''(u_i)\le-12$ and $B''(u_i)=-2\hat{\Theta}''(u_i)$, so the last two terms
are
\[
 (-\hat{\Theta}''(u_i))\bigl(2-r_i^2\phi(u_{i+1})\bigr)>23.
\]
Every off-diagonal entry has magnitude at most $9/32$, and each row has at
most two such entries. Gershgorin's theorem gives
\begin{equation}\label{eq:weighted-box-strong-convex}
 \nabla^2\mathsf W_{m,K}\succeq I
 \quad\text{a.e. on }\mathcal B_{\boldsymbol a}.
\end{equation}
Since the gradient is Lipschitz, integration along line segments shows that
$\mathsf W_{m,K}$ is $1$-strongly convex on this box. Its zero
$\boldsymbol a$ lies in the box. Hence, for $x\in\mathcal B_{\boldsymbol a}$,
\begin{equation}\label{eq:weighted-PL-in-box}
 \mathsf W_{m,K}(x)
 \le\frac12\norm{\nabla\mathsf W_{m,K}(x)}^2.
\end{equation}

\emph{Step 2: the first bad coordinate.}

Let $x\notin\mathcal B_{\boldsymbol a}$ and let $j$ be the first coordinate
with $u_j\notin[7/8,9/8]$. Then $A_j\ge27/32$. Put
$c_j=r_j^2\phi(u_{j+1})\in[0,1]$ if $j<N$, and $c_j=0$ if $j=N$.
Dividing \eqref{eq:weighted-gradient-coordinate} by $a_j$ gives
\begin{equation}\label{eq:bad-coordinate-normalized}
 D_j:=\frac{\nabla_j\mathsf W_{m,K}(x)}{a_j}
 =A_j\phi'(u_j)+B'(u_j)+c_j\hat{\Theta}'(u_j).
\end{equation}
We claim
\begin{equation}\label{eq:bad-coordinate-lower}
 \abs{D_j}\ge\frac3{32}.
\end{equation}
If $u_j>9/8$, then $\hat{\Theta}'(u_j)=0$, $B'(u_j)=u_j-1\ge1/8$,
and $\phi'(u_j)\ge0$. If $1/2<u_j<7/8$, then
$B'+c_j\hat{\Theta}'=(c_j-2)\hat{\Theta}'\le0$, while
\[
 -A_j\phi'(u_j)
 \ge\frac{27}{32}\frac18e^{-1/128}>\frac3{32}.
\]
If $-1/2\le u_j\le1/2$, then $B'=\hat{\Theta}'=0$ and
$-\phi'(u_j)=(1-u_j)e^{-(1-u_j)^2/2}\ge3/8$, so the same conclusion holds.
Finally, if $u_j\le-1/2$, then either $u_j\le-1$, in which case
$B'(u_j)\le-1/2$, or $-1<u_j\le-1/2$, in which case
$-\phi'(u_j)\ge2e^{-2}>1/4$; all remaining terms are nonpositive. This
proves \eqref{eq:bad-coordinate-lower}, and therefore
\begin{equation}\label{eq:bad-coordinate-scaled}
 \abs{\nabla_j\mathsf W_{m,K}(x)}\ge\frac3{32}a_j.
\end{equation}

\emph{Step 3: prefix and tail control.}

When $j\ge2$, let $P_j(\cdot;x_j)$ consist of all terms in
\eqref{eq:weighted-chain-def} involving coordinates $1,\ldots,j-1$, including
the boundary coupling
$a_j^2\hat{\Theta}(u_{j-1})\phi(u_j)$. For $j=1$, put $P_1=0$.
For fixed $x_j$, this prefix is $1$-strongly convex on
$\prod_{i<j}[7a_i/8,9a_i/8]$. The preceding Hessian proof still applies:
for the final prefix coordinate, $\phi(u_j)\le1$ gives
$B''+r^2\hat{\Theta}''\phi(u_j)\ge-\hat{\Theta}''\ge12$ whenever that coordinate is
below one. Moreover, on every lower face the corresponding partial
derivative is negative, and on every upper face it is positive, by the same
calculation as in \eqref{eq:bad-coordinate-normalized}. Thus its constrained
minimizer is interior. Strong convexity then gives
\begin{equation}\label{eq:prefix-bound}
 P_j(x_{<j};x_j)
 \le P_j(\boldsymbol a_{<j};x_j)
 +\frac12\norm{\nabla_{<j}\mathsf W_{m,K}(x)}^2
 \le a_j^2+\frac12\norm{\nabla_{<j}\mathsf W_{m,K}(x)}^2.
\end{equation}

Define the tail function, with $A_j=1$ for $j=1$ and
$A_j=\hat{\Theta}(u_{j-1})$ otherwise, by
\begin{align}\label{eq:tail-function}
 Q_j(x_{j:N})
 :={}&A_ja_j^2\phi(u_j)
 +\sum_{i=j}^{N-1}a_{i+1}^2\hat{\Theta}(u_i)\phi(u_{i+1})
 +\sum_{i=j}^N a_i^2B(u_i).
\end{align}
Then
\begin{equation}\label{eq:tail-gradient-equality}
 \nabla Q_j(x_{j:N})=\nabla_{j:N}\mathsf W_{m,K}(x).
\end{equation}
Let $r=\norm{x_{j:N}}$. The non-$B$ part of $Q_j$ is at most $S_j$ and
its gradient norm is at most $4\sqrt{S_j}$. Lemma~\ref{lem:scalar} therefore
gives
\begin{align}
 Q_j(x_{j:N})&\le\frac12r^2+3S_j,
 \label{eq:tail-value-radial}\\
 \ip{\nabla Q_j(x_{j:N})}{x_{j:N}}
 &\ge\frac12r^2-4\sqrt{S_j}\,r-\frac{13}{2}S_j.
 \label{eq:tail-gradient-radial}
\end{align}
If $r\ge18\sqrt{S_j}$, the right-hand side of
\eqref{eq:tail-gradient-radial} is at least $r^2/4$, so
$\norm{\nabla Q_j}\ge r/4$, and \eqref{eq:tail-value-radial} gives
$Q_j\le9\norm{\nabla Q_j}^2$. If $r<18\sqrt{S_j}$, then
$Q_j<165S_j$. Hence, in all cases,
\begin{equation}\label{eq:tail-combined-bound}
 Q_j(x_{j:N})
 \le9\norm{\nabla_{j:N}\mathsf W_{m,K}(x)}^2+165S_j.
\end{equation}

\emph{Step 4: global assembly.}

For $j\ge2$, the boundary coupling into $j$ is counted in both $P_j$ and
$Q_j$; it is nonnegative. Thus
$\mathsf W_{m,K}\le P_j+Q_j$ (with equality when $j=1$). Combining
\eqref{eq:prefix-bound}, \eqref{eq:tail-combined-bound},
Lemma~\ref{lem:geometric-tail-mass}, and
\eqref{eq:bad-coordinate-scaled},
\begin{align*}
 \mathsf W_{m,K}(x)
 &\le\frac12\norm{\nabla_{<j}\mathsf W_{m,K}(x)}^2
 +9\norm{\nabla_{j:N}\mathsf W_{m,K}(x)}^2
 +a_j^2+165S_j\\
 &\le\frac12\norm{\nabla_{<j}\mathsf W_{m,K}(x)}^2
 +9\norm{\nabla_{j:N}\mathsf W_{m,K}(x)}^2
 +(1+825m)\frac{1024}{9}\abs{\nabla_j\mathsf W_{m,K}(x)}^2\\
 &\le10^5m\norm{\nabla\mathsf W_{m,K}(x)}^2.
\end{align*}
Together with \eqref{eq:weighted-PL-in-box}, this proves (v) and the lemma.
\end{proof}

The weighted chain is already globally PL. The next construction bounds all effective queries without lifting its zero minimum.

The random-rotation argument requires all effective queries to have bounded
norm. The quadratic guard used in the stochastic construction is unsuitable
here because it would assign order-$m$ value to the geometrically weighted
minimizer and overwhelm the final layer. We instead use a radial guard that
vanishes on a ball containing the entire minimizer.

Set
\begin{equation}\label{eq:weighted-R-eta}
 R:=256\sqrt m,
 \qquad \eta_{\rm g}:=16.
\end{equation}
Define $h_R$ and $\rho_R$ by
\begin{equation}\label{eq:weighted-radial-map}
 h_R(r):=\begin{cases}
 r,&0\le r\le R,\\
 R+R(1-e^{-(r-R)/R}),&r>R,
 \end{cases}
 \qquad
 \rho_R(x):=\begin{cases}
 h_R(\norm{x})x/\norm{x},&x\ne0,\\
 0,&x=0.
 \end{cases}
\end{equation}
The proof of Lemma~\ref{lem:radial}, with $R_m$ replaced by $R$, gives
\begin{equation}\label{eq:weighted-radial-properties}
 \norm{\rho_R(x)}<2R,
 \qquad \norm{J\rho_R(x)}_{\rm op}\le1,
 \qquad \Lip(J\rho_R)\le\frac{16}{R}.
\end{equation}
Let
\begin{equation}\label{eq:radial-guard}
 Q_R(x):=\frac{\eta_{\rm g}}2\left(\norm{x}-\frac R2\right)_+^2.
\end{equation}
This is $\eta_{\rm g}$-smooth because it equals
$\eta_{\rm g}\operatorname{dist}(x,\overline B(0,R/2))^2/2$ and the gradient
of half the squared distance to a closed convex set is $1$-Lipschitz.

For $d\ge N$ and an orthonormal frame $U\in\mathbb R^{d\times N}$, define
\begin{equation}\label{eq:weighted-compactified-function}
 \widehat{\mathsf W}_{m,K,U}(x)
 :=\mathsf W_{m,K}(U^\top\rho_R(x))+Q_R(x).
\end{equation}

\begin{lemma}
\label[lemma]{lem:weighted-compactification}
Let
\begin{equation}\label{eq:weighted-compact-constants}
 L_{\rm det}:=144,
 \qquad C_{\rm det}:=2C_{\rm W}=2\cdot10^5.
\end{equation}
Then, for every $m,K,d,U$ as above,
\begin{align}
 \Lip(\nabla\widehat{\mathsf W}_{m,K,U})&\le L_{\rm det},
 \label{eq:weighted-compact-smooth}\\
 \widehat{\mathsf W}_{m,K,U}(x)
 &\le C_{\rm det}m
 \norm{\nabla\widehat{\mathsf W}_{m,K,U}(x)}^2
 \qquad\text{for every }x\in\mathbb R^d.
 \label{eq:weighted-compact-PL}
\end{align}
Moreover,
\begin{equation}\label{eq:weighted-compact-minimum}
 \widehat{\mathsf W}_{m,K,U}^\star=0,
 \qquad
 \widehat{\mathsf W}_{m,K,U}(0)\le C_{\rm init}m.
\end{equation}
\end{lemma}

\begin{proof}
By Lemma~\ref{lem:scalar} and the coordinate formula
\eqref{eq:weighted-gradient-coordinate},
\begin{equation}\label{eq:weighted-gradient-growth}
 \norm{\nabla\mathsf W_{m,K}(y)}
 \le\norm{y}+11\sqrt{S_1}.
\end{equation}
Indeed, the non-$B$ part has gradient norm at most $4\sqrt{S_1}$, while
$\abs{B'(u_i)-u_i}\le7$. On the range $\norm{y}<2R$ this is at most $3R$.
The chain-rule estimate, Lemma~\ref{lem:weighted-chain}, and
\eqref{eq:weighted-radial-properties} therefore give
\[
 \Lip\bigl(\nabla[\mathsf W_{m,K}(U^\top\rho_R(\cdot))]\bigr)
 \le79+\frac{16}{R}(3R)=127.
\]
Adding the $16$-smooth guard proves \eqref{eq:weighted-compact-smooth}.

Since $\norm{\boldsymbol a}^2=S_1<5m<R^2/4$, the point
$U\boldsymbol a$ lies inside the zero region of the guard, where $\rho_R$ is
the identity. Both summands in \eqref{eq:weighted-compactified-function}
are nonnegative and vanish there. This proves
\eqref{eq:weighted-compact-minimum}; the initial bound follows from
Lemma~\ref{lem:weighted-chain}.

We prove \eqref{eq:weighted-compact-PL} in three radial regions. If
$\norm{x}\le R/2$, then $\rho_R(x)=x$ and $Q_R(x)=0$. Hence
\[
 \widehat{\mathsf W}(x)=\mathsf W(U^\top x),
 \qquad
 \norm{\nabla\widehat{\mathsf W}(x)}
 =\norm{\nabla\mathsf W(U^\top x)},
\]
and Lemma~\ref{lem:weighted-chain} applies.

Suppose $R/2<r:=\norm{x}\le R$. Put
\[
 y=U^\top x,
 \qquad g=\nabla\mathsf W_{m,K}(y),
 \qquad e=x/r,
 \qquad b=\eta_{\rm g}(r-R/2).
\]
Then
\begin{equation}\label{eq:annulus-gradient}
 \nabla\widehat{\mathsf W}(x)=Ug+be.
\end{equation}
The weighted analogue of \eqref{eq:F-radial-lower}, obtained from
\eqref{eq:B-aux-bounds}, is
\begin{equation}\label{eq:weighted-radial-inner}
 \ip{\nabla\mathsf W(y)}{y}
 \ge\frac12\norm{y}^2-4\sqrt{S_1}\norm{y}-\frac{13}{2}S_1.
\end{equation}
Because $R/4\ge18\sqrt{S_1}$, if $\norm{y}\ge R/4$ then
$\ip{g}{y}\ge0$, and \eqref{eq:annulus-gradient} gives
$\norm{\nabla\widehat{\mathsf W}}^2\ge\norm{g}^2+b^2$. If
$\norm{y}<R/4$, set $z=U^\top e$. Since $r>R/2$, $\norm{z}<1/2$, and the
orthogonal decomposition of $e$ gives
\begin{align*}
 \norm{Ug+be}^2
 &=\norm{g+bz}^2+b^2(1-\norm{z}^2)\\
 &\ge\frac12\norm{g}^2+b^2(1-2\norm{z}^2)
 \ge\frac12(\norm{g}^2+b^2).
\end{align*}
In both cases,
\begin{equation}\label{eq:annulus-gradient-lower}
 \norm{\nabla\widehat{\mathsf W}(x)}^2
 \ge\frac12(\norm{g}^2+b^2).
\end{equation}
Now $\mathsf W(y)\le C_{\rm W}m\norm{g}^2$ and
$Q_R(x)=b^2/(2\eta_{\rm g})$, so
\eqref{eq:weighted-compact-PL} follows in the annulus with coefficient
$2C_{\rm W}m$.

Finally, let $r>R$. On the clipped range,
\eqref{eq:weighted-gradient-growth} gives
$\norm{\nabla\mathsf W(U^\top\rho_R(x))}\le3R$, and
$\norm{J\rho_R}\le1$. The guard gradient has magnitude
$b=\eta_{\rm g}(r-R/2)\ge8R$. Therefore
\begin{equation}\label{eq:outer-guard-gradient}
 \norm{\nabla\widehat{\mathsf W}(x)}\ge b-3R
 \ge5R,
 \qquad
 \norm{\nabla\widehat{\mathsf W}(x)}\ge\frac58b.
\end{equation}
Also, by $B(t)\le t^2/2+2$ and nonnegativity of the coupling terms,
\begin{equation}\label{eq:weighted-value-clipped-range}
 \mathsf W(U^\top\rho_R(x))
 \le\frac12\norm{U^\top\rho_R(x)}^2+3S_1<3R^2.
\end{equation}
Equations \eqref{eq:outer-guard-gradient} and
\eqref{eq:weighted-value-clipped-range} control both this term and
$Q_R=b^2/(2\eta_{\rm g})$ by a universal multiple of the squared gradient,
which is smaller than $C_{\rm det}m$ for $m\ge1$. This completes the proof.
\end{proof}

The decreasing amplitudes require a rotation lemma with nonuniform activation thresholds. The proof below adapts the earlier projection-history argument to the smallest threshold.

\begin{lemma}
\label[lemma]{lem:nonuniform-rotation}
Let $D\ge M\ge2$, $R_0\ge1$, $\delta\in(0,1/2)$, and
$0<\underline a\le a_j\le1$ for $j=1,\ldots,D$, and define
$\prog_{\boldsymbol a}$ as in \eqref{eq:weighted-progress}. Suppose a
measurable deterministic map $h:\mathbb R^D\to\mathbb R^D$ satisfies
\begin{equation}\label{eq:nonuniform-support-rule}
 \prog_0(h(z))\le\prog_{\boldsymbol a}(z)+1
 \qquad\text{for every }z.
\end{equation}
Let $U\in\mathbb R^{d\times D}$ be Haar-uniform on the Stiefel manifold and
independent of an arbitrary randomized adaptive algorithm. The algorithm
makes an integer number $1\le t<M$ of queries $y_i$ satisfying $\norm{y_i}\le R_0$ and receives
$Uh(U^\top y_i)$. For the universal constant $C_{\rm rot,d}=1024$, if
\begin{equation}\label{eq:nonuniform-rotation-dimension}
 d\ge2D+C_{\rm rot,d}\frac{R_0^2t}{\underline a^2}
 \log\left(\frac{8Dt}{\delta}\right),
\end{equation}
then
\begin{equation}\label{eq:nonuniform-rotation-conclusion}
 \mathbb P\left(
 \prog_{\boldsymbol a}(U^\top y_i)<M
 \text{ for all }i\le t\right)\ge1-\delta.
\end{equation}
The probability includes the random frame and the algorithm's internal
randomness.
\end{lemma}

\begin{proof}
Condition first on the algorithm's random seed. Let
\[
 \gamma_i:=\max_{s\le i}\prog_0(h(U^\top y_s)),
 \qquad \gamma_0:=0.
\]
Call query $i$ nonleaking if
$\prog_{\boldsymbol a}(U^\top y_i)\le\gamma_{i-1}$. On a nonleaking
history, \eqref{eq:nonuniform-support-rule} implies
$\gamma_i\le\gamma_{i-1}+1$, and hence $\gamma_i\le i$.

We bound the probability of the first leakage by deferred decisions. For
column $u_j$ of $U$ and query $y_i$, apply Gram--Schmidt to $y_i$ after
$u_1,\ldots,u_{j-1},y_1,\ldots,y_{i-1}$; call the normalized residual
$w_{j,i}$, or zero if the residual vanishes. On the event
$j>\gamma_{i-1}$, all previous responses use only columns below $j$ and can
be reconstructed from the previous projections and those columns. Thus,
conditional on this enlarged history, $y_i$ and $w_{j,i}$ are fixed, while
the normalized residual of $u_j$ orthogonal to
\[
 \spanop\{u_1,\ldots,u_{j-1},y_1,\ldots,y_{i-1}\}
\]
is uniform on the sphere in that orthogonal complement. This is exactly the
Haar-invariance argument given in the proof of Lemma~\ref{lem:rotation}; the
remaining dimension is at least $d-i-j$.

Lemma~\ref{lem:spherical-tail} therefore yields, conditionally on every
nonleaking past,
\begin{equation}\label{eq:nonuniform-one-leak}
 \mathbb P\left(
 \abs{\ip{u_j}{w_{j,i}}}\ge
 \frac{\underline a}{4R_0\sqrt t}
 \ \middle|\ \text{past}\right)
 \le4\exp\left[-
 \frac{(d-i-j)\underline a^2}{256R_0^2t}\right].
\end{equation}
If none of these events occurs for a fixed $j$ up to query $s$, orthogonality
and Cauchy--Schwarz give
\[
 \abs{\ip{u_j}{y_s}}
 \le\left(\sum_{i=1}^s\abs{\ip{u_j}{w_{j,i}}}^2\right)^{1/2}
 \norm{y_s}
 <\frac{\underline a}{4}\le\frac{a_j}{4}.
\]
With $C_{\rm rot,d}=1024$, the exponent in
\eqref{eq:nonuniform-one-leak} is at least
$4\log(8Dt/\delta)$. A union bound over $i\le t$ and $j\le D$ therefore
makes the probability of any leakage at most
$4Dt(\delta/(8Dt))^4\le\delta$. On the complementary event, $\gamma_i\le i<M$
and every query has progress at most $\gamma_{i-1}<M$. This proves
\eqref{eq:nonuniform-rotation-conclusion}. Averaging back over the
algorithm's random seed completes the proof.
\end{proof}

We now combine the weighted barrier, compactification, and rotation estimate, and scale them to the prescribed smoothness and PL constants.

\begin{theorem}
\label[theorem]{thm:noiseless-accuracy}
There are universal constants $\kappa_{\rm det}\ge1$, $c_{\rm det}>0$, and
$c_{0,\rm det}>0$ such that the following holds. Let
$\kappa=L/\mu\ge\kappa_{\rm det}$ and
$0<\epsilon\le c_{0,\rm det}\Delta_0$. For every adaptive randomized
gradient-only method using fewer than
\begin{equation}\label{eq:noiseless-accuracy-lower}
 c_{\rm det}\kappa\log\frac{\Delta_0}{\epsilon}
\end{equation}
exact-gradient calls, there are a finite dimension $d$ and a fixed globally
$L$-smooth, globally $\mu$-PL function with initial gap at most $\Delta_0$
such that
\begin{equation}\label{eq:noiseless-probability-gap}
 \mathbb P\bigl(f(\widehat x_T)-f^\star\ge4\epsilon\bigr)
 \ge\frac34.
\end{equation}
In particular, its expected error is larger than $\epsilon$. The proof gives
an explicit admissible dimension in \eqref{eq:hd-dd}.
\end{theorem}

\begin{proof}
Let
\begin{equation}\label{eq:noiseless-m-choice}
 m_{\rm d}:=\left\lfloor\frac{\kappa}{4C_{\rm det}L_{\rm det}}\right\rfloor,\qquad m:=m_{\rm d}.
\end{equation}
Take $\kappa_{\rm det}:=24C_{\rm det}L_{\rm det}$. Then $m\ge5$ and
$m\ge\kappa/(8C_{\rm det}L_{\rm det})$. Put
\begin{equation}\label{eq:noiseless-lambda-s}
 \lambda:=\frac{\Delta_0}{C_{\rm init}m},
 \qquad
 s^2:=\frac{L}{\lambda L_{\rm det}}.
\end{equation}
Let
\begin{equation}\label{eq:noiseless-K-choice}
 K:=2+\left\lfloor
 \frac{\log(\Delta_0/(16C_{\rm init}\epsilon))}
 {\log(1/r_\gamma)}\right\rfloor.
\end{equation}
Take $c_{0,\rm det}:=(16C_{\rm init})^{-2}$. Then the logarithm is
nonnegative and, writing $B=\log(16C_{\rm init})$,
$\log(\Delta_0/\epsilon)\ge2B$. Since $1+\lfloor z\rfloor\ge z$,
\begin{equation}\label{eq:K-log-lower}
 K-1\ge\frac{1}{2\log(1/r_\gamma)}
 \log\frac{\Delta_0}{\epsilon}
 =:c_K\log\frac{\Delta_0}{\epsilon}.
\end{equation}

For an orthonormal $U$, define
\begin{equation}\label{eq:noiseless-scaled-function}
 f_U(x):=\lambda\widehat{\mathsf W}_{m,K,U}(sx).
\end{equation}
By Lemma~\ref{lem:weighted-compactification}, this function is $L$-smooth.
Furthermore,
\begin{align*}
 f_U(x)-f_U^\star
 &\le\lambda C_{\rm det}m
 \norm{\nabla\widehat{\mathsf W}_{m,K,U}(sx)}^2\\
 &=\frac{C_{\rm det}mL_{\rm det}}{L}\norm{\nabla f_U(x)}^2
 \le\frac1{4\mu}\norm{\nabla f_U(x)}^2
 \le\frac1{2\mu}\norm{\nabla f_U(x)}^2.
\end{align*}
Thus it is globally $\mu$-PL, and
$f_U(0)-f_U^\star\le\lambda C_{\rm init}m=\Delta_0$.

The calibrated chain has $N=mK$ coordinates; set $H:=N-m=m(K-1)$
for the number that must be traversed before its final layer. Recall that
$R=256\sqrt m$ is the clipping radius. The scaled exact gradient can be simulated from the
rotated-chain response. At an external query $x$, send
$y=\rho_R(sx)$ to the oracle returning
$U\nabla\mathsf W_{m,K}(U^\top y)$, and compute
\[
 \lambda s\left[J\rho_R(sx)^\top
 U\nabla\mathsf W_{m,K}(U^\top\rho_R(sx))
 +\nabla Q_R(sx)\right]=\nabla f_U(x).
\]
Thus every effective query has norm below $2R$. Append the algorithm's output
as one final effective query whose response is ignored. Apply
Lemma~\ref{lem:nonuniform-rotation} with
\[
 D=N,
 \quad M=H+1,
 \quad R_0=2R,
 \quad \underline a=a_N=\gamma^{K-1},
 \quad \delta=\frac14.
\]
For $T\le H/4$, the appended output gives at most
$t_{\rm d}:=\lfloor H/4\rfloor+1$ effective queries. In particular,
$t_{\rm d}<H+1=M$. This determines an admissible ambient dimension:
\begin{equation}\label{eq:hd-dd}
 d_{\rm d}:=\left\lceil2D+C_{\rm rot,d}
       \frac{(2R)^2t_{\rm d}}{\underline a^2}
       \log(32Dt_{\rm d})\right\rceil.
\end{equation}
It satisfies \eqref{eq:nonuniform-rotation-dimension}, so with probability
at least $3/4$,
\[
 \prog_{\boldsymbol a}
 \bigl(U^\top\rho_R(s\widehat x_T)\bigr)\le N-m.
\]
The compactified guard is nonnegative, so
Lemma~\ref{lem:weighted-chain} gives on this event
\begin{align}
 f_U(\widehat x_T)-f_U^\star
 &\ge2\lambda m\gamma^{2(K-1)}.
 \label{eq:noiseless-scaled-barrier}
\end{align}
If
$x_\epsilon=\log(\Delta_0/(16C_{\rm init}\epsilon))/\log(1/r_\gamma)$,
then $K-1=1+\lfloor x_\epsilon\rfloor\le1+x_\epsilon$ and, because
$r_\gamma<1$,
\[
 r_\gamma^{K-1}\ge r_\gamma^{1+x_\epsilon}
 =16C_{\rm init}r_\gamma\frac{\epsilon}{\Delta_0}.
\]
Consequently,
\begin{equation}\label{eq:noiseless-final-barrier}
 f_U(\widehat x_T)-f_U^\star
 \ge32r_\gamma\epsilon>4\epsilon.
\end{equation}
Averaging the probability over Haar $U$ and then using
$\sup_U p(U)\ge\mathbb E_Up(U)$ selects one fixed rotation for the given
randomized algorithm.

Finally, \eqref{eq:noiseless-m-choice} and \eqref{eq:K-log-lower} imply
$H\ge c\kappa\log(\Delta_0/\epsilon)$. Choosing $c_{\rm det}$ below one
quarter of this universal constant proves the theorem.
\end{proof}

\subsection{Completion of the high-dimensional lower bound}\label{subsec:hd-completion}
All ingredients are now available. We combine the two families at the minimax
level rather than adding their objectives, then verify the finite hard
dimension and its extension to larger prescribed dimensions.

\begin{proof}[Proof of Theorem~\ref{thm:hd-accuracy}]
Choose $\kappa_{\rm hd}:=\max\{\kappa_0,\kappa_{\rm det}\}$ and
$c_{0,\alpha}:=\min\{1/10,c_{0,\rm det}\}$. These choices enforce the
parameter ranges of both constructions. To place either hard instance in
one prescribed dimension, set
\begin{equation}\label{eq:hd-dimension-maximum}
 d_{\rm hd}:=\max\{d_{\rm s},d_{\rm d}\}.
\end{equation}
Let
\[
 A:=\kappa\log\frac{\Delta_0}{\epsilon},
 \qquad
 B:=\kappa\left(\frac{\sigma^2}{\mu\epsilon}\right)^{\thetaalpha}.
\]
The stochastic construction above provides a constant
$c_{{\rm s},\alpha}>0$ such that every
$T<c_{{\rm s},\alpha}B$ fails on a fixed legal instance in dimension
$d_{\rm s}$.  The exact-gradient construction provides a universal
$c_{\rm d}>0$ such that every $T<c_{\rm d}A$ fails on a fixed legal
instance in dimension $d_{\rm d}$.  The latter oracle has centered
$\alpha$-moment zero and is therefore legal for every announced
$\sigma\ge0$.

Choose
\[
 c_{1,\alpha}:=\frac14\min\{c_{{\rm s},\alpha},c_{\rm d}\}.
\]
If
$T<c_{1,\alpha}(A+B)$, then
$T<\frac12\min\{c_{{\rm s},\alpha},c_{\rm d}\}\max\{A,B\}$.
If $A\ge B$, the exact-gradient family applies; if $B>A$, the stochastic
family applies.  In either case the selected instance has dimension at most
$d_{\rm hd}$, expected gap larger than $\epsilon$, and the stated
constant-probability error.  This proves
\eqref{eq:hd-main-lower}.

It remains to simplify the dimension.  Since
$m_{\rm s}=\Theta(\kappa)$ and, when $\sigma>0$,
\[
 q_{\rm s}^{-1}
 \le C_\alpha\max\left\{1,
 \left(\frac{\sigma^2}{\mu\epsilon}\right)^{\thetaalpha}\right\},
\]
equation \eqref{eq:hd-ds} gives
\begin{equation}\label{eq:hd-ds-simple}
 d_{\rm s}
 =\widetilde O_\alpha\left[
 \kappa^2\left(
 1+\left(\frac{\sigma^2}{\mu\epsilon}\right)^{\thetaalpha}
 \right)\right].
\end{equation}
For the exact chain, $m_{\rm d}=\Theta(\kappa)$,
$K=\Theta(\log(\Delta_0/\epsilon))$, and
$t_{\rm d}=O(m_{\rm d}K)$.  The choice of $K$ gives
\[
 r_\gamma^{K-1}
 \ge16C_{\rm init}r_\gamma\frac{\epsilon}{\Delta_0},
 \qquad
 \underline a^{-2}
 \le\frac{\Delta_0}
 {16C_{\rm init}r_\gamma\epsilon}.
\]
Using $R^2=2^{16}m_{\rm d}$ in \eqref{eq:hd-dd},
\begin{equation}\label{eq:hd-dd-simple}
 d_{\rm d}
 =\widetilde O\left(
 \kappa^2\frac{\Delta_0}{\epsilon}\right).
\end{equation}
Because $\Delta_0/\epsilon$ is bounded below by a universal constant in the
theorem's range, combining these dimension estimates gives
\begin{equation}\label{eq:hd-dimension-simplified}
 d_{\rm hd}=\widetilde O_\alpha\!\left(
 \kappa^2\left[\frac{\Delta_0}{\epsilon}+Q_\epsilon\right]\right).
\end{equation}
To obtain the prescribed-dimensional formulation in
Theorem~\ref{thm:hd-accuracy}, let $d\ge d_{\rm hd}$ and let a selected hard
instance have dimension $d_0\le d_{\rm hd}$. Extend it by
\[
 \widetilde f(y,z)=f(y)+\frac\mu2\|z\|^2,\qquad
 \widetilde G(y,z)=(G(y),\mu z),\qquad
 (y,z)\in\mathbb R^{d_0}\times\mathbb R^{d-d_0}.
\]
This extension is $L$-smooth and $\mu$-PL, has the same initial gap and
centered moment bound, and reveals no new information about the hard part.
Any algorithm on the extension can be simulated with one original oracle
call per query. Its gap is at least the original-part gap. Thus the lower
bound holds in every dimension at or above the explicit threshold.
\end{proof}

The delicate points are the following.  Bernoulli thinning has centered
moment $O(q^{1-\alpha})$, not $O(q^{-\alpha})$.  Under the scaling
$f(x)=\lambda\widehat F(sx)$, the oracle noise is multiplied by
$\lambda s=\sqrt{\lambda L/L_b}$, which is why the exponent is
$\theta_\alpha$.  The case $q_{\rm s}=1$ is not discarded: the exact
order-$m_{\rm s}$ chain already dominates the desired stochastic expression
there.  The weighted chain's PL coefficient is independent of its number of
layers because $S_j\le5m_{\rm d}a_j^2$ while
$|\nabla_j\mathsf W|\gtrsim a_j$.  Finally, Haar randomization is only an
existence device: averaging over $U$ and then taking a supremum selects one
fixed hard rotation for the given randomized algorithm.

\section{Proofs for Section~\ref{sec:upper}}\label{app:upper-proofs}
We first resolve the distinction between ordinary PL and feasible descent,
then analyze the robust gradient estimator, and finally combine the two
through a deterministic contraction. Throughout this section $\mathcal X$
is nonempty, closed, and convex. The mirror map is differentiable on an open
convex neighborhood of $\mathcal X$ and is $1$-strongly convex. Its Bregman divergence therefore satisfies
\eqref{eq:bregman-definition}.
The model $Q_x(z;v)=\langle v,z-x\rangle+LD_\psi(z,x)$ is continuous,
coercive, and $L$-strongly convex as a function of $z\in\mathcal X$.
Consequently its minimizer exists and is unique, even on an unbounded domain.
The mirror-PL assumption is the feasible-decrement inequality
\eqref{eq:mirror-PL}, not ordinary gradient-norm PL on a constrained set.

\begin{proof}[Proof of Proposition \ref{prop:boundary-obstruction}]
Let
\[
 \cX:=\R\times[0,\infty),
 \qquad
 u(s):=1-\cos s,
 \qquad
 q(s):=u(s)-u(s)^2,
\]
and define
\begin{equation}\label{eq:constraint-counterexample-f}
 f(s,y):=q(s)+(y+1)^2.
\end{equation}
Because $u(s)\in[0,2]$, one has $-2\le q(s)\le1/4$.  Moreover,
\[
 q'(s)=\sin s\,(2\cos s-1),
 \qquad
 q''(s)=4\cos^2s-\cos s-2.
\]
The quadratic polynomial $4c^2-c-2$ lies in $[-33/16,3]$ for
$c\in[-1,1]$.  Hence $\abs{q''(s)}\le3$, and the Hessian of $f$ is the
diagonal matrix $\operatorname{diag}(q''(s),2)$.  Thus $f$ is globally
$3$-smooth.

The minimum over $\cX$ is attained when $y=0$ and $q(s)=-2$, so
$f^\star:=\inf_{\cX}f=-1$.  Put $a=(y+1)^2\ge1$.  Then
\[
 f(s,y)-f^\star=q(s)+a+1\le a+\frac54,
\]
whereas
\[
 \frac12\norm{\nabla f(s,y)}^2
 =\frac12\abs{q'(s)}^2+2(y+1)^2\ge2a.
\]
Since $2a\ge\frac89(a+5/4)$ for every $a\ge1$, inequality
\eqref{eq:ordinary-pl-counterexample} holds with $\mu=8/9$.

Now take $x_0=(0,0)$.  Then $\nabla f(x_0)=(0,2)$ and
$f(x_0)-f^\star=2$.  Let $\psi$ be any differentiable strongly convex
mirror map, and let $\eta_{\rm md}>0$ be arbitrary.  For every $z=(z_1,z_2)\in\cX$,
\[
 \ip{\nabla f(x_0)}{z-x_0}+\frac1{\eta_{\rm md}} D_\psi(z,x_0)
 =2z_2+\frac1{\eta_{\rm md}} D_\psi(z,x_0)\ge0,
\]
with equality at $z=x_0$.  Thus $x_0$ is the unique mirror-descent update.
The same argument repeats at every iteration, so the method never reaches a
minimizer.  This remains true for an exact oracle, completing the proof.
\end{proof}   

\begin{remark}
    The obstruction is not a technical artifact: a constrained upper bound must
control a feasible descent measure rather than the full gradient norm.  We
therefore use Definition~\ref{def:mirror-pl}.  In the unconstrained
Euclidean setting it is exactly the ordinary PL condition used in the lower
bound, so the two parts of the manuscript match on their common domain.
\end{remark}

\subsection{Centered clipping under a finite \texorpdfstring{$\alpha$}{alpha}-moment}
\label{subsec:upper-clipping}
This subsection separates three tasks: controlling the moment of a block
mean, bounding the number of unreliable blocks, and converting those bounds
into a translation-invariant gradient estimate. Clipping is centered at an
estimated mean, so none of the following bounds requires a bounded true
gradient. The geometric-median principle is classical
\citep{Minsker2015,LugosiMendelson2019}; the elementary conditional argument
below makes the needed constants explicit.

\begin{lemma}
\label[lemma]{lem:block-alpha-moment}
Let $\zeta_1,\ldots,\zeta_b$ be independent random vectors in a Euclidean
space with $\mathbb E\zeta_i=0$ and
$\mathbb E\|\zeta_i\|^\alpha\le\sigma^\alpha$.  Then
\begin{equation}\label{eq:block-alpha-moment}
 \mathbb E\left\|\frac1b\sum_{i=1}^b\zeta_i\right\|^\alpha
 \le\frac{2^\alpha\sigma^\alpha}{b^{\alpha-1}}.
\end{equation}
The same statement holds conditionally on any sigma-field with respect to
which the vectors are conditionally independent and centered.
\end{lemma}

\begin{proof}
Introduce an independent copy $\zeta_1',\ldots,\zeta_b'$ and independent
Rademacher signs $\varepsilon_i$.  Jensen's inequality and conditional
symmetry give
\begin{align*}
 \mathbb E\left\|\sum_i\zeta_i\right\|^\alpha
 &\le\mathbb E\left\|\sum_i(\zeta_i-\zeta_i')\right\|^\alpha\\
 &=\mathbb E\left\|\sum_i\varepsilon_i
   (\zeta_i-\zeta_i')\right\|^\alpha\\
 &\le2^{\alpha-1}\mathbb E\left\|\sum_i\varepsilon_i\zeta_i\right\|^\alpha
   +2^{\alpha-1}\mathbb E\left\|\sum_i\varepsilon_i\zeta_i'\right\|^\alpha\\
 &=2^\alpha\mathbb E\left\|\sum_i\varepsilon_i\zeta_i\right\|^\alpha.
\end{align*}
Conditioning on $\zeta_1,\ldots,\zeta_b$ and using
$\alpha/2\le1$,
\begin{align*}
 \mathbb E_\varepsilon\left\|\sum_i\varepsilon_i\zeta_i\right\|^\alpha
 &\le\left(
 \mathbb E_\varepsilon\left\|\sum_i\varepsilon_i\zeta_i\right\|^2
 \right)^{\alpha/2}\\
 &=\left(\sum_i\|\zeta_i\|^2\right)^{\alpha/2}
 \le\sum_i\|\zeta_i\|^\alpha.
\end{align*}
Taking expectations yields
$\mathbb E\|\sum_i\zeta_i\|^\alpha\le2^\alpha b\sigma^\alpha$.
Division by $b^\alpha$ proves the result.  Every step is valid after
conditioning, proving the conditional version.
\end{proof}

Fix a query $x$, target accuracy $r>0$, and failure probability
$\rho\in(0,1/2)$.  Define
\begin{equation}\label{eq:heavy-estimator-constants}
 B_\alpha:=(16\cdot8^\alpha)^{1/(\alpha-1)},
 \qquad
 b:=\max\left\{1,
 \left\lceil B_\alpha
 \left(\frac\sigma r\right)^{\alpha/(\alpha-1)}
 \right\rceil\right\},
 \qquad
 K:=\left\lceil16\log\frac4\rho\right\rceil.
\end{equation}
Draw $2Kb$ fresh samples at the fixed point $x$, divide them into $2K$
blocks, and write
\begin{equation}\label{eq:app-block-means}
 Z_j:=\frac1b\sum_{\ell=1}^b G(x,\xi_{j,\ell}),
 \qquad j=1,\ldots,2K.
\end{equation}
Choose a geometric median of the first half,
\begin{equation}\label{eq:app-geometric-median}
 a\in\arg\min_{u\in\mathbb R^d}\sum_{j=1}^K\|u-Z_j\|,
\end{equation}
and define
\begin{equation}\label{eq:app-centered-clipped-blocks}
 Y_j:=a+\operatorname{clip}_r(Z_{K+j}-a),
 \qquad
 \widehat g(x):=\frac1K\sum_{j=1}^K Y_j.
\end{equation}

The block-moment estimate gives only a constant success probability for
one block. We therefore need a concentration statement that remains valid
when its indicators are merely adapted rather than independent.

\begin{lemma}
\label[lemma]{lem:bad-block-concentration}
Let $I_1,\ldots,I_K\in\{0,1\}$ be adapted to a filtration and suppose
$\mathbb E[I_j\mid\mathscr F_{j-1}]\le p$.  Then, for $a>p$,
\begin{equation}\label{eq:conditional-hoeffding}
 \mathbb P\left(\sum_{j=1}^K I_j\ge aK\right)
 \le\exp[-2K(a-p)^2].
\end{equation}
In particular, for $p=1/16$ and $a=1/4$, the right-hand side is
$\exp(-9K/128)$.
\end{lemma}

\begin{proof}
Put $D_j=I_j-\mathbb E[I_j\mid\mathscr F_{j-1}]$.  Conditional Hoeffding's
lemma gives
$\mathbb E[e^{\lambda D_j}\mid\mathscr F_{j-1}]\le e^{\lambda^2/8}$.
Iteration and Markov's inequality imply
\[
 \mathbb P\left(\sum_{j=1}^K D_j\ge u\right)
 \le\exp(-2u^2/K).
\]
If $\sum I_j\ge aK$, then
$\sum D_j\ge(a-p)K$, proving the claim.
\end{proof}

We can now combine the block and counting estimates. The pilot controls
the clipping center, and the second collection limits the contribution of
all remaining outliers.

\begin{lemma}
\label[lemma]{lem:centered-clipping}
Under \eqref{eq:lower-oracle-assumption}, for every adaptively selected $x$, every
$r>0$, and $\rho\in(0,1/2)$, the estimator above satisfies
\begin{equation}\label{eq:centered-clipping-prob}
 \mathbb P\left(
 \|\widehat g(x)-\nabla f(x)\|\le r
 \ \middle|\ \text{history before the estimator}\right)\ge1-\rho.
\end{equation}
It uses at most
\begin{equation}\label{eq:centered-clipping-cost}
 C_{{\rm est},\alpha}
 \left[1+\left(\frac\sigma r\right)^{\alpha/(\alpha-1)}\right]
 \log\frac4\rho
\end{equation}
stochastic-gradient calls, where one may take
$C_{{\rm est},\alpha}=34(1+B_\alpha)$.
\end{lemma}

\begin{proof}
Condition on the full history before the estimator and put
$m_x=\nabla f(x)$.  Lemma~\ref{lem:block-alpha-moment} and the definition of
$b$ give, conditionally at the beginning of every block,
\[
 \mathbb E\|Z_j-m_x\|^\alpha
 \le\frac{2^\alpha\sigma^\alpha}{b^{\alpha-1}}.
\]
Consequently,
\begin{align}
 \mathbb P\left(\|Z_j-m_x\|>\frac r4
 \ \middle|\ \text{past blocks}\right)
 &\le\frac{2^\alpha4^\alpha\sigma^\alpha}
 {b^{\alpha-1}r^\alpha}
 =\frac{8^\alpha\sigma^\alpha}{b^{\alpha-1}r^\alpha}
 \le\frac1{16}.
 \label{eq:block-good-prob}
\end{align}
The last inequality also holds when $b=1$: in that case
$B_\alpha(\sigma/r)^{\alpha/(\alpha-1)}\le1$, which is equivalent to the
required bound.

Call a block good if $\|Z_j-m_x\|\le r/4$.  Applying
Lemma~\ref{lem:bad-block-concentration} separately to the pilot and clipping
collections shows that, with conditional probability at least $1-\rho/2$,
at least $3K/4$ blocks are good in each collection.  We work on this event.

Let $G,B$ be the good and bad pilot indices and put
$d_0=\|a-m_x\|$.  For $j\in G$,
\[
 \|a-Z_j\|-\|m_x-Z_j\|\ge d_0-r/2,
\]
whereas for $j\in B$ the same difference is at least $-d_0$.  Optimality of
the geometric median gives
\[
 0\ge(|G|-|B|)d_0-|G|r/2.
\]
Since $|G|\ge3K/4$ and $|B|\le K/4$,
\begin{equation}\label{eq:pilot-center-error}
 \|a-m_x\|\le\frac{3r}{4}.
\end{equation}
For a good block in the second collection,
$\|Z_{K+j}-a\|\le r$, so clipping leaves it unchanged and
$\|Y_j-m_x\|\le r/4$.  For a bad block,
$\|Y_j-m_x\|\le r+3r/4=7r/4$.  At most $K/4$ blocks are bad; hence
\[
 \|\widehat g(x)-m_x\|
 \le\frac34\frac r4+\frac14\frac{7r}{4}
 =\frac{5r}{8}<r.
\]
This proves \eqref{eq:centered-clipping-prob}.  Finally,
$b\le1+B_\alpha(\sigma/r)^{\alpha/(\alpha-1)}$ and
$K\le17\log(4/\rho)$, which gives \eqref{eq:centered-clipping-cost}.
\end{proof}

\subsection{Inexact mirror contraction and the stagewise method}
\label{subsec:upper-md}
The preceding estimator supplies a deterministic error tolerance on a
simultaneous high-probability event. We now quantify how that error perturbs
the exact mirror step; no probabilistic argument is needed in this part.
For $L$-smooth $f$, \eqref{eq:bregman-definition} implies
$f(z)\le f(x)+Q_x(z;\nabla f(x))$. Hence
$\mathcal G_{\psi,L}^2(x)\le2L(f(x)-f^\star)$.
Together with mirror-PL, this shows $\mu\le L$ whenever the objective has a
nonoptimal feasible point. If $\mu>L$, every feasible point is optimal and
returning $x_0$ without any oracle call suffices. This is the first branch
of Algorithm~\ref{alg:stagewise-centered-clipped-mirror-descent} and
Theorem~\ref{thm:upper-accuracy}. The positive-query bound below concerns
only the remaining case $\kappa=L/\mu\ge1$.

Given $x\in\cX$ and an approximate gradient
$\widehat g=\nabla f(x)+e$, define
\begin{equation}\label{eq:inexact-mirror-step}
 x_+:=\arg\min_{z\in\cX}Q_x(z;\widehat g).
\end{equation}
Let
\begin{equation}\label{eq:exact-mirror-step}
 z_+:=\arg\min_{z\in\cX}Q_x(z;\nabla f(x))
\end{equation}
be the corresponding exact step.

\begin{lemma}
\label[lemma]{lem:inexact-mirror-contraction}
Under $1$-strong convexity of $\psi$,
$L$-smoothness of $f(x)$, and \eqref{eq:mirror-PL},
\begin{equation}\label{eq:inexact-contraction}
 f(x_+)-f^\star
 \le\left(1-\frac\mu L\right)(f(x)-f^\star)
     +\frac1L\norm{e}^2.
\end{equation}
\end{lemma}

\begin{proof}
The variational inequalities for \eqref{eq:inexact-mirror-step} and
\eqref{eq:exact-mirror-step} are
\begin{align}
 \ip{\widehat g+L(\nabla\psi(x_+)-\nabla\psi(x))}{z_+-x_+}&\ge0,
 \label{eq:VI-inexact}\\
 \ip{\nabla f(x)+L(\nabla\psi(z_+)-\nabla\psi(x))}{x_+-z_+}&\ge0.
 \label{eq:VI-exact}
\end{align}
Adding them and using
\[
 \ip{\nabla\psi(x_+)-\nabla\psi(z_+)}{x_+-z_+}
 \ge\norm{x_+-z_+}^2
\]
yields
\begin{equation}\label{eq:mirror-point-stability}
 L\norm{x_+-z_+}^2
 \le\ip{e}{z_+-x_+}
 \le\norm{e}\,\norm{x_+-z_+}.
\end{equation}
Consequently,
\begin{equation}\label{eq:mirror-distance-error}
 \norm{x_+-z_+}\le\frac{\norm{e}}L.
\end{equation}

Optimality of $x_+$ for the perturbed model gives
$Q_x(x_+;\widehat g)\le Q_x(z_+;\widehat g)$.  Therefore
\begin{align}
 Q_x(x_+;\nabla f(x))
 &=Q_x(x_+;\widehat g)-\ip{e}{x_+-x}\notag\\
 &\le Q_x(z_+;\widehat g)-\ip{e}{x_+-x}\notag\\
 &=Q_x(z_+;\nabla f(x))+\ip{e}{z_+-x_+}\notag\\
 &\le Q_x(z_+;\nabla f(x))+\frac1L\norm{e}^2,
 \label{eq:model-perturbation-bound}
\end{align}
where the last step uses \eqref{eq:mirror-distance-error}.
Smoothness and $1$-strong convexity of $\psi$ imply
\begin{align*}
 f(x_+)
 &\le f(x)+\ip{\nabla f(x)}{x_+-x}
       +\frac L2\norm{x_+-x}^2\\
 &\le f(x)+Q_x(x_+;\nabla f(x)).
\end{align*}
Combining this with \eqref{eq:model-perturbation-bound} and the definition
\eqref{eq:mirror-gradient-decrement},
\begin{align*}
 f(x_+)-f^\star
 &\le f(x)-f^\star
   -\frac1{2L}\mathcal G_{\psi,L}^2(x)
   +\frac1L\norm{e}^2\\
 &\le\left(1-\frac\mu L\right)(f(x)-f^\star)
   +\frac1L\norm{e}^2,
\end{align*}
where the last inequality is \eqref{eq:mirror-PL}.
\end{proof}

If $\mu>L$, Algorithm~\ref{alg:stagewise-centered-clipped-mirror-descent}
returns $x_0$ without querying the oracle. Otherwise $\kappa\ge1$; for
$\epsilon\in(0,\Delta_0)$ and $\delta\in(0,1/2)$, define
\begin{equation}\label{eq:stage-parameters}
 S:=\left\lceil\log_2\frac{\Delta_0}{\epsilon}\right\rceil,
 \qquad
 J:=\lceil\kappa\log4\rceil,
 \qquad
 \Delta_s:=\frac{\Delta_0}{2^s},
 \qquad
 r_s:=\frac12\sqrt{\mu\Delta_s}.
\end{equation}
At every one of the $J$ iterations of stage $s$, call the estimator of
Subsection~\ref{subsec:upper-clipping} at the current point with accuracy $r_s$
and conditional failure probability
\begin{equation}\label{eq:per-estimate-failure}
 \rho:=\frac{\delta}{JS},
\end{equation}
then update
\begin{equation}\label{eq:stagewise-mirror-update}
 x_{t+1}:=\arg\min_{z\in\mathcal X}
 \left\{\langle\widehat g_t,z-x_t\rangle+L D_\psi(z,x_t)\right\}.
\end{equation}
Return the last iterate after stage $S-1$. In
\eqref{eq:heavy-estimator-constants}, use $r=r_s$ to obtain $b_s$ and
$\rho=\delta/(JS)$ to obtain the common $K$. Together with
\eqref{eq:stage-parameters}, these are all parameters of
Algorithm~\ref{alg:stagewise-centered-clipped-mirror-descent}.
For $\mu\le L$, we establish the following explicit version of the
sample bound in Theorem~\ref{thm:upper-accuracy}; for $\mu>L$, $N=0$:
\begin{equation}\label{eq:upper-sample-complexity}
 N\le C_{0,\alpha}\kappa\left[
 \left\lceil\log_2\frac{\Delta_0}{\epsilon}\right\rceil
       +Q_\epsilon\right]
 \log\left(\frac{16\kappa}{\delta}
       \left\lceil\log_2\frac{\Delta_0}{\epsilon}\right\rceil\right).
\end{equation}

\begin{proof}[Proof of Theorem~\ref{thm:upper-accuracy}]
The case $\mu>L$ follows from the feasible-decrement argument at the start
of this subsection: every feasible point is optimal, so the returned $x_0$
is exact and costs no queries. Suppose henceforth that $\mu\le L$.
There are $JS$ robust-estimator calls.  The conditional guarantee in
Lemma~\ref{lem:centered-clipping} and a union bound imply that, with
probability at least $1-\delta$,
\begin{equation}\label{eq:all-gradient-estimates-good}
 \|\widehat g_t-\nabla f(x_t)\|\le r_s
\end{equation}
at every update of every stage. Work pathwise on this event; all
probabilistic estimates have already been established.

We prove inductively that the objective gap at the beginning of stage $s$ is
at most $\Delta_s$.  Write $\Delta_{s,j}:=f(x_{sJ+j})-f^\star$ for the gap after $j$ updates in stage $s$.
Lemma~\ref{lem:inexact-mirror-contraction} gives
\begin{equation}\label{eq:stage-recurrence}
 \Delta_{s,j+1}\le\left(1-\frac1\kappa\right)\Delta_{s,j}+\frac{r_s^2}{L}.
\end{equation}
After $J$ iterations,
\begin{align}
 \Delta_{s,J}
 &\le e^{-J/\kappa}\Delta_s+\frac{r_s^2}{\mu}
 \le\frac{\Delta_s}{4}+\frac{\Delta_s}{4}
 =\Delta_{s+1}.
 \label{eq:stage-halving}
\end{align}
Thus $f(x_{\mathrm{out}})-f^\star\le\Delta_S\le\epsilon$.

It remains to count oracle calls.  By
Lemma~\ref{lem:centered-clipping}, one estimator call in stage $s$ costs at
most
\begin{align*}
 C_{{\rm est},\alpha}
 \left[1+\left(\frac{\sigma}{r_s}\right)^{\alpha/(\alpha-1)}\right]
 \log\left(\frac{4JS}{\delta}\right)
 &=C_{{\rm est},\alpha}
 \left[1+2^{2\thetaalpha}
 \left(\frac{\sigma^2}{\mu\Delta_s}\right)^{\thetaalpha}\right]
 \log\left(\frac{4JS}{\delta}\right).
\end{align*}
Hence
\begin{align}
 N
 &\le C_{{\rm est},\alpha}J
 \log\left(\frac{4JS}{\delta}\right)
 \left[S+2^{2\thetaalpha}
 \left(\frac{\sigma^2}{\mu}\right)^{\thetaalpha}
 \sum_{s=0}^{S-1}\Delta_s^{-\thetaalpha}\right].
 \label{eq:upper-count-before-sum}
\end{align}
The geometric sum satisfies
\begin{align}
 \sum_{s=0}^{S-1}\Delta_s^{-\thetaalpha}
 &=\frac{\Delta_S^{-\thetaalpha}-\Delta_0^{-\thetaalpha}}
 {2^{\thetaalpha}-1}
 <\frac{\Delta_S^{-\thetaalpha}}{2^{\thetaalpha}-1}
 <\frac{1}{1-2^{-\thetaalpha}}\epsilon^{-\thetaalpha}.
 \label{eq:geometric-heavy-sum}
\end{align}
Because $J\le3\kappa$ and $4JS\le12\kappa S$, all factors depending only
on $\alpha$ can be absorbed into $C_{0,\alpha}$, yielding
\eqref{eq:upper-sample-complexity}.
\end{proof}

\section{Proofs for Section~\ref{sec:fixed}}\label{app:fixed-dimensional proof}
We work with $f\in\mathcal F_d(L,\mu,\Delta_0)$, the oracle model
\eqref{eq:lower-oracle-assumption}, and a fixed dimension $d$.
Theorem~\ref{thm:fd-upper-main} states the explicit bounds underlying
Theorem~\ref{cor:fd-123}; Theorem~\ref{thm:app-fixed-lower} supplies the
expected-risk and confidence lower bounds.
The algorithms below are specified before their analysis. Bisection uses the
one-dimensional ordering of minimizers, whereas gradient-flow trapping
uses cutting surfaces and reconstructs relative values from gradients only.
Their subroutines share the inputs $G,L,\sigma,\alpha$ and the exponent
$p_\alpha$ from \eqref{eq:heavy-exponents}. In bisection, $M$ is the
number of blocks per half of the estimator; it plays the role of $K$ in
Appendix~\ref{subsec:upper-clipping}.
For reproducible parameter choices fix
\begin{equation}\label{eq:app-algorithm-constants}
 \mathfrak c_d=2^{16}d^2\quad(d\ge2),\qquad
 C_{{\rm rep},\alpha}=\max\{3,32^{1/(\alpha-1)}\}.
\end{equation}
These constants are deliberately conservative; the complexity statements
absorb dimension- and moment-dependent factors. Finite minimizations use a
fixed lexicographic tie-breaking rule. A geometric median may be chosen as
the minimum-norm element of its nonempty compact convex minimizer set.
The upper-bound algorithms handle the trivial constant-objective case by
returning $x_0$ when the supplied parameters satisfy $\mu>L$.

\begin{algorithm}[H]
\caption{Robust Stochastic Bisection}
\label{alg:robust-stochastic-bisection}
\small
\begin{algorithmic}[1]
\Require A globally $L$-smooth, $\mu$-PL function
    $f:\mathbb R\to\mathbb R$ and a conditionally unbiased
    stochastic-gradient oracle $G$ with centered $\alpha$-moment
    bounded by $\sigma^\alpha$, where $\alpha\in(1,2]$.
\Require $x_0\in\mathbb R$ with $f(x_0)-f^\star\leq\Delta_0$;
    $\epsilon\in(0,\Delta_0)$ and $\delta\in(0,1/2)$.
\Ensure $x_{\mathrm{out}}$ such that
    $\mathbb P(f(x_{\mathrm{out}})-f^\star\leq\epsilon)\geq1-\delta$.

\Statex $\displaystyle
    \operatorname{clip}_r(v)
    :=\operatorname{sgn}(v)\min\{|v|,r\},
    \qquad
    \log_2^+(z):=\max\{0,\log_2 z\}.
$
\State \textbf{if} $\mu>L$ \textbf{then return} $x_0$.
\State $\displaystyle
    \kappa\gets\frac{L}{\mu},\qquad
    D_0\gets\sqrt{\frac{2\Delta_0}{\mu}},\qquad
    r_\epsilon\gets\frac14\sqrt{2\mu\epsilon}.
$
\State $\displaystyle
    K_\epsilon\gets
    1+\left\lceil
        \frac12\log_2^+\!\left(\frac{\kappa\Delta_0}{\epsilon}\right)
    \right\rceil,
    \qquad
    \rho\gets\frac{\delta}{K_\epsilon}.
$
\State $\displaystyle
    B_\alpha\gets(16\cdot8^\alpha)^{1/(\alpha-1)},
    \qquad
    M\gets\left\lceil16\log\frac4\rho\right\rceil.
$
\State $\displaystyle
    b\gets\max\left\{
        1,\left\lceil
            B_\alpha
            \left(\frac{\sigma}{r_\epsilon}\right)^{\alpha/(\alpha-1)}
        \right\rceil
    \right\}.
$
\State $(\ell_0,u_0)\gets(x_0-D_0,x_0+D_0)$.

\For{$k=0,\ldots,K_\epsilon-1$}
    \State $c_k\gets(\ell_k+u_k)/2$.
    \State Draw $2Mb$ fresh, conditionally independent oracle samples
        $G_{j,q}=G(c_k,\xi_{j,q})$,
        $j=1,\ldots,2M$, $q=1,\ldots,b$.
    \State $\displaystyle
        Z_j\gets\frac1b\sum_{q=1}^{b}G_{j,q},
        \qquad j=1,\ldots,2M.
    $
    \State $\displaystyle
        a_k\in\operatorname*{arg\,min}_{a\in\mathbb R}
        \sum_{j=1}^{M}|a-Z_j|.
    $
    \State $\displaystyle
        \widehat g_k\gets
        a_k+\frac1M\sum_{j=1}^{M}
        \operatorname{clip}_{r_\epsilon}(Z_{M+j}-a_k).
    $

    \If{$|\widehat g_k|\leq2r_\epsilon$}
        \State \Return $c_k$.
    \ElsIf{$\widehat g_k>2r_\epsilon$}
        \State $(\ell_{k+1},u_{k+1})\gets(\ell_k,c_k)$.
    \Else
        \State $(\ell_{k+1},u_{k+1})\gets(c_k,u_k)$.
    \EndIf
\EndFor

\State \Return $(\ell_{K_\epsilon}+u_{K_\epsilon})/2$.
\end{algorithmic}
\end{algorithm}

\begin{algorithm}[H]
\caption{Stagewise Gradient-Flow Trapping}
\label{alg:gradient-flow-trapping}
\small
\begin{algorithmic}[1]
\Require A fixed dimension $d\geq2$; a globally $L$-smooth,
    $\mu$-PL function $f:\mathbb R^d\to\mathbb R$;
    a conditionally unbiased stochastic-gradient oracle $G$
    with centered $\alpha$-moment bounded by $\sigma^\alpha$.
\Require $x_0$ with $f(x_0)-f^\star\leq\Delta_0$;
    $\alpha\in(1,2]$, $\epsilon\in(0,\Delta_0)$,
    and $\delta\in(0,1/2)$.
\Ensure $x_{\mathrm{out}}$ such that
    $\mathbb P(f(x_{\mathrm{out}})-f^\star\leq\epsilon)\geq1-\delta$.

\State \textbf{if} $\mu>L$ \textbf{then return} $x_0$.
\State $\displaystyle
    \kappa\gets\frac L\mu,\qquad
    S\gets\left\lceil\log_2\frac{\Delta_0}{\epsilon}\right\rceil,
    \qquad
    J\gets \mathfrak c_d\left\lceil\log(2+\kappa)\right\rceil.
$
\State $x\gets x_0$, \quad $\rho\gets\delta/(2SJ)$.

\For{$s=0,\ldots,S-1$}
    \State $\displaystyle
        \Delta_s\gets\frac{\Delta_0}{2^s},\qquad
        \eta\gets\sqrt{\mu\Delta_s},\qquad
        c\gets\frac{\eta}{8},\qquad
        a\gets\frac{\eta}{\mathfrak c_dJ}.
    $
    \State $\displaystyle
        H_0\gets\frac{2\Delta_s}{c},
        \qquad
        R\gets\prod_{j=1}^{d}[x_j-H_0,x_j+H_0].
    $

    \For{$t=0,\ldots,J-1$}
        \State Write $R=\prod_{j=1}^{d}[\ell_j,u_j]$ and choose
            $i\in\operatorname*{arg\,max}_{1\leq j\leq d}(u_j-\ell_j)$.
        \State $s_t\gets u_i-\ell_i$.
        \If{$s_t\leq\eta/(4L\sqrt d)$}
            \State \textbf{break}.
        \EndIf

        \State $\displaystyle
            \tau_t\gets\frac{as_t}{\mathfrak c_d},\qquad
            h_t\gets\sqrt{\frac{as_t}{\mathfrak c_dL}},\qquad
            n_t\gets
            2^{\left\lceil\log_2(\mathfrak c_ds_t/h_t)\right\rceil}.
        $

        \State Define the oriented coordinate
            $\zeta(z)=z_i-\ell_i$ if $x_i\geq(\ell_i+u_i)/2$,
            and $\zeta(z)=u_i-z_i$ otherwise.
        \State $\displaystyle
            E\gets\{z\in R:\zeta(z)=s_t/6\},
            \qquad
            F\gets\{z\in R:\zeta(z)=s_t/3\}.
        $

        \ForAll{$H\in\{E,F\}$}
            \State $(\mathcal N_H,\widehat D_H)\gets$
                \Call{CutValues}{$H,x,n_t,\tau_t,\rho$}.
        \EndFor
        \State Let $\widehat D(z)=\widehat D_H(z)$ for
            $z\in\mathcal N_H$, $H\in\{E,F\}$.
        \State $\displaystyle
            A\gets
            \left\{
                z\in\mathcal N_E\cup\mathcal N_F:
                \widehat D(z)\leq-c\|z-x\|-\tau_t
            \right\}.
        $

        \If{$A=\varnothing$}
            \State $R\gets R\cap\{z:\zeta(z)\geq s_t/3\}$.
            \Comment{Keep the current pivot}
        \Else
            \State $\displaystyle
                z_\star\in\operatorname*{arg\,min}_{z\in A}
                \widehat D(z).
            $
            \If{$z_\star\in\mathcal N_F$}
                \State $R\gets R\cap\{z:\zeta(z)\geq s_t/6\}$.
            \Else
                \State $R\gets R\cap\{z:\zeta(z)\leq s_t/3\}$.
            \EndIf
            \State $x\gets z_\star$.
        \EndIf

        \State $c\gets c+a$.
    \EndFor
\EndFor

\State \Return $x$.
\end{algorithmic}
\end{algorithm}

\begin{algorithm}[t]
\caption{Gradient-Only Potential Reconstruction on a Cutting Surface}
\label{alg:cut-values}
\small
\begin{algorithmic}[1]
\Require A $k$-dimensional cutting hyperrectangle $H$;
    a reference pivot $x$ within distance $\sqrt{k+1}r$ of $H$,
    where all side lengths of $H$ lie in $[r/3,r]$; a dyadic resolution
    $n=2^M$, $M\in\mathbb N_0$; accuracy $\tau>0$ and failure probability $\rho\in(0,1/2)$.

\Function{CutValues}{$H,x,n,\tau,\rho$}
    \State $M\gets\log_2 n$.
    \State Construct nested Cartesian grids
        $\mathcal N_0,\ldots,\mathcal N_M$ on $H$,
        using $2^\ell$ equal subintervals per free coordinate
        at level $\ell$.
    \State $\mathcal E_0\gets\{(x,z):z\in\mathcal N_0\}$.
        \Comment{Connect the coarsest grid to the pivot}

    \For{$\ell=1,\ldots,M$}
        \ForAll{$z\in\mathcal N_\ell\setminus\mathcal N_{\ell-1}$}
            \State $\displaystyle
                \pi(z)\in
                \operatorname*{arg\,min}_{y\in\mathcal N_{\ell-1}}
                \|z-y\|.
            $
        \EndFor
        \State $\displaystyle
            \mathcal E_\ell\gets
            \{(\pi(z),z):
              z\in\mathcal N_\ell\setminus\mathcal N_{\ell-1}\}.
        $
    \EndFor

    \State $\displaystyle
        \mathcal E\gets\bigcup_{\ell=0}^{M}\mathcal E_\ell,
        \qquad
        \rho_{\mathrm{edge}}\gets\frac{\rho}{|\mathcal E|}.
    $
    \State $\displaystyle
        A_{\mathrm{det}}\gets
        \sum_{j=0}^{M}2^{(3k-4)j/5},
        \qquad
        A_{\mathrm{noise}}\gets
        \sum_{j=0}^{M}2^{(k-p_\alpha)j/(p_\alpha+1)}.
    $

    \For{$\ell=0,\ldots,M$}
        \State $\displaystyle
            u_\ell\gets
            \frac{\tau\,2^{(3k-4)\ell/5}}{2A_{\mathrm{det}}},
            \qquad
            v_\ell\gets
            \frac{\tau\,2^{(k-p_\alpha)\ell/(p_\alpha+1)}}{2A_{\mathrm{noise}}}.
        $
        \ForAll{$e=(a_e,b_e)\in\mathcal E_\ell$}
            \State $\widehat I_e\gets$
                \Call{RobustLineIntegral}
                {$a_e,b_e,u_\ell,v_\ell,\rho_{\mathrm{edge}}$}.
        \EndFor
    \EndFor

    \ForAll{$z\in\mathcal N_M$}
        \State Let $P(x,z)$ be the directed tree path from $x$ to $z$.
        \State $\displaystyle
            \widehat D(z)\gets\sum_{e\in P(x,z)}\widehat I_e.
        $
    \EndFor

    \State \Return $(\mathcal N_M,\widehat D)$.
\EndFunction
\end{algorithmic}
\end{algorithm}

\begin{algorithm}[t]
\caption{Robust Stochastic Line Integral}
\label{alg:robust-stochastic-line-integral}
\small
\begin{algorithmic}[1]
\Require Endpoints $a,b\in\mathbb R^d$;
    quadrature and noise tolerances $u,v>0$;
    failure probability $\rho\in(0,1/2)$.

\Function{RobustLineIntegral}{$a,b,u,v,\rho$}
    \State $h\gets\|b-a\|$.
    \If{$h=0$}
        \State \Return $0$.
    \EndIf

    \State $\displaystyle
        m\gets
        \max\left\{
            1,\left\lceil
                C_{{\rm rep},\alpha}\left[
                    \left(\frac{Lh^2}{u}\right)^{2/3}
                    +\left(\frac{\sigma h}{v}\right)^{p_\alpha}
                \right]
            \right\rceil
        \right\}.
    $
    \State $\displaystyle
        B\gets\left\lceil 4\log\frac4\rho\right\rceil.
    $

    \For{$r=1,\ldots,B$}
        \For{$j=1,\ldots,m$}
            \State Draw a fresh
                $U_{r,j}\sim
                \operatorname{Unif}([(j-1)/m,j/m])$.
            \State Query the oracle at $a+U_{r,j}(b-a)$ with fresh randomness.
            \State $\displaystyle
                W_{r,j}\gets
                \left\langle G(a+U_{r,j}(b-a),\xi_{r,j}),\,b-a\right\rangle.$
        \EndFor
        \State $\displaystyle
            \widehat I^{(r)}
            \gets\frac1m\sum_{j=1}^{m}W_{r,j}.
        $
    \EndFor

    \State \Return
        $\operatorname{median}
        \{\widehat I^{(1)},\ldots,\widehat I^{(B)}\}$.
\EndFunction
\end{algorithmic}
\end{algorithm}

\clearpage
With the procedures specified, we state their guarantees explicitly before
proving the lower and upper bounds separately.

\begin{theorem}
\label[theorem]{thm:fd-upper-main}
Let $d\ge1$ be fixed, $1<\alpha\le2$, $\sigma\ge0$,
$0<\epsilon<\Delta_0$, and $0<\delta<1/2$.
Assume $f\in\mathcal F_d(L,\mu,\Delta_0)$, $\kappa\ge1$, and the oracle
convention stated above. The algorithms may use a translated initial point
with the same gap bound.

\begin{enumerate}[label=\textup{(\roman*)}]
\item In dimension $d=1$, Algorithm~\ref{alg:robust-stochastic-bisection} returns
$f(x_{\mathrm{out}})-f^\star\le\epsilon$ with probability at least $1-\delta$ using
\begin{equation}\label{eq:fd-d1-upper-explicit}
 N\le C_\alpha K_\epsilon
 \bigl[1+Q_\epsilon\bigr]
 \log\frac{4K_\epsilon}{\delta},
 \qquad
 K_\epsilon:=1+\left\lceil
 \frac12\log_2^+\frac{\kappa\Delta_0}{\epsilon}
 \right\rceil.
\end{equation}

\item Let $d\ge2$, put $k=d-1$, and let
\begin{equation}\label{eq:fd-J-Lambda}
 J:=\mathfrak c_d\left\lceil\log(2+\kappa)\right\rceil,
 \qquad
 S:=\left\lceil\log_2\frac{\Delta_0}{\epsilon}\right\rceil,
\end{equation}
where $\mathfrak c_d$ is fixed in \eqref{eq:app-algorithm-constants}. Define
\begin{equation}\label{eq:fd-Dk}
 \mathcal D_k(z):=
 \begin{cases}
   z^{2/3},&k=1,\\
   z^{k/2},&k\ge2,
 \end{cases}
\end{equation}
and
\begin{equation}\label{eq:fd-Hkp}
 \mathcal H_{k,p}(\kappa,J):=
 \begin{cases}
   J^{p+1},&k<p,\\
   J^{p+1}\bigl[\log(2+\kappa J)\bigr]^{p+1},&k=p,\\
   J^p(\kappa J)^{(k-p)/2},&k>p.
 \end{cases}
\end{equation}
Then Algorithm~\ref{alg:gradient-flow-trapping} uses
\begin{equation}\label{eq:fd-general-upper-explicit}
 N\le C_{d,\alpha}\Lambda_{d,\alpha}
 \left[
 S\,\mathcal D_k(\kappa J)
 +\mathcal H_{k,p_\alpha}(\kappa,J)
 Q_\epsilon
 \right]
\end{equation}
stochastic-gradient calls, where
\begin{equation}\label{eq:fd-Lambda}
 \Lambda_{d,\alpha}:=
 \log\left(
 \frac{C_{d,\alpha}(S+1)(J+1)(2+\kappa J)^{k/2}}{\delta}
 \right).
\end{equation}
With probability at least $1-\delta$, its output satisfies
$f(x_{\mathrm{out}})-f^\star\le\epsilon$.
\end{enumerate}
\end{theorem}

\subsection{Main ideas and proof sketches}
The proof has four independent ingredients.

\begin{enumerate}
\item \emph{A universal one-coordinate heavy-tail obstruction.}
A pair of shifted quadratics is coupled so that the two oracle laws are
identical unless a rare event of probability $q$ occurs. At a fixed mean
shift, the centered $\alpha$-moment scales as $q^{1-\alpha}$.
Calibrating the shift to the moment budget makes the minimizer separation
of order $\sigma q^{(\alpha-1)/\alpha}/\mu$.  Taking
$q\asymp(\mu\epsilon/\sigma^2)^{\thetaalpha}$ produces the sharp mean-estimation
scale $Q_\epsilon$.  Embedding the construction in the first coordinate
works in every fixed dimension.

\item \emph{One-dimensional localization.}
For a one-dimensional PL function, the minimizer set is an interval, the
derivative is negative to its left and positive to its right, and PL implies a
quadratic-growth inequality.  A robust estimate of the derivative sign
therefore supports bisection.  This removes every polynomial dependence on
$\kappa$ from the stochastic term.

\item \emph{Gradient-flow traps in dimensions at least two.}
A compact box whose boundary cannot be reached from its pivot while decreasing
$f$ faster than slope $c$ must contain a point with gradient norm at most $c$.
Two parallel cuts of the longest side shrink such a box by a constant volume
factor.  Approximate relative values on a fine net of the new cut preserve the
trap after paying only
$O_d(\tau/s+Lh^2/s+ch^2/s^2)$ in the slope parameter.

\item \emph{Potential reconstruction from stochastic gradients.}
Relative function values are line integrals of the gradient.  Stratified randomized quadrature on a segment of length $h$ has
quadrature and oracle-noise errors of order $Lh^2/m^{3/2}$ and
$\sigma h/m^{(\alpha-1)/\alpha}$, respectively, at constant confidence.  A dyadic tree shares these line
integrals among all points of a $(d-1)$-dimensional cut.  The accumulated noise
is bounded up to logarithms exactly when $p_\alpha\ge d-1$; this is the source of the phase
transition in \eqref{eq:fd-matching-range}.
\end{enumerate}

Appendix~\ref{subsec:fd-lower-bisection} proves the two lower-bound
mechanisms and the bisection guarantee.
Appendices~\ref{subsec:fd-reconstruction} and \ref{subsec:fd-trapping}
analyze reconstruction and box refinement; Appendix~\ref{subsec:fd-phase}
then assembles the phase diagram.

\subsection{Universal lower bounds and the complete one-dimensional theory}\label{subsec:fd-lower-bisection}

The lower bound uses two distinct information-hiding mechanisms. Rare-event
observations conceal the sign of a shifted quadratic, whereas exact
gradient queries reveal only a few bits of a nested minimizer location.
We state the expected-risk and confidence consequences separately, then
prove the bisection guarantee from one-dimensional PL geometry.

\begin{theorem}
\label[theorem]{thm:app-fixed-lower}
Let $d\ge1$ be fixed, $\kappa\ge6$, $1<\alpha\le2$, and $\sigma\ge0$.
There are positive constants $c_\alpha,C,c_0$ such that, for every integer
$T\ge0$,
\begin{equation}\label{eq:fd-budget-lower}
 \mathsf R_T^{(d)}\ge c_\alpha\left[
 \Delta_0e^{-CT}+\min\left\{\Delta_0,
       \frac{\sigma^2}{\mu(T+1)^{\beta_\alpha}}\right\}\right].
\end{equation}
For every algorithm there is a fixed legal instance on which the gap is at
least the same bracket times a positive constant with probability strictly
larger than $1/4$. For $0<\epsilon\le c_0\Delta_0$ and
$0<\delta\le1/4$,
\begin{equation}\label{eq:app-fixed-det-confidence}
 \mathsf T_{\epsilon,\delta}^{(d)}\ge c\log(\Delta_0/\epsilon).
\end{equation}
If also $\sigma^2\ge C_\alpha\mu\epsilon$, then
\begin{equation}\label{eq:fd-lower-stochastic-confidence}
 \mathsf T_{\epsilon,\delta}^{(d)}
 \ge c_\alpha Q_\epsilon\log\frac1{2\delta}.
\end{equation}
At constant confidence $1/4$, combining the two bounds gives
$c_\alpha[\log(\Delta_0/\epsilon)+Q_\epsilon]$ for all $\sigma\ge0$.
\end{theorem}

The following two-point family supplies the stochastic part of this theorem.

\begin{lemma}
\label[lemma]{lem:fd-rare-event}
Fix $d\ge1$ and $q\in(0,1/2]$.  Let $\vartheta\in\{-1,+1\}$ and set
\begin{equation}\label{eq:fd-rare-defs}
 a:=\frac{\sigma}{(2q)^{1/\alpha}},
 \qquad
 v_\vartheta:=-\frac{\vartheta qa}{\mu},
 \qquad
 f_\vartheta(x):=\frac\mu2
 \|x-v_\vartheta e_1\|^2.
\end{equation}
At every query $x$, let the oracle return
\begin{equation}\label{eq:fd-rare-oracle}
 G_\vartheta(x):=\mu x+Z_\vartheta e_1,
 \qquad
 Z_\vartheta:=
 \begin{cases}
   \vartheta a,&\text{with probability }q,\\
   0,&\text{with probability }1-q,
 \end{cases}
\end{equation}
using fresh conditionally independent draws.  Then $f_\vartheta$ is
$\mu$-smooth and $\mu$-PL,
\begin{equation}\label{eq:fd-rare-valid}
 \mathbb E[G_\vartheta(x)]=\nabla f_\vartheta(x),
 \qquad
 \mathbb E\|G_\vartheta(x)-\nabla f_\vartheta(x)\|^\alpha
 \le\sigma^\alpha.
\end{equation}
If an algorithm makes $T$ calls and returns $X$, then under the uniform prior
on $\vartheta$,
\begin{equation}\label{eq:fd-rare-failure}
 \mathbb P\left(
 f_\vartheta(X)-f_\vartheta^\star
 \ge \frac{\sigma^2}{2^{1+2/\alpha}\mu}q^{\betaalpha}
 \right)
 \ge\frac12(1-q)^T.
\end{equation}
Writing $b_\alpha=2^{-1-2/\alpha}$, the same experiment satisfies
\begin{equation}\label{eq:app-rare-expected}
 \mathbb E[f_\vartheta(X)-f_\vartheta^\star]
 \ge b_\alpha\frac{\sigma^2}{\mu}q^{\beta_\alpha}(1-q)^T.
\end{equation}
\end{lemma}

\begin{proof}
The mean of $Z_\vartheta$ is $\vartheta qa$, so
\[
 \mathbb E G_\vartheta(x)
 =\mu x+\vartheta qa e_1
 =\mu(x-v_\vartheta e_1)=\nabla f_\vartheta(x).
\]
The centered noise has only its first coordinate nonzero, and
\begin{align*}
 \mathbb E|Z_\vartheta-\mathbb EZ_\vartheta|^\alpha
 &=a^\alpha\bigl[q(1-q)^\alpha+(1-q)q^\alpha\bigr]\\
 &\le2qa^\alpha=\sigma^\alpha.
\end{align*}
The quadratic identity
$\|\nabla f_\vartheta(x)\|^2/2
 =\mu(f_\vartheta(x)-f_\vartheta^\star)$ proves smoothness and PL.

Realize each draw as $Z_\vartheta=\vartheta a B_i$, where
$B_i\sim\operatorname{Bernoulli}(q)$ are independent of the sign and the
algorithm's random seed. Let $E_T=\{B_1=\cdots=B_T=0\}$, so that
$\mathbb P(E_T)=(1-q)^T$ even when $\sigma=0$. On
$E_T$, every response equals $\mu x$, independently of $\vartheta$; hence the
complete adaptive transcript, including the algorithm's internal randomness,
has the same conditional law for the two signs.  The two minimizers have first
coordinates $\pm r$, where
\[
 r=\frac{qa}{\mu}
 =\frac{\sigma}{2^{1/\alpha}\mu}
 q^{(\alpha-1)/\alpha}.
\]
For every realized output $X$,
$\max\{|X_1-r|,|X_1+r|\}\ge r$.  Conditional on $E_T$ and averaged over the
uniform sign, the probability of squared distance at least $r^2$ is therefore
at least $1/2$.  Multiplying this conditional probability by
$\mathbb P(E_T)=(1-q)^T$ proves \eqref{eq:fd-rare-failure}.
For the expected gap, average over the two signs on $E_T$:
$\bigl((X_1-r)^2+(X_1+r)^2\bigr)/2=X_1^2+r^2\ge r^2$.
The other coordinates contribute nonnegative terms. Multiplying by
$\mu/2$ and by $\mathbb P(E_T)$ proves \eqref{eq:app-rare-expected}.
Averaging over the two signs selects a fixed sign for either criterion;
no adversarial randomness remains in the resulting instance.
\end{proof}

\begin{proof}[Stochastic part of Theorem~\ref{thm:app-fixed-lower}]
Assume first $\sigma>0$ and put $b_\alpha=2^{-1-2/\alpha}$.
For the confidence bound, choose
\begin{equation}\label{eq:fd-q-epsilon}
 q_\epsilon=\left(\frac{4\mu\epsilon}{b_\alpha\sigma^2}
                  \right)^{\theta_\alpha}.
\end{equation}
Then the gap threshold in \eqref{eq:fd-rare-failure} is exactly $4\epsilon$.
The condition $\sigma^2\ge(4\cdot2^{\beta_\alpha}/b_\alpha)\mu\epsilon$
ensures $q_\epsilon\le1/2$. Its initial gap is $4\epsilon\le\Delta_0$
provided $\epsilon\le\Delta_0/4$. For
$T<(2q_\epsilon)^{-1}\log(1/(2\delta))$, the inequality
$\log(1-q)\ge-2q$ on $[0,1/2]$ gives
$\tfrac12(1-q_\epsilon)^T>\delta$.
Lemma~\ref{lem:fd-rare-event} therefore proves
\eqref{eq:fd-lower-stochastic-confidence} on a fixed sign instance.

For a fixed budget instead choose
\begin{equation}\label{eq:app-rare-budget-q}
 q_T=\min\left\{\frac1{2(T+1)},
        \left(\frac{\mu\Delta_0}{b_\alpha\sigma^2}
        \right)^{\theta_\alpha}\right\}.
\end{equation}
This choice has $0<q_T\le1/2$ and the initial gap is
\[
 g_T=\min\left\{\Delta_0,
  \frac{b_\alpha\sigma^2}{\mu[2(T+1)]^{\beta_\alpha}}\right\}.
\]
Bernoulli's inequality yields
$(1-q_T)^T\ge1-Tq_T>1/2$. Consequently
\eqref{eq:app-rare-expected} gives expected gap greater than $g_T/2$, and
\eqref{eq:fd-rare-failure} gives gap at least $g_T$ with probability greater
than $1/4$, each on a fixed legal instance. The initial gap is at most
$\Delta_0$ by construction. This proves the stochastic term of
\eqref{eq:fd-budget-lower}. When $\sigma=0$ that term is zero and no
rare-event construction is required.
\end{proof}

We turn from statistical indistinguishability to one-dimensional geometry. The following structural lemma explains why derivative signs suffice to localize a minimizer.

\begin{lemma}
\label[lemma]{lem:fd-1d-geometry}
Let $f:\mathbb R\to\mathbb R$ be globally $L$-smooth, bounded below, and
$\mu$-PL.  Then its minimizer set $\mathcal S=\arg\min f$ is a nonempty closed interval
(possibly a singleton, ray, or all of $\mathbb R$).  Moreover,
\begin{align}
 f'(x)&<0 &&(x<\inf \mathcal S),
 &f'(x)&>0 &&(x>\sup \mathcal S),
 \label{eq:fd-sign-separation}\\
 f(x)-f^\star&\ge\frac\mu2\operatorname{dist}(x,\mathcal S)^2,
 \label{eq:fd-1d-QG}\\
 |f'(x)|&\ge\mu\operatorname{dist}(x,\mathcal S),
 \label{eq:fd-1d-RSI}\\
 f(x)-f^\star&\le\frac L2\operatorname{dist}(x,\mathcal S)^2.
 \label{eq:fd-1d-smooth-upper}
\end{align}
\end{lemma}

\begin{proof}
By PL, every critical point is a global minimizer.  We first prove that a
minimizer exists.  Otherwise $f'$ never vanishes and, by continuity, has a
constant sign.  Suppose $f'>0$; the other case is symmetric.  Put
$h=f-f^\star>0$.  PL yields
$f'(x)\ge\sqrt{2\mu h(x)}$, hence
\[
 \frac{d}{dx}\sqrt{h(x)}
 =\frac{f'(x)}{2\sqrt{h(x)}}\ge\sqrt{\frac\mu2}.
\]
Integrating from $x<x_0$ to $x_0$ gives
$\sqrt{h(x_0)}-\sqrt{h(x)}
 \ge\sqrt{\mu/2}(x_0-x)$, impossible as $x\to-\infty$.

The zero set of $f'$ therefore equals the nonempty minimizer set $\mathcal S$.  If
$a,b\in \mathcal S$ and some $c\in(a,b)$ had $f(c)>f^\star$, then the maximum of $f$
on $[a,b]$ would be attained at an interior nonoptimal critical point, a
contradiction.  Thus $\mathcal S$ is an interval.  On each component of
$\mathbb R\setminus \mathcal S$, $f'$ has constant sign.  To the right it cannot be
negative, because then $f$ would fall below $f^\star$; to the left it cannot
be positive.  This proves \eqref{eq:fd-sign-separation}.

For $x>\sup \mathcal S=:s$, the same differential inequality now gives
\[
 \sqrt{f(x)-f^\star}\ge\sqrt{\frac\mu2}(x-s).
\]
The argument to the left of $\mathcal S$ is symmetric. This proves
\eqref{eq:fd-1d-QG}.  Combining
it with PL gives \eqref{eq:fd-1d-RSI}.  Finally, for
$s=\Pi_\mathcal S(x)$, smoothness and $f'(s)=0$ imply
$f(x)-f^\star\le L|x-s|^2/2$, proving
\eqref{eq:fd-1d-smooth-upper}.
\end{proof}

The geometry lemma supplies both an initial bracket and a stopping certificate. We use it to analyze Algorithm~\ref{alg:robust-stochastic-bisection}.

From \eqref{eq:fd-1d-QG} and the initial gap,
\begin{equation}\label{eq:fd-D0}
 \operatorname{dist}(x_0,\mathcal S)
 \le D_0:=\sqrt{\frac{2\Delta_0}{\mu}}.
\end{equation}
Start with $I_0=[x_0-D_0,x_0+D_0]$ and set
\begin{equation}\label{eq:fd-1d-r-eps}
 r_\epsilon:=\frac14\sqrt{2\mu\epsilon}.
\end{equation}
At the midpoint $c_k$ of the current interval, invoke the centered-clipping
estimator of Lemma~\ref{lem:centered-clipping} with accuracy $r_\epsilon$ and
failure probability $\delta/K_\epsilon$.  If
$|\widehat g_k|\le2r_\epsilon$, return $c_k$.  If
$\widehat g_k>2r_\epsilon$, retain the left half; if
$\widehat g_k<-2r_\epsilon$, retain the right half.  If no early termination
occurs, return the midpoint after $K_\epsilon$ halvings.

\begin{proof}[Proof of Theorem~\ref{thm:fd-upper-main}\textup{(i)}]
By Lemma~\ref{lem:centered-clipping} and a union bound, with probability at
least $1-\delta$ all estimates satisfy
$|\widehat g_k-f'(c_k)|\le r_\epsilon$.  Work on this event.  If the method
stops early, then $|f'(c_k)|\le3r_\epsilon$, so PL gives
\[
 f(c_k)-f^\star
 \le\frac{9r_\epsilon^2}{2\mu}
 =\frac9{16}\epsilon<\epsilon.
\]
If $\widehat g_k>2r_\epsilon$, then $f'(c_k)>0$, and
Lemma~\ref{lem:fd-1d-geometry} places every minimizer to the left of $c_k$.
The negative case is symmetric.  Hence every retained interval meets $\mathcal S$.
After $K_\epsilon$ halvings, the final midpoint $c$ satisfies
\[
 \operatorname{dist}(c,\mathcal S)\le D_0 2^{-K_\epsilon},
\]
and \eqref{eq:fd-1d-smooth-upper} yields
\[
 f(c)-f^\star
 \le\frac L2D_0^2 4^{-K_\epsilon}
 =\kappa\Delta_0 4^{-K_\epsilon}\le\epsilon.
\]
One estimator call costs, by Lemma~\ref{lem:centered-clipping},
\[
 C_\alpha\left[1+
 \left(\frac{\sigma}{r_\epsilon}\right)^{p_\alpha}\right]
 \log\frac{4K_\epsilon}{\delta}
 \le C_\alpha[1+Q_\epsilon]
 \log\frac{4K_\epsilon}{\delta}.
\]
Multiplication by $K_\epsilon$ proves
\eqref{eq:fd-d1-upper-explicit}.
\end{proof}

To complete the one-dimensional characterization, we still need the
logarithmic noiseless transient. The rare-event family alone cannot provide
it. We now give a fixed one-dimensional distribution of
strongly convex functions for which each exact-gradient query reveals only a
constant expected number of nested location bits.

For an interval $I=[a,a+s]$, define
\begin{equation}\label{eq:fd-children}
 I_0:=\left[a+\frac s8,a+\frac{3s}{8}\right],
 \qquad
 I_1:=\left[a+\frac{5s}{8},a+\frac{7s}{8}\right],
 \qquad
 A(s):=\frac67\mu s.
\end{equation}
Fix a root interval $I_\varnothing=[-\ell/2,\ell/2]$.  Given a word
$\omega\in\{0,1\}^K$, define a continuous piecewise-affine function
$h_\omega$ recursively.  On a node interval $I=[a,a+s]$, its boundary
values are $-A(s)$ and $A(s)$.  If the next bit is $b$, use the recursively
defined profile on $I_b$; on each component of $I\setminus I_b$, interpolate
affinely between the adjacent parent and child boundary values.  At depth
$K$, interpolate affinely across the final interval.  Outside the root
interval, extend with slope $\mu$.

Every affine slope can be computed explicitly.  On a terminal interval of
length $s$, the change from $-A(s)$ to $A(s)$ is
$2A(s)=12\mu s/7$, hence the slope is $12\mu/7$.  If the left child is
selected, the left connector has length $s/8$ and changes by
$A(s)-A(s/4)=9\mu s/14$, hence its slope is $36\mu/7$; the right connector
has length $5s/8$ and the same change, hence slope $36\mu/35$.  The two
values are interchanged when the right child is selected.  Outside the root
the prescribed slope is $\mu$.  Therefore every affine slope belongs to
\begin{equation}\label{eq:fd-tree-slopes}
 \left\{\mu,\frac{36}{35}\mu,
 \frac{12}{7}\mu,\frac{36}{7}\mu\right\}.
\end{equation}
Let $x_\omega^\star$ be the midpoint of the final interval and define
\begin{equation}\label{eq:fd-tree-function}
 f_\omega(x):=\int_{x_\omega^\star}^{x}h_\omega(t)\,dt.
\end{equation}
Then $f_\omega$ is $\mu$-strongly convex and $(36/7)\mu$-smooth, hence
$\mu$-PL.

\begin{lemma}
\label[lemma]{lem:fd-tree-info}
Let $\omega$ be uniform on $\{0,1\}^K$.  For every adaptive randomized
exact-gradient algorithm using $T$ queries, if $K=8T+2$, then with probability
at least $1/4$,
\begin{equation}\label{eq:fd-tree-gap}
 f_\omega(\widehat x_T)-f_\omega^\star
 \ge \frac{\mu\ell^2}{128}4^{-K},
\end{equation}
where $\ell$ is the length of the root interval.
\end{lemma}

\begin{proof}
Condition on the algorithm's internal random seed, independently of
$\omega$. This is the distributional argument underlying Yao's principle
\citep{Yao1977}. Keep $\omega$ uniform.  Give the algorithm a stronger oracle: in addition to
$h_\omega(x)$, a query reveals the next unknown bit and continues revealing
bits for as long as the query point lies in the selected child interval.  This
stronger response determines the exact derivative value, because once the
query lies outside the selected child, all deeper modifications are confined
to that child.

Conditional on the currently revealed prefix, the query is fixed and every
unrevealed bit is independent and uniform.  The first new bit is revealed
for free.  To reveal at least $j\ge2$ new bits, the query must lie in the
selected child at each of the preceding $j-1$ levels.  The two children are
disjoint, so each continuation has conditional probability at most $1/2$.
Thus, if $J_t$ is the number of new bits revealed at query $t$,
\[
 \mathbb P(J_t\ge j\mid\text{past})\le2^{-(j-1)},
 \qquad
 \mathbb E[J_t\mid\text{past}]\le2.
\]
If $R_T$ is the total revealed prefix length, then
$\mathbb ER_T\le2T$, and Markov's inequality gives
$\mathbb P(R_T\ge K/2)<1/2$.

On the event $R_T<K/2$, the next bit is still uniform.  If the current parent
interval has length $s=\ell4^{-R_T}$, then for every output $x$,
\[
 \max\{\operatorname{dist}(x,I_0),
       \operatorname{dist}(x,I_1)\}\ge\frac s8.
\]
Hence, with conditional probability at least $1/2$, the actual minimizer lies
at distance at least $s/8\ge\ell4^{-K/2}/8$ from the output.  Strong convexity
then gives \eqref{eq:fd-tree-gap}.  Multiplying the two conditional
probabilities gives a probability strictly larger than $1/4$.
The argument holds after conditioning on every internal random seed;
averaging over that seed proves the assertion for randomized algorithms.
\end{proof}

Choose
\begin{equation}\label{eq:fd-tree-root-length}
 \ell^2:=\frac{4\Delta_0}{3\mu}.
\end{equation}
Since every $f_\omega$ is $6\mu$-smooth and
$|x_\omega^\star|\le\ell/2$, its initial gap at $x_0=0$ is at most
$3\mu\ell^2/4=\Delta_0$.  Lemma~\ref{lem:fd-tree-info} therefore gives
\begin{equation}\label{eq:fd-1d-noiseless-budget}
 \inf_{\mathsf A}\sup_{f\in\mathcal F_1(L,\mu,\Delta_0)}
 \mathbb P\left(
 f(\widehat x_T)-f^\star\ge c\Delta_0e^{-CT}
 \right)\ge\frac14
\end{equation}
whenever $L/\mu\ge6$.  Equivalently,
\begin{equation}\label{eq:fd-1d-noiseless-accuracy}
 \mathsf T^{(1)}_{\epsilon,1/4}
 \ge c\log\frac{\Delta_0}{\epsilon}
\end{equation}
for $\epsilon\le c_0\Delta_0$.  Adding
$\frac\mu2\sum_{j=2}^dx_j^2$ embeds this family into every fixed dimension
without changing smoothness, PL, or oracle information.
The strict probability bound in the proof of Lemma~\ref{lem:fd-tree-info}
justifies \eqref{eq:fd-1d-noiseless-accuracy} even at confidence $1/4$:
choose the budget so that its gap threshold exceeds $\epsilon$.

\begin{proof}[Completion of Theorem~\ref{thm:app-fixed-lower} and proof of
Theorem~\ref{thm:fd-lower-main}]
With $K=8T+2$, \eqref{eq:fd-tree-gap} and
\eqref{eq:fd-tree-root-length} give a gap of at least
$\Delta_0\exp[-8(\log4)T]/1536$ with probability greater than $1/4$.
They therefore also give an expected gap greater than one quarter of that
quantity. Together with the stochastic construction above and
$\max\{u,v\}\ge(u+v)/2$, this proves \eqref{eq:fd-budget-lower}.
For the probability claim, choose whichever hard family has the larger
threshold; this chooses one instance whose threshold controls half the sum
and whose failure probability is strictly greater than $1/4$.
The same deterministic construction proves
\eqref{eq:app-fixed-det-confidence}. In the regime where
\eqref{eq:fd-lower-stochastic-confidence} applies, combine the two
complexity lower bounds by taking their maximum. In the complementary
regime $Q_\epsilon$ is bounded by a constant depending only on
$\alpha$, and the deterministic term absorbs it. This proves the last
claim of Theorem~\ref{thm:app-fixed-lower}.

For clarity, the expected-error statement of
Theorem~\ref{thm:fd-lower-main} follows directly, not by converting an
upper confidence guarantee into an expectation. The exact family forces
$T\ge c\log(\Delta_0/\epsilon)$ when $\epsilon/\Delta_0$ is a sufficiently
small universal constant. If $Q_\epsilon$ is above the fixed
threshold required for \eqref{eq:fd-q-epsilon}, use that choice of
$q_\epsilon$. For $T\le1/(4q_\epsilon)$,
\eqref{eq:app-rare-expected} and $(1-q_\epsilon)^T\ge1-Tq_\epsilon\ge3/4$
give expected gap at least $3\epsilon>\epsilon$.
Thus $T\ge c_\alpha Q_\epsilon$ in this regime. Below this threshold,
$Q_\epsilon$ is bounded and is absorbed by the deterministic
logarithm. Taking the maximum of these two necessary budgets proves
\eqref{eq:fd-lower-sum}, the expected-error bound in the main text.
\end{proof}

\subsection{Gradient-only potential reconstruction on cutting surfaces}\label{subsec:fd-reconstruction}
The algorithms for $d\ge2$ use no function-value oracle.  Every required
relative value is reconstructed from the identity
\[
 f(b)-f(a)=\int_0^1
 \left\langle\nabla f(a+t(b-a)),b-a\right\rangle\,dt.
\]

The moment estimate below is a martingale-difference analogue of the
classical von Bahr--Esseen inequality in
\citet{vonBahrEsseen1965}. Sharp extensions of von Bahr--Esseen-type
bounds to martingales and more general moment functions were developed
by \citet{Pinelis2015VBE}. We include a short self-contained
proof because the precise conditional formulation needed for the
subsequent adaptive line-integral estimator is particularly simple.

\begin{lemma}
\label[lemma]{lem:fd-martingale-moment}
Let $1<\alpha\le2$ and let $(X_i,\mathscr F_i)$ be real-valued martingale
differences.  Then
\begin{equation}\label{eq:fd-martingale-vbe}
 \mathbb E\left|\sum_{i=1}^mX_i\right|^\alpha
 \le C_\alpha\sum_{i=1}^m\mathbb E|X_i|^\alpha.
\end{equation}
Moreover,
\begin{equation}\label{eq:fd-martingale-max}
 \mathbb E\max_{j\le m}
 \left|\sum_{i=1}^jX_i\right|^\alpha
 \le C_\alpha\sum_{i=1}^m\mathbb E|X_i|^\alpha.
\end{equation}
\end{lemma}

\begin{proof}
Let $r=\alpha-1\in(0,1]$ and $H(u)=\operatorname{sgn}(u)|u|^r$.
For equal signs, concavity gives
$|H(u)-H(v)|\le|u-v|^r$; for opposite signs,
$|u|^r+|v|^r\le2^{1-r}(|u|+|v|)^r$.
Thus $H$ is $r$-H\"older with constant $2^{2-\alpha}$.
Integrating the derivative of $|u+tv|^\alpha$ from $t=0$ to $1$ yields
\[
 |u+v|^\alpha\le |u|^\alpha
 +\alpha |u|^{\alpha-2}uv+2^{2-\alpha}|v|^\alpha,
\]
with the linear term zero at $u=0$. Applying this inequality to the
successive partial sums and taking conditional expectations eliminates the
linear term. Iteration proves \eqref{eq:fd-martingale-vbe} with constant
$2^{2-\alpha}$.

For completeness, write $M_j=\sum_{i\le j}X_i$ and
$M^*=\max_{j\le m}|M_j|$. The first crossing time of $\lambda>0$ and the
submartingale property of $|M_j|$ imply
$\lambda\mathbb P(M^*\ge\lambda)
\le\mathbb E[|M_m|\mathbf1_{\{M^*\ge\lambda\}}]$.
Integrating this inequality against $\alpha\lambda^{\alpha-2}$ gives
\[
 \mathbb E(M^*)^\alpha
 \le\frac{\alpha}{\alpha-1}
 \mathbb E\bigl[|M_m|(M^*)^{\alpha-1}\bigr].
\]
H\"older's inequality yields
$\mathbb E(M^*)^\alpha\le
(\alpha/(\alpha-1))^\alpha\mathbb E|M_m|^\alpha$.
All moments are finite if the right-hand side of
\eqref{eq:fd-martingale-vbe} is finite; otherwise the assertion is trivial.
This proves \eqref{eq:fd-martingale-max}, including a possible larger
constant. Every argument remains valid after conditioning on an initial
history. This is the usual maximal-inequality argument; see also
\citet{HallHeyde1980}.
\end{proof}

The next estimator has two error sources with different scalings.
Stratification controls the randomized quadrature error through a
second moment, while the oracle error uses only the martingale
$\alpha$-moment estimate just proved.

\begin{lemma}
\label[lemma]{lem:fd-line-integral}
Let $a,b\in\mathbb R^d$ be measurable before this request,
$h=\|b-a\|$, and let $u,v>0$ and
$\rho\in(0,1/2)$.  There is a gradient-only estimator
$\widehat I_{a,b}$ such that
\begin{equation}\label{eq:fd-line-integral-guarantee}
 \mathbb P\left(
 \left|\widehat I_{a,b}-(f(b)-f(a))\right|>u+v
 \ \middle|\ \text{past}\right)\le\rho,
\end{equation}
and the number of oracle calls is at most
\begin{equation}\label{eq:fd-line-integral-cost}
 C_\alpha\left[
 1+\left(\frac{Lh^2}{u}\right)^{2/3}
 +\left(\frac{\sigma h}{v}\right)^{p_\alpha}
 \right]\log\frac4\rho.
\end{equation}
\end{lemma}

\begin{proof}
Condition on the history preceding the request. If $a=b$, return zero
without a query. Otherwise parameterize the segment by
$\gamma(t)=a+t(b-a)$ and put
$\varphi(t)=\langle\nabla f(\gamma(t)),b-a\rangle$.  Then
$|\varphi(t)-\varphi(s)|\le Lh^2|t-s|$.

Partition $[0,1]$ into $m$ equal intervals.  In interval $j$, draw $U_j$
uniformly, independently of the past, query the oracle at $\gamma(U_j)$,
and form
\[
 \widehat I^{(1)}=\frac1m\sum_{j=1}^m
 \langle G(\gamma(U_j),\xi_j),b-a\rangle.
\]
Write its error as $Q+N$.  If
$\overline\varphi_j=m\int_{(j-1)/m}^{j/m}\varphi(t)\,dt$, then the quadrature
increments
$m^{-1}(\varphi(U_j)-\overline\varphi_j)$ are centered and bounded by
$Lh^2/m^2$.  Therefore
\begin{equation}\label{eq:fd-Q-second-moment}
 \mathbb EQ^2\le\frac{L^2h^4}{m^3}.
\end{equation}
For the oracle noise
$\zeta_j=G(\gamma(U_j),\xi_j)-\nabla f(\gamma(U_j))$, the increments
$m^{-1}\langle\zeta_j,b-a\rangle$ are martingale differences with conditional
$\alpha$-moment at most $\sigma^\alpha h^\alpha/m^\alpha$.  Hence
Lemma~\ref{lem:fd-martingale-moment} gives
\begin{equation}\label{eq:fd-N-alpha-moment}
 \mathbb E|N|^\alpha
 \le C_\alpha\frac{\sigma^\alpha h^\alpha}{m^{\alpha-1}}.
\end{equation}
Choose
\[
 m=\max\left\{1,\left\lceil C_{{\rm rep},\alpha}\left[
 \left(\frac{Lh^2}{u}\right)^{2/3}
 +\left(\frac{\sigma h}{v}\right)^{p_\alpha}\right]\right\rceil\right\},
\]
with $C_{{\rm rep},\alpha}$ from \eqref{eq:app-algorithm-constants}.
In \eqref{eq:fd-N-alpha-moment} one may use $2^{2-\alpha}\le2$.
Consequently Markov's inequality bounds
$\mathbb P(|Q|>u\mid\text{past})$ and
$\mathbb P(|N|>v\mid\text{past})$ by $1/16$ each.
One replicate is accurate to $u+v$ with conditional probability at least
$7/8$. Generate $B=\lceil4\log(4/\rho)\rceil$ replicates sequentially,
using fresh quadrature points and oracle samples, and return any median.
Independence between replicate errors is unnecessary: each bad-replicate
indicator has conditional mean at most $1/8$.
Lemma~\ref{lem:bad-block-concentration}, with $p=1/8$ and threshold $1/2$,
bounds the probability of at least $B/2$ bad replicates by
$\exp(-9B/32)\le\rho$. On the complementary event more than half of the
replicates are accurate, so their median is accurate. Counting $Bm$ oracle
calls proves \eqref{eq:fd-line-integral-cost} and justifies
Algorithm~\ref{alg:robust-stochastic-line-integral} under adaptive noise.
\end{proof}

A single line integral is not enough to reconstruct a cutting surface efficiently. We therefore share edge estimates along a nested dyadic tree.

Let $H$ be a $k$-dimensional axis-aligned hyperrectangle with all side lengths
between $r/3$ and $r$, and let $n=2^M$.  Its nested dyadic grids have at most
$C_k2^{k\ell}$ new points at level $\ell$, and each new point can be joined to
a parent at level $\ell-1$ by a segment of length at most
$C_kr2^{-\ell}$.  The resulting tree has $O_k(n^k)$ edges and a root-to-leaf
path uses at most one edge per level.  The finest grid is a
$C_kr/n$-net of $H$.

We use the following elementary allocation identity.

\begin{lemma}
\label[lemma]{lem:fd-allocation}
For $q>0$, weights $w_0,\dots,w_M>0$, and $\tau>0$,
\begin{equation}\label{eq:fd-allocation-identity}
 \inf_{x_\ell>0:\,\sum x_\ell\le\tau}
 \sum_{\ell=0}^Mw_\ell x_\ell^{-q}
 =\tau^{-q}\left(
 \sum_{\ell=0}^Mw_\ell^{1/(q+1)}
 \right)^{q+1}.
\end{equation}
\end{lemma}

\begin{proof}
The Lagrange equations give
$x_\ell=cw_\ell^{1/(q+1)}$; the active constraint determines
$c=\tau/(\sum w_j^{1/(q+1)})$.  Substitution proves the identity.  Convexity
of $x\mapsto x^{-q}$ makes this stationary point the global minimizer.
\end{proof}

The allocation identity now balances accuracy across levels. Applying it
separately to quadrature and oracle noise exposes the different dimension
thresholds of the two costs.

\begin{lemma}
\label[lemma]{lem:fd-hyperplane-profile}
Let $H$ be as above, let $z_0\in H$, and let
$\tau>0$, $\rho\in(0,1/2)$. All values
$f(z)-f(z_0)$ on a $C_kr/n$-net of $H$ can be estimated simultaneously to
absolute error at most $\tau$, with conditional failure probability at most
$\rho$, using
\begin{equation}\label{eq:fd-hyperplane-cost}
 C_{k,\alpha}\log\left(
 \frac{C_kn^k(M+1)}{\rho}
 \right)
 \left[
 n^k+
 \left(\frac{Lr^2}{\tau}\right)^{2/3}\mathcal A_k(n)
 +\left(\frac{\sigma r}{\tau}\right)^{p_\alpha}
 \mathcal C_{k,p_\alpha}(n)
 \right]
\end{equation}
oracle calls, where
\begin{equation}\label{eq:fd-Ak}
 \mathcal A_k(n):=
 \begin{cases}
  1,&k=1,\\
  n^{k-4/3},&k\ge2,
 \end{cases}
\end{equation}
and
\begin{equation}\label{eq:fd-Ckp}
 \mathcal C_{k,p}(n):=
 \begin{cases}
  1,&k<p,\\
  (1+\log n)^{p+1},&k=p,\\
  n^{k-p},&k>p.
 \end{cases}
\end{equation}
An external reference $z_0$ is also allowed provided
$\operatorname{dist}(z_0,H)\le\sqrt{k+1}r$. The necessary root edges then
have length $O_k(r)$ and do not change the order of the bound. Without a
bound on this distance, the root-edge cost must be counted separately.
\end{lemma}

\begin{proof}
Estimate every tree edge using Lemma~\ref{lem:fd-line-integral}, assign each
edge failure probability at most $\rho/(C_kn^k(M+1))$, and sum edge estimates
along root-to-leaf paths.  If every edge estimate succeeds and if the
level-$\ell$ deterministic and stochastic error allowances are
$u_\ell,v_\ell$, respectively, the uniform path error is at most
$\sum_\ell(u_\ell+v_\ell)$.  Allocate each sum at most $\tau/2$.

There are $O_k(2^{k\ell})$ level-$\ell$ edges of length
$O_k(r2^{-\ell})$.  The deterministic quadrature cost, before the common
logarithmic factor, is therefore
\[
 C_k(Lr^2)^{2/3}
 \sum_{\ell=0}^M
 2^{(k-4/3)\ell}u_\ell^{-2/3}.
\]
Apply Lemma~\ref{lem:fd-allocation} with $q=2/3$.  If $k=1$, the resulting
geometric sum is bounded; if $k\ge2$, it is dominated by its last term and
equals $O_k(n^{k-4/3})$ after taking the required power.  This yields the
second term in \eqref{eq:fd-hyperplane-cost}.

The stochastic cost is
\[
 C_{k,\alpha}(\sigma r)^{p_\alpha}
 \sum_{\ell=0}^M
 2^{(k-p_\alpha)\ell}v_\ell^{-p_\alpha}.
\]
Applying Lemma~\ref{lem:fd-allocation} with $q=p_\alpha$ gives
\[
 C_{k,\alpha}
 \left(\frac{\sigma r}{\tau}\right)^{p_\alpha}
 \left[
 \sum_{\ell=0}^M
 2^{(k-p_\alpha)\ell/(p_\alpha+1)}
 \right]^{p_\alpha+1},
\]
which is exactly \eqref{eq:fd-Ckp}.  Finally, every edge requires at least one
sample, producing the $n^k$ term.  The union bound over edges and the common
median-amplification factor give the logarithm in
\eqref{eq:fd-hyperplane-cost}.  A constant number of additional root edges is
absorbed into the same estimates.
\end{proof}

\subsection{Gradient-flow trapping and the upper bound for \texorpdfstring{$d\ge2$}{d at least 2}}\label{subsec:fd-trapping}
The gradient-flow viewpoint originates in \citet{bubeck2020trap}.
Here we need a version that tolerates approximate relative values and uses
no function-value oracle. We first prove that a boundary certificate traps a
small gradient. We then show how two cuts preserve that certificate with a
controlled increase in its slope, including a strict interpolation margin.
The preceding reconstruction lemma will supply all values used in these cuts.

\begin{definition}[Gradient-flow trap]\label[definition]{def:fd-trap}
Let $R\subset\mathbb R^d$ be a compact axis-aligned box with positive side
lengths, $x\in\operatorname{int}(R)$, and $c\ge0$. The pair $(R,x)$ is a
$c$-trap if
\begin{equation}\label{eq:fd-trap-def}
 f(y)>f(x)-c\|y-x\|\qquad(y\in\partial R).
\end{equation}
\end{definition}

\begin{lemma}[Trap lemma]\label[lemma]{lem:fd-trap}
Every $c$-trap contains a point $z$ with $\|\nabla f(z)\|\le c$.
\end{lemma}
\begin{proof}
Suppose instead that $\|\nabla f\|>c$ on $R$. Compactness gives
$m_R:=\min_R\|\nabla f\|>c\ge0$. The vector field
$-\nabla f/\|\nabla f\|$ is Lipschitz on a neighborhood of $R$, because the
gradient is Lipschitz and bounded away from zero there. Starting at $x$, its
unit-speed trajectory is therefore uniquely defined up to exit and can be
continued as long as it remains in $R$. While inside,
\[
 \frac{d}{dt}f(z(t))=-\|\nabla f(z(t))\|\le-m_R<-c.
\]
The trajectory cannot remain in $R$ for all time: otherwise its objective
value tends to $-\infty$, contrary to the minimum of the continuous function
$f$ on the compact box. At the finite first boundary time $\tau>0$, let
$y=z(\tau)\in\partial R$. Its length is $\tau$, so
\[
 f(y)-f(x)\le-m_R\tau<-c\tau\le-c\|y-x\|,
\]
contradicting \eqref{eq:fd-trap-def}.
\end{proof}

The trap lemma is existential, but it becomes algorithmic once the box is
small. The next refinement rule shrinks volume by a fixed factor using only
a finite set of approximate values; it also accounts explicitly for the
error between grid vertices.

\begin{lemma}
\label[lemma]{lem:fd-box-refinement}
Let $d\ge2$ and let $(R,x)$ be a $c$-trap whose aspect ratio is at most $3$.
Write $s>0$ for its longest side. There are two parallel
$(d-1)$-dimensional cross-sections $E,F$ with the following property.
Equip each cut with a Cartesian grid of coordinate mesh at most $h$, where
$0<h\le s$, and use its standard simplicial triangulation. Suppose
$\tau>0$ and every grid vertex has an estimate satisfying
\begin{equation}\label{eq:fd-cut-value-errors}
 |\widehat D(z)-(f(z)-f(x))|\le\tau.
\end{equation}
One can select $R_+\subset R$ and $x_+\in\operatorname{int}(R_+)$ such that
\begin{equation}\label{eq:fd-box-volume-aspect}
 \operatorname{vol}(R_+)\le\tfrac56\operatorname{vol}(R),
 \qquad \operatorname{aspect}(R_+)\le3,
\end{equation}
and $(R_+,x_+)$ is a $c_+$-trap with
\begin{equation}\label{eq:fd-box-c-update}
 c_+\le c+400d\left(\frac\tau s+\frac{Lh^2}s+
\frac{ch^2}{s^2}\right).
\end{equation}
Any larger nonnegative slope is also a valid certificate.
\end{lemma}
\begin{proof}
Permute and reflect coordinates so that
$R=[0,s]\times R'$ and $x_1\ge s/2$; every side of $R'$ lies in $[s/3,s]$.
Set $E=\{s/6\}\times R'$ and $F=\{s/3\}\times R'$.
Call a grid vertex certified if
$\widehat D(z)\le-c\|z-x\|-\tau$, and let $A$ contain all certified
vertices of both cuts. If $A=\varnothing$, keep
$R_+=[s/3,s]\times R'$ and $x_+=x$. Otherwise choose
$z_\star\in\arg\min_{z\in A}\widehat D(z)$ and put $x_+=z_\star$.
If $z_\star\in F$, keep $[s/6,s]\times R'$; if $z_\star\in E$, keep
$[0,s/3]\times R'$. This is the rule in
Algorithm~\ref{alg:gradient-flow-trapping}.

No certified vertex can lie on $\partial R$: the old trap inequality and
\eqref{eq:fd-cut-value-errors} give
$\widehat D(z)>-c\|z-x\|-\tau$ on that boundary. Hence the new pivot is
interior, including in the unchanged-pivot case. The new side length is
$2s/3$, $5s/6$, or $s/3$, and all other side lengths lie in $[s/3,s]$.
This proves \eqref{eq:fd-box-volume-aspect}.
If a new pivot is selected, certification implies
\begin{equation}\label{eq:fd-certified-descent}
 f(z_\star)-f(x)\le-c\|z_\star-x\|.
\end{equation}
For every retained point $y$ of the old boundary, the triangle inequality
therefore yields
\[
 f(y)-f(z_\star)>c(\|z_\star-x\|-\|y-x\|)
                 \ge-c\|y-z_\star\|.
\]
For an unchanged pivot the old-boundary certificate is immediate.

Let $H$ be the newly introduced boundary cut. Its distance from $x_+$ is
at least $s/6$. At a certified vertex $z\in H$, minimality of
$\widehat D(z_\star)$ gives $f(z)-f(x_+)\ge-2\tau$.
At a noncertified vertex,
$f(z)-f(x)>-c\|z-x\|-2\tau$; combining this with
\eqref{eq:fd-certified-descent} and the triangle inequality gives the same
uniform bound
\[
 f(z)-f(x_+)\ge-c\|z-x_+\|-2\tau.
\]
This bound also holds when $A=\varnothing$. Define
\begin{equation}\label{eq:fd-gamma-net}
 \gamma_{\rm net}=c+18\tau/s,
 \qquad \Phi(y)=f(y)+\gamma_{\rm net}\|y-x_+\|\quad(y\in H).
\end{equation}
Every grid vertex satisfies $\Phi(z)\ge f(x_+)+\tau$.
The distance term has tangential Hessian norm at most $6/s$, so
$\Phi|_H$ has Lipschitz gradient at most $L+6\gamma_{\rm net}/s$.

For $y$ in a simplex with vertices $z_i$ and barycentric weights
$\lambda_i$, smoothness applied at $y$ and
$\sum_i\lambda_i(z_i-y)=0$ imply
\[
 \Phi(y)\ge\sum_i\lambda_i\Phi(z_i)-E_h,
 \qquad E_h=\frac{d-1}{2}(L+6\gamma_{\rm net}/s)h^2.
\]
Set $c_+=\gamma_{\rm net}+6E_h/s$. Since $\|y-x_+\|\ge s/6$,
\[
 f(y)+c_+\|y-x_+\|\ge\Phi(y)+E_h\ge f(x_+)+\tau>f(x_+).
\]
Thus the full new face has the required strict trap inequality. Increasing
$c$ also preserves the old-boundary inequalities. Finally,
\begin{align*}
 c_+-c
 &=18\tau/s+3(d-1)Lh^2/s
           +18(d-1)ch^2/s^2+324(d-1)\tau h^2/s^3\\
 &\le400d\left(\tau/s+Lh^2/s+ch^2/s^2\right),
\end{align*}
where $h\le s$. This proves \eqref{eq:fd-box-c-update} with a fixed
constant and preserves strictness even when some vertex estimates are at
their error limits.
\end{proof}

We now allocate this slope increase across a finite sequence of refinements.
The construction uses the explicit constants in
\eqref{eq:app-algorithm-constants}; the resulting inequalities establish both
termination and the claimed query count on one simultaneous success event.

\begin{proof}[Proof of Theorem \ref{thm:fd-upper-main} (ii)]
Assume a current pivot $x$ satisfies
\begin{equation}\label{eq:fd-stage-gap}
 f(x)-f^\star\le\Delta,
\end{equation}
and put
\begin{equation}\label{eq:fd-stage-eta}
 \eta:=\sqrt{\mu\Delta}.
\end{equation}
If a new point has gradient norm at most $\eta$, PL gives gap at most
$\Delta/2$.

Set $c_0=\eta/8$ and let $R_0$ be the cube centered at $x$ whose distance from
$x$ to every face is $H=2\Delta/c_0$.  For every boundary point $y$,
$f(y)\ge f^\star\ge f(x)-\Delta>f(x)-c_0\|y-x\|$, so $(R_0,x)$ is a
$c_0$-trap.  Its longest side is
\begin{equation}\label{eq:fd-initial-side}
 s_0=2H=\frac{4\Delta}{c_0}=\frac{32\Delta}{\eta}.
\end{equation}

Run Lemma~\ref{lem:fd-box-refinement} repeatedly.  Because each refinement
shrinks volume by at least $5/6$ and keeps aspect ratio at most $3$, after
\begin{equation}\label{eq:fd-refinement-number}
 J=\mathfrak c_d\lceil\log(2+\kappa)\rceil
\end{equation}
steps the longest side is at most $\eta/(4L\sqrt d)$, with the choice \eqref{eq:app-algorithm-constants}: indeed, the initial-to-final ratio is $O_d(L\Delta/\eta^2)=O_d(\kappa)$.
More explicitly, $s_t\le3(5/6)^{t/d}s_0$, so it suffices that
$t\ge d\log(384\kappa\sqrt d)/\log(6/5)$.
For $d\ge2$ and $\kappa\ge1$, this is smaller than
$\mathfrak c_d\lceil\log(2+\kappa)\rceil$:
use $\log(384\kappa\sqrt d)\le
(8+d)\log(2+\kappa)$ and $1/\log(6/5)<6$.
If the stopping length is reached earlier, we stop and use its last box
and pivot as $R_J,x_J$; the same bound $c_J\le\eta/4$ applies. In the
cost sums below, the index $t$ runs only over refinements actually performed,
whose number is at most $J$.

Allocate
\begin{equation}\label{eq:fd-refinement-allocation}
 a:=\frac{\eta}{\mathfrak c_dJ},
 \qquad
 \tau_t:=\frac{as_t}{\mathfrak c_d},
 \qquad
 h_t^2:=\frac{as_t}{\mathfrak c_d L},
\end{equation}
where $s_t$ is the current longest side. We maintain
$c_t=\eta/8+ta$. For $t\le J$ this is at most
$\eta/8+\eta/\mathfrak c_d\le\eta/4$.
Before termination $Ls_t>\eta/(4\sqrt d)$, so
$c_t/(Ls_t)\le\sqrt d$ and
$h_t^2/s_t^2<4\sqrt d/(\mathfrak c_d^2J)<1$.
Using the explicit constant in Lemma~\ref{lem:fd-box-refinement}, the
required slope increase is at most
\[
 400d\left(2+\frac{c_t}{Ls_t}\right)\frac a{\mathfrak c_d}
 \le\frac{400d(2+\sqrt d)}{\mathfrak c_d}a\le a.
\]
Thus the algorithm's update $c_{t+1}=c_t+a$ preserves the trap by induction;
no assumption about a later slope is used to prove an earlier step.
Consequently
\begin{equation}\label{eq:fd-final-c}
 c_J\le\frac\eta4.
\end{equation}
By Lemma~\ref{lem:fd-trap}, the final box contains $z$ with
$\|\nabla f(z)\|\le\eta/4$.  If $x_J$ is the final pivot, then
\[
 \|\nabla f(x_J)\|
 \le\|\nabla f(z)\|+L\|x_J-z\|
 \le\frac\eta4+L\operatorname{diam}(R_J)
 \le\frac\eta2<\eta.
\]
Thus one stage halves the objective gap.

At refinement $t$, each of the two cross-sections has dimension
$k=d-1$, diameter $O_d(s_t)$, and target mesh $h_t$.  Put
\begin{equation}\label{eq:fd-nt}
 n_t:=2^{\left\lceil\log_2(\mathfrak c_ds_t/h_t)\right\rceil},
 \qquad
 \chi_t:=\frac{Ls_t}{a}.
\end{equation}
Then $n_t^2\asymp_d\chi_t$, while
$L s_t^2/\tau_t\asymp_d\chi_t$ and
$\sigma s_t/\tau_t\asymp_d\sigma/a$.  The root of each cut tree is connected
to the pivot by a constant number of segments of length $O_d(s_t)$.
Lemma~\ref{lem:fd-hyperplane-profile} therefore bounds one refinement by
\begin{equation}\label{eq:fd-one-refinement-cost}
 C_{d,\alpha}\Lambda_{d,\alpha}
 \left[
 \mathcal D_k(\chi_t)
 +\left(\frac{\sigma}{a}\right)^{p_\alpha}
 \mathcal C_{k,p_\alpha}(n_t)
 \right].
\end{equation}
For $k=1$, the exponent $2/3$ comes from stratified line integration; for
$k\ge2$, the $n_t^k$ tree size and deterministic quadrature have the same
order $\chi_t^{k/2}$.

The longest sides decrease geometrically up to dimension-dependent constants:
from the volume reduction and bounded aspect ratio,
\begin{equation}\label{eq:fd-side-decay}
 s_t\le C_d(5/6)^{t/d}s_0.
\end{equation}
Using $a=\eta/(\mathfrak c_dJ)$,
$s_0=32\Delta/\eta$, and $\eta^2=\mu\Delta$, we obtain
\begin{equation}\label{eq:fd-chi0}
 \chi_0\le C_d\kappa J.
\end{equation}
Consequently,
\begin{equation}\label{eq:fd-det-stage-sum}
 \sum_{\text{performed }t}\mathcal D_k(\chi_t)
 \le C_d\mathcal D_k(\kappa J).
\end{equation}

For the stochastic part, first note
\begin{equation}\label{eq:fd-sigma-over-a}
 \left(\frac{\sigma}{a}\right)^{p_\alpha}
 \le C_{d,\alpha}J^{p_\alpha}
 \left(\frac{\sigma^2}{\mu\Delta}\right)^{\thetaalpha}.
\end{equation}
If $k<p_\alpha$, the surface factor in
\eqref{eq:fd-Ckp} is uniformly bounded and summing over the $J$ refinements
adds one factor $J$.  If $k=p_\alpha$, it adds at most
$J[\log(2+\kappa J)]^{p_\alpha+1}$.  If $k>p_\alpha$, the geometric decay in
\eqref{eq:fd-side-decay} gives
\[
 \sum_t n_t^{k-p_\alpha}
 \le C_{d,\alpha}(\kappa J)^{(k-p_\alpha)/2}.
\]
Combining these estimates proves that a gap-$\Delta$ stage costs at most
\begin{equation}\label{eq:fd-stage-cost}
 C_{d,\alpha}\Lambda_{d,\alpha}
 \left[
 \mathcal D_k(\kappa J)
 +\mathcal H_{k,p_\alpha}(\kappa,J)
 \left(\frac{\sigma^2}{\mu\Delta}\right)^{\thetaalpha}
 \right].
\end{equation}

Every cutting-surface reference is the current pivot, lying inside the
parent box. Its distance from a cut is at most $s_t$, so the external-reference
condition in Lemma~\ref{lem:fd-hyperplane-profile} is satisfied.
Assign each of the at most $2JS$ cutting-surface requests conditional
failure probability $\delta/(2JS)$. Each reconstruction allocates this
budget among its tree edges; the line-integral guarantee is conditional on
the full history at each edge request. Iterated conditioning and a union
bound therefore cover all adaptively selected cuts, without assuming that
their errors are independent. The number of edges per cut is
$O_d((\kappa J)^{k/2})$ and its depth is $O_d(\log(2+\kappa J))$.
Since $\log(2+\kappa J)=O_d(J+1)$, the logarithm in
\eqref{eq:fd-Lambda}, enlarged by a dimension-dependent constant, bounds all
median-amplification factors.

Now set $\Delta_s=\Delta_0/2^s$ and run
$S=\lceil\log_2(\Delta_0/\epsilon)\rceil$ stages, assigning failure
probability $\delta/S$ to each.  The deterministic cost repeats $S$ times,
whereas
\begin{equation}\label{eq:fd-stage-noise-geometric}
 \sum_{s=0}^{S-1}
 \left(\frac{\sigma^2}{\mu\Delta_s}\right)^{\thetaalpha}
 \le C_\alpha Q_\epsilon.
\end{equation}
A union bound makes all line-integral and trap certificates simultaneous with
probability at least $1-\delta$.  Equations
\eqref{eq:fd-stage-cost}--\eqref{eq:fd-stage-noise-geometric} prove
Theorem~\ref{thm:fd-upper-main}\textup{(ii)}.
\end{proof}

\subsection{Assembly of the fixed-dimensional phase diagram}\label{subsec:fd-phase}
The remaining step is to compare the two lower-bound mechanisms with the
explicit upper bounds. This requires keeping the confidence criterion
separate from expected risk and retaining the condition-number restriction
of the exact-gradient lower family. The following corollary records the
precise scope of the main-text comparison, including its last regime.

\begin{corollary}
\label[corollary]{cor:app-fixed-phase}
Let $d\ge1$ be fixed, $\kappa\ge6$, $1<\alpha\le2$, $\sigma\ge0$,
$0<\epsilon\le c_0\Delta_0$, and $0<\delta\le1/4$, with $c_0>0$ small
enough. Then
\begin{equation}\label{eq:app-d1-phase}
 \mathsf T^{(1)}_{\epsilon,\delta}
 =\widetilde\Theta_\alpha\left[
      \log(\Delta_0/\epsilon)+Q_\epsilon\right],
\end{equation}
where logarithmic factors in $\kappa,\Delta_0/\epsilon,1/\delta$ are
suppressed. For $d\ge2$,
\begin{equation}\label{eq:app-general-phase}
 c_\alpha\left[\log(\Delta_0/\epsilon)+Q_\epsilon\right]
 \le \mathsf T^{(d)}_{\epsilon,1/4}
 \le\widetilde O_{d,\alpha}\left[
     K_d(\kappa)\log(\Delta_0/\epsilon)
     +\kappa^{(d-1-p_\alpha)_+/2}Q_\epsilon\right],
\end{equation}
where $(r)_+=\max\{r,0\}$,
$K_2(\kappa)=\kappa^{2/3}$, and
$K_d(\kappa)=\kappa^{(d-1)/2}$ for $d\ge3$.
In particular, for $d>3$,
\begin{equation}\label{eq:fd-matching-range}
 p_\alpha\ge d-1\quad\Longleftrightarrow\quad
 \alpha\le(d-1)/(d-2),
\end{equation}
and the upper and lower stochastic terms established here match up to
logarithmic factors in this range. No necessity claim is made
for the extra factor when $p_\alpha<d-1$.
\end{corollary}

\begin{proof}[Proof of Corollary~\ref{cor:app-fixed-phase} and Theorem~\ref{cor:fd-123}]
For $d=1$, combine Theorem~\ref{thm:fd-upper-main}(i) with the
constant-confidence consequence of Theorem~\ref{thm:app-fixed-lower}.
The latter also applies to every $\delta\le1/4$, because the minimax
complexity is nonincreasing in the allowed failure probability. The
bisection bound introduces only the logarithmic factors displayed in
\eqref{eq:fd-d1-upper-explicit}. This proves \eqref{eq:app-d1-phase}.

For $d\ge2$ put $k=d-1$. Both $J$ and $\Lambda_{d,\alpha}$ in
Theorem~\ref{thm:fd-upper-main}(ii) are logarithmic factors for fixed $d$.
If $k<p_\alpha$, $\mathcal H_{k,p_\alpha}$ contains only such factors. If
$k=p_\alpha$, the dyadic level sum adds
$[\log(2+\kappa J)]^{p_\alpha+1}$, still logarithmic. If $k>p_\alpha$,
its polynomial contribution is $\kappa^{(k-p_\alpha)/2}$.
Meanwhile $\mathcal D_1(\kappa J)$ contributes $\kappa^{2/3}$ and
$\mathcal D_k(\kappa J)$ contributes $\kappa^{k/2}$ for $k\ge2$.
These observations give the upper side of \eqref{eq:app-general-phase};
Theorem~\ref{thm:app-fixed-lower} gives its lower side for every $\sigma$.

For $d=2$, $k=1<p_\alpha$ throughout the allowed moment range.
For $d=3$, $k=2\le p_\alpha$, with equality only at $\alpha=2$.
These are the rates in \eqref{eq:fd-d2-summary} and
\eqref{eq:fd-d3-summary}. For $d>3$, multiplication by the positive
quantity $\alpha-1$ gives
$p_\alpha\ge d-1$ if and only if
$(d-2)\alpha\le d-1$, proving \eqref{eq:fd-matching-range}.
This also explains the rate in \eqref{eq:fd-d-summary} and supplies the
last row of Theorem~\ref{cor:fd-123}, namely \eqref{eq:fd-dgreater3-summary}.
\end{proof}

% \begin{thebibliography}{9}

% \bibitem{ArjevaniEtAl2023}
% Y.~Arjevani, Y.~Carmon, J.~C.~Duchi, D.~J.~Foster, N.~Srebro, and
% B.~Woodworth.
% \newblock Lower bounds for non-convex stochastic optimization.
% \newblock \emph{Mathematical Programming}, 199:165--214, 2023.
% \newblock Preliminary version: arXiv:1912.02365.

% \end{thebibliography}

\end{document}